\documentclass[10pt,reqno]{amsart}

\usepackage{amssymb,amsrefs,amscd,amsmath}
\usepackage{a4wide}
\usepackage{graphicx,nicefrac}
\usepackage{todonotes}
\usepackage[shortlabels]{enumitem}
\usepackage{mathrsfs}
\usepackage{overpic}
\usepackage{hyperref}
\usepackage{ourbib}
\usepackage[all,cmtip]{xy}
\usepackage{callouts}

\definecolor{gelb}{RGB}{255,170,0}

\setlist{leftmargin=8mm}
\hypersetup{
    colorlinks,
    linkcolor={red!50!black},
    citecolor={blue!50!black},
    urlcolor={blue!80!black}
}

\newcommand{\RR}{\mathscr{R}}

\newcommand{\X}{\mathfrak{X}}
\newcommand{\II}{\mathrm{II}} 
\newcommand{\tr}{\mathrm{tr}} 
\newcommand{\R}{\mathbb{R}}

\newcommand{\N}{\mathbb{N}}
 
\newcommand{\Reg}{(H)}

\newcommand{\eps}{\varepsilon}
\renewcommand{\phi}{\varphi}

\newcommand{\scal}{\mathrm{scal}}
\newcommand{\ric}{\mathrm{ric}}
 
\renewcommand{\div}{\mathrm{div}}
\newcommand{\<}{\left\langle}
\renewcommand{\>}{\right\rangle}

\newcommand{\myicon}{$\,\,\,\triangleright$}
\renewcommand{\bullet}{{\mathbin{\vcenter{\hbox{\scalebox{2}{\mbox{$\cdot$}}}}}}}

\newcommand{\SO}{\mathrm{SO}}
\renewcommand{\O}{\mathrm{O}}

\DeclareMathOperator{\Stab}{\mathrm{Stab}} 

\DeclareMathOperator{\id}{\mathrm{id}}
\DeclareMathOperator{\Diff}{\mathrm{Diff}}

\newcommand{\LL}{\mathscr{L}}
\newcommand{\OO}{\mathrm{O}}
\newcommand{\grad}{\mathrm{grad}}
\renewcommand{\div}{\mathrm{div}}
\newcommand{\QED}{\tag*{$\qed$}}

\newtheorem{maintheorem}{Theorem}

\newtheorem{theorem}{Theorem}[section]
\newtheorem{lemma}[theorem]{Lemma}
\newtheorem{proposition}[theorem]{Proposition}
\newtheorem{corollary}[theorem]{Corollary} 
\newtheorem{addendum}[theorem]{Addendum} 
\newtheorem{setting}[theorem]{Setting} 
\newtheorem{assumption}[theorem]{Assumption} 

\theoremstyle{definition}
\newtheorem{remark}[theorem]{Remark}
\newtheorem{notation}[theorem]{Notation}
 
\newtheorem{example}[theorem]{Example}

\numberwithin{equation}{section}

\begin{document}

\title[Contractibility of spaces of psc metrics with symmetry]{Contractibility of spaces of positive scalar curvature metrics with symmetry}
\author{Christian B\"ar}
\address{Universit\"at Potsdam, Institut f\"ur Mathematik, 14476 Potsdam, Germany}
\email{\href{mailto:christian.baer@uni-potsdam.de}{christian.baer@uni-potsdam.de}}
\urladdr{\url{https://www.math.uni-potsdam.de/baer}}
\author{Bernhard Hanke}
\address{Universit\"at Augsburg, Institut f\"ur Mathematik, 86135 Augsburg, Germany}
\email{\href{mailto:hanke@math.uni-augsburg.de}{hanke@math.uni-augsburg.de}}
\urladdr{\url{https://www.math.uni-augsburg.de/diff/hanke}}

\begin{abstract} We show the contractibility of spaces of invariant Riemannian metrics of positive scalar curvature on compact connected manifolds of dimension at least two, with and without boundary and equipped with compact Lie group actions.
On manifolds without boundary, we assume that the Lie group contains a normal $S^1$-subgroup with at least one fixed-point component of codimension two.
In this situation, the existence of invariant metrics of positive scalar curvature was known previously.

For the proof, we combine equivariant Morse theory with conformal deformations near unstable manifolds. 
On manifolds without boundary, we also use local flexibility properties of positive scalar curvature metrics and the smoothing of mean-convex singularities.
\end{abstract}



\maketitle

\tableofcontents

\section{Introduction}

Let $M$ be a compact connected smooth manifold, possibly with boundary. 
It is well known that in dimensions $\ge 5$, the space of all Riemannian metrics of positive scalar curvature on $M$ (and satisfying appropriate boundary conditions if $\partial M \neq \emptyset$) is either empty or can have a very rich topology, see \cites{HSS14, BERW, BaerHanke2, ERW} among others.

Under symmetry assumptions, the situation can change drastically: 
positive scalar curvature metrics always exist under certain assumptions on the group action, by the work of Lawson and Yau \cite{LY74}*{Main Theorem} and by Wiemeler \cite{MW16}*{Theorem~1.1}.
In the present work we show that under Wiemeler's assumptions the space of all such metrics is contractible.

The existence question of $S^1$-invariant metrics of positive scalar curvature has furthermore been studied by B\'erard-Bergery, Hanke, and Lott in \cites{BB, Lott99, Hanke08}, among others.
For example, by B\'erard-Bergery's Theorem~C in \cite{BB}, a closed free $S^1$-manifold $M$ admits an $S^1$-invariant metric of positive scalar curvature if and only if the orbit space $M/S^1$ admits a metric of positive scalar curvature.
This can be used to construct closed free $S^1$-manifolds admitting metrics of positive scalar curvature, but no $S^1$-invariant such metrics, see \cite{BB}*{Example 9.2}. 

In general, the codimension of any fixed-point component in an $S^1$-mani\-fold is even because $S^1$ acts on the normal bundle without trivial summands.
Examples of $S^1$-manifolds with fixed-point components of codimension $4k$, $k \geq 1$, and not admitting $S^1$-invariant metrics of positive scalar curvature can be constructed using Lott's Theorem~5 in \cite{Lott99}.

Let $\Gamma$ be a compact, possibly non-connected Lie group acting smoothly on $M$.
We denote by $\mathscr{R}^\Gamma_{>0}(M)$ the space of $\Gamma$-invariant Riemannian metrics of positive scalar curvature on $M$, equipped with the $C^{\infty}$-topology.

Wiemeler showed the following result in \cite{MW16}*{Theorem~1.1}: 
Suppose $\Gamma$ is a compact Lie group and let $M$ be a closed connected smooth $\Gamma$-manifold such that $\Gamma$ contains a central $S^1$-subgroup with at least one  fixed-point component of codimension $2$.
Then $\mathscr{R}^{\Gamma}_{>0}(M) \neq \emptyset$.
This result is independent of spin and fundamental group assumptions.
Wiemeler's proof can be adapted to the case when $\Gamma$ contains a normal (rather than central) $S^1$-subgroup with at least one  fixed-point component of codimension~$2$.

The main result of the present paper, Theorem~\ref{mainB}, shows that under this weaker assumption, $\mathscr{R}^{\Gamma}_{>0}(M)$ is in fact contractible.
In particular, our approach provides an alternative proof of Wiemeler's existence result. 

Complete descriptions of the homotopy types of spaces with positive scalar curvature metrics on closed manifolds have so far remained elusive in dimensions greater than three, regardless of symmetry assumptions. 
Such homotopy types are well understood in dimension two by the uniformization theorem, and in dimension three  due to the work of Marques \cite{Marques} and Bamler-Kleiner \cite{BamKlein}.

Before proving Theorem~\ref{mainB}, we show the following result for spaces of $\Gamma$-invariant metrics of positive scalar curvature on manifolds with boundary.
This result is of independent interest. 

\begin{maintheorem} \label{mainA} Let $\Gamma$ be a compact Lie group, let $M$ be a compact connected smooth $\Gamma$-manifold of dimension at least $2$ and with non-empty (not necessarily connected) boundary.
Then $\mathscr{R}^{\Gamma}_{>0}(M)$ is contractible.
\end{maintheorem}

For $\Gamma = \{1\}$, Theorem~\ref{mainA} is a consequence of Gromov's h-principle \cite{GromovPartial}.
An equivariant version of Gromov's h-principle by Bierstone \cite{Bierstone}*{Theorem~3.2} applies, if for every closed subgroup $H < \Gamma$ each connected component of the union of all $\Gamma$-orbits with isotropy group conjugate to $H$ has a non-empty intersection with the boundary of $M$.
This condition is not required in Theorem~\ref{mainA}.

Our proof of Theorem~\ref{mainA}, which is independent of \cites{Bierstone, GromovPartial}, is based on equivariant Morse theory \cites{Mayer,  Wasserman} and conformal deformations around unstable manifolds.
In contrast to general h-principle techniques, it makes the construction of the relevant contracting homotopies quite explicit.

\begin{maintheorem} \label{mainB} 
Let $\Gamma$ be a compact Lie group and let $M$ be a closed connected smooth $\Gamma$-manifold.
Suppose that $\Gamma$ contains a normal $S^1$-subgroup with at least one fixed-point component of codimension $2$.
Then $\mathscr{R}^{\Gamma}_{>0}(M)$ is contractible.
\end{maintheorem} 

For this we combine Theorem~\ref{mainA} with the local flexibility lemma and with deformation properties of boundary conditions for positive scalar curvature metrics, proved by the authors in \cite{BaerHanke1} and \cite{BaerHanke2}, respectively.
Appendix~\ref{sec:EqFlexLemma} gives a short proof of an equivariant local flexibility lemma for smooth submanifolds, which is sufficient for our purposes.

Let $T^n = (S^1)^n$ be the $n$-torus.
Recall that a {\em torus manifold} is a closed connected smooth $2n$-dimensional effective $T^n$-manifold with non-empty fixed-point set.
Smooth compact toric varieties are examples of torus manifolds.
The following corollary lists some situations where the assumptions of Theorem~\ref{mainB} are satisfied; compare the proofs of \cite{MW16}*{Cor.~2.6, Cor.~2.8 and Cor.~2.9}.

\begin{corollary} \begin{enumerate}[(a)]
   \item Let $M$ be a torus manifold of dimension $2n$.
            Then $\mathscr{R}^{T^n}_{>0}(M)$ is contractible.
  \item Let $M$ be a closed connected smooth effective $S^1$-manifold of dimension $4$ and with negative Euler characteristic. 
             Then $\RR^{S^1}_{>0}(M)$ is contractible.
   \item Let $M$ be a closed connected oriented smooth effective $S^1$-manifold of dimension $4$.
   Assume that the action is semi-free, i.e., there are only fixed points or free orbits, and that $M$ has non-vanishing signature.
   Then $\RR^{S^1}_{>0}(M)$ is contractible.
 \end{enumerate}
\end{corollary}

The paper is structured as follows.
In Section~\ref{sec.ProofMainA} we prove Theorem~\ref{mainA}.
We use equivariant Morse theory as an essential tool.
The rest of the paper is devoted to the proof of Theorem~\ref{mainB}.
In Section~\ref{sec.scalmeansub} we study Riemannian submersions with a focus on scalar curvature and mean curvature of the boundary.
In Section~\ref{sec.Killing} we investigate the scalar curvature of deformations of $S^1$-invariant metrics obtained by rescalings with different rates along and orthogonal to the $S^1$-orbits.
Section~\ref{sec.uniformization} is devoted to a uniformization result.
We show that a family of invariant metrics can be deformed by suitable diffeomorphisms such that the normal exponential maps of an invariant subset coincide in a tubular neighborhood for all members of the family.
Section~\ref{sec.smoothing} deals with the smoothing of metrics on manifolds obtained by gluing together two manifolds along a common boundary.
Under a condition on the mean curvatures of the boundaries, such smoothings can be performed while keeping the scalar curvature positive.
This is known for metrics without symmetry.
The novelty here is that we can also perform the smoothing in an equivariant way.
This requires an equivariant version of the flexibility lemma, which is proved in Appendix~\ref{sec:EqFlexLemma}.
All the pieces are put together in Section~\ref{sec.proofmainB} to prove Theorem~\ref{mainB}.
This requires another use of the equivariant flexibility lemma as well as a geometric discussion of the $1$-jet of a metric along a submanifold which we include in Appendix~\ref{sec.jet1}.

\medskip

\textit{Acknowledgments:} 
The authors thank Kai Cieliebak, Georg Frenck and Misha Gromov for helpful comments.
They gratefully acknowledge support by the SPP 2026 ``Geometry at Infinity,'' funded by the DFG, and by the IAS Princeton, the Sapienza Universit\`a di Roma, and the University of Augsburg and the University of Potsdam.
B.~H.\ was supported by the Hausdorff Research Institute for Mathematics.

\section{Proof of Theorem~\ref{mainA}}
\label{sec.ProofMainA}

This section is devoted to the proof of Theorem~\ref{mainA}.

\begin{remark} \label{weaklyetc} 
Let $\Gamma$ be a compact Lie group and let $M$ be a smooth $\Gamma$-manifold, possibly with boundary.
The space $\RR^{\Gamma}_{>0}(M)$ of $\Gamma$-invariant Riemannian metrics of positive scalar curvature is an open subset of the Fr\'{e}chet space of all $\Gamma$-invariant sections of the symmetric $(0,2)$-tensor bundle $\mathrm{Sym}^2 (M) \to M$, equipped with the weak $C^{\infty}$-topology.
Hence, $\RR^{\Gamma}_{>0}(M)$ is a metrizable manifold in the sense of Palais, see \cite{Palais}*{Corollary 1}.
It follows from the corollary to Theorem~15 in \cite{Palais} that $\RR^{\Gamma}_{>0}(M)$ is contractible if and only if it is path connected and has trivial homotopy groups in all degrees.
This is what we will prove in the sequel.
\end{remark}

Let $\Gamma$ be a compact Lie group and let $M$ be a compact connected smooth $\Gamma$-manifold of dimension at least $2$ and with non-empty boundary $\partial M$.
Let $m \geq 0$, let $D^m$ be the closed $m$-ball and let 
\[
       g \colon \partial D^m \to \mathscr{R}^{\Gamma}_{>0}({M})
\]
be a continuous map. 
This assumption is empty for $m = 0$.
To prove Theorem~\ref{mainA}, we need to show that there exists a continuous map  $G \colon D^{m} \to \RR^{\Gamma}_{>0} (M)$ such that $G|_{\partial D^m} = g$, compare Remark~\ref{weaklyetc}.

Let 
\[
    f \colon {M} \to \R
\]
be a $\Gamma$-invariant Morse function, see \cite{Mayer}*{Definition~1.1} and \cite{ Wasserman}*{\S4}. 
In particular, the set of critical points of $f$ consists of non-degenerate critical orbits.
As the space of $\Gamma$-invariant Morse functions is open and dense in the space of all $\Gamma$-invariant smooth functions $M \to \R$ in the $C^2$-topology, we may assume that $f \leq 1$, that $f^{-1}(\{1\}) = \partial {M}$ and that all critical orbits of $f$ are contained in the interior $\mathrm{int}({M})$ of $M$.
It follows from the equivariant Morse lemma \cite{Mayer}*{Satz~1.3} that $f$ has only finitely many critical orbits.
After a $C^2$-small deformation, we can and will assume that the values of $f$ on these critical orbits are pairwise different.

Given a Euclidean vector space $V$, we denote by $B(V) \subset V$ the open unit ball in $V$.
Let 
\[
     \mathcal{O}_1, \ldots, \mathcal{O}_k \subset \mathrm{int}({M})
\]
be the  critical orbits of $f$, ordered in such a way that $f(\mathcal{O}_i) < f(\mathcal{O}_j)$ for $i < j$.
By the equivariant Morse lemma \cite{Mayer}*{Satz~1.3}, the following holds:
for each $1 \leq i \leq k$, there exists a closed subgroup $H_i < \Gamma$, two orthogonal linear $H_i$-representations $W^u_i$ and $W^s_i$ and a $\Gamma$-equivariant smooth embedding of a {\em $\Gamma$-handle}
\[ 
     \varphi_i \colon   \Gamma \times_{H_i} (B(W^u_i) \times B(W^s_i)) \hookrightarrow \mathrm{int}( {M}) 
\]
onto an open neighborhood of $\mathcal{O}_i$ sending  $\Gamma / H_i = \Gamma \times_{H_i} 0$ onto $\mathcal{O}_i$ and such that for $(\gamma, u, v) \in \Gamma \times_{H_i} (B(W^u_i) \times B(W^s_i))$, we have
\[
      (f \circ \varphi_i) ( \gamma, u, v) = f(\mathcal{O}_i) - |u|^2 + |v|^2 . 
\]
Note that the isotropy group of each point in $\mathcal{O}_i$ is conjugate to $H_i$.
We set 
\[
      V_i := \mathrm{im}(\varphi_i) \subset \mathrm{int}(M) .
\] 
We can assume that the $V_i$ are pairwise disjoint for $1 \leq i \leq k$.
Note that
\begin{equation} \label{dimsum}
    \dim_\R W^u_i + \dim_\R  W^s_i + \dim \Gamma - \dim H_i =  \dim M.
\end{equation}
The dimension of $W^u_i$ is called the {\em index} and the dimension of $W^s_i$ is called the {\em coindex} of $f$ at $\mathcal{O}_i$.

Fix a bi-invariant Riemannian metric on $\Gamma$.
Each $\Gamma \times_{H_i} (B(W^u_i) \times B(W^s_i))$ carries a canonical $\Gamma$-invariant Riemannian metric induced by the chosen metric on $\Gamma$ and the Euclidean metrics on $W^u_i$ and  $W^s_i$.
These metrics and the diffeomorphisms $\varphi_i$ induce $\Gamma$-invariant metrics on the neighborhoods $V_i$ of $\mathcal{O}_i$. 

Choose a $\Gamma$-invariant Riemannian metric $\mu$ on ${M}$ that restricts to these metrics on $V_i$.
Let $X:=-\grad(f) \in C^{\infty}(T{M})$ be the gradient field of $-f$ with respect to $\mu$ and let $\Psi^t \colon {M} \to {M}$ be the flow induced by $X$.
Then, $X$ and $\Psi^t$ are $\Gamma$-equivariant.
For each $x \in {M}$,  there is a maximal interval $\mathcal{I}_x \subset \R$ such that $\Psi^t(x)$ is defined for all $t \in \mathcal{I}_x$.
Since $f$ decreases along the integral curves of $X$, we have $[0, \infty) \subset \mathcal{I}_x$.

We consider the ``ancient'' subset
\[
    \mathcal{A} := \{ x \in M \mid  \mathcal{I}_x = \R \} \subset M.
\]

\begin{lemma} \label{closed}  $\mathcal{A}$ is a closed subset of $M$.
\end{lemma}

\begin{proof}
To each $x\in\mathcal{A}$ we associate the curve $\gamma_x\colon\R\to M$ given by $\gamma_x(t)=\Psi^t(x)$.
The set $\mathcal{F}:=\{\gamma_x \mid x\in\mathcal{A}\}$ is equicontinuous and hence, by the Arzel\`a-Ascoli theorem, relatively compact in the space of continuous curves in $M$.

Let $x_i\in\mathcal{A}$ converge to $x\in M$.
Then the curves $\gamma_{x_i}$ subconverge and hence $\Psi^t(x)$ is defined for all $t\in\R$.
\end{proof}

For $1 \leq i \leq k$, we consider the $\Gamma$-invariant subsets of $\mathrm{int}(M)$ defined by 
\begin{align*}
      \mathscr{W}^u_i & := \{ x \in \mathcal{A}  \mid  \lim_{t \to  -\infty} \Psi^{t}(x) \in \mathcal{O}_i \} , \\
       \mathscr{C}_{\leq i}  & := \bigcup_{1 \leq j \leq i} \mathscr{W}^u_{j} .
\end{align*}
We call $\mathscr{W}^u_i$  the  {\em unstable manifold} of the critical orbit $\mathcal{O}_i$.

\begin{lemma} \label{morse_flow} For $1 \leq i \leq k$, the following assertions hold: 
\begin{enumerate}[(i)]
    \item \label{one}  $\mathscr{W}^u_i$ is a smooth submanifold of $\mathrm{int}(M)$;
    \item \label{two}  $\mathscr{C}_{\leq i}$ is a closed subset of  $M$;
    \item \label{three} for every open neighborhood $U \subset {M}$ of $\mathscr{C}_{\leq k}$ there exists a $t_0 \geq 0$ such that for all $t \geq t_0$ we have $\Psi^t({M}) \subset U$.
\end{enumerate}
\end{lemma}
 
\begin{proof} 
We start by proving~\ref{one}.
Let $\partial_{r}^u$ be the unit radial vector field on $B(W^u_i) \setminus 0$ and let $\partial_{r}^s$ be the unit radial vector field on $B(W^s_i) \setminus 0$.
These vector fields pull back to vector fields on $\Gamma \times_{H_i} (B(W^u_i)\setminus 0  \times B(W^s_i) \setminus 0)$ which we denote by the same symbols.
The restriction of  $X$  to $V_i$ takes the form 
\begin{equation} \label{grad_field}
    (\gamma, u, v) \mapsto + 2 |u| \cdot \partial_{r}^u - 2  |v| \cdot \partial_{r}^s .
\end{equation}
For $1 \leq i \leq k$, define the $\Gamma$-invariant submanifold
\begin{equation} \label{defF}
     F^u_i := \varphi_i \big( \Gamma \times_{H_i} (B(W^u_i) \times 0) \big) \subset {M}. 
 \end{equation}
We claim that
\begin{equation} \label{fund_eq}
   \mathscr{W}^u_i = \bigcup_{n \in \N_0} \Psi^n ( F^u_i) .
\end{equation}
Indeed, on the one hand, it follows from \eqref{grad_field} that $\Psi^n(F^u_i) \subset \mathscr{W}^u_i$ for all $n$.
On the other hand, if $x \in \mathscr{W}^u_i$, then $\lim_{t \to -\infty} \Psi^{t} (x) \in \mathcal{O}_i$ and hence there exists $n_0 \in \N_0$ such that $\Psi^{-n}(x) \in V_i$ for all $n \geq n_0$.
Using  \eqref{grad_field}, this implies $\Psi^{-n_0}(x) \in F^u_i$. 
This shows  $x \in \Psi^{n_0}(F^u_i)$ and  finishes the proof of \eqref{fund_eq}.

We turn to the proof of  \ref{one}.
First, using \eqref{grad_field}, we observe that 
\begin{equation}\label{WVF}
\mathscr{W}^u_i \cap V_i = F^u_i , 
\end{equation}
which  is a submanifold of $\mathrm{int}(M)$. 
It remains to construct submanifold charts of $\mathscr{W}^u_i$ around each $x \in  \mathscr{W}^u_i \setminus \mathcal{O}_i$.
By \eqref{fund_eq} there exists  $n_0 \in \N$ with $x = \Psi^{n_0} (x_0)$ where $x_0 \in F^u_i \setminus \mathcal{O}_i$.

Since  $\Psi^{n_0} (F^u_i)$  is a smooth submanifold of $\mathrm{int}(M)$, there exists a submanifold chart defined on $W \subset \mathrm{int}(M)$ for $\Psi^{n_0} (F^u_i)$ with $x \in W$.
We claim that choosing $W$ small enough, this yields a submanifold chart for $\mathscr{W}^u_i$. 

Assume this is not the case. 
By \eqref{fund_eq} and since $\Psi^{n-1}(F^u_i) \subset \Psi^{n}(F^u_i)$ for all $n \in \N$, this implies that there exists a sequence $x_n \in \Psi^n(F^u_i) \setminus \Psi^{n-1}(F^u_i) \subset \mathscr{W}_i^u$, $n \in \N$, which subconverges to $x$ in ${M}$.
As $\Psi^{-n_0} (x) = x_0 \notin \partial {M}$, we can ensure, choosing $W$ small enough, that $\Psi^{-n_0}$ is defined on $W$.
Since  $\Psi^{-n_0} \colon W \to {M}$ is continuous, we can furthermore assume that $\Psi^{-n_0}(W) \subset V_i$.
Choose $n \geq n_0 +1 $ large enough such that $x_n \in W$.
Then, $\Psi^{-n_0} (x_n) \in V_i$.
Using $x_n \in \mathscr{W}^u_i$ and hence $\Psi^{-n_0} (x_n) \in \mathscr{W}^u_i$, \eqref{WVF} yields $\Psi^{-n_0}(x_n) \in F^u_i$.
Therefore, $x_{n} \in \Psi^{n_0}(F^u_i)$.
Since $\Psi^{n_0}(F^u_i) \subset \Psi^{n-1}(F^u_i)$ by the choice of $n$, this contradicts the choice of $x_{n}$.
The  proof of \ref{one} is hence complete.

To prove \ref{two}, it is enough to show that
\[
      \mathscr{C}_{\leq i} = \{x \in \mathcal{A}  \mid  f(\Psi^{-t}(x)) \leq f(\mathcal{O}_i) \text{ for all } t \in [0,\infty) \}. 
\]
Indeed, the right-hand side is a closed subset of $M$ as each $f \circ \Psi^{-t} \colon \mathcal{A} \to \R$ is continuous and $\mathcal{A} \subset M$ is closed by Lemma~\ref{closed}.
 
Clearly, the left-hand side is contained in the right-hand side.
Conversely, let $x \in \mathcal{A}$ satisfy   $f(\Psi^{-t}(x)) \leq f(\mathcal{O}_i)$ for all $t \geq 0$.
By compactness of $M$,  the sequence $\Psi^{-n}(x)$, $n \in \N$,  subconverges to some point $z \in M$ satisfying $f(z) \leq f(\mathcal{O}_i)$.
The point $z$ must be critical for $f$, say $z \in \mathcal{O}_j$ for some $1 \leq j \leq i$.
Hence, for every $n_0 \in \N$,  there exists $n \geq n_0$ with $\Psi^{-n}(x) \in V_j$. 
Using \eqref{grad_field}, this shows $x \in \mathscr{W}^u_j \subset  \mathscr{C}_{\leq i}$.

To prove \ref{three}, by the compactness of ${M}$, it is enough to show that 
\[
   \bigcap_{t \geq 0} \Psi^t({M}) = \mathscr{C}_{\leq k} . 
\]
By \eqref{fund_eq} and since $\Psi^{t}(F^u_i) \subset \Psi^s(F^u_i) \subset \Psi^t(M)$ for all $s\ge t\ge 0$ and all $i\in\{1,\dots,k\}$, the right-hand side is contained in the left-hand side.
Conversely, if $x \in  \bigcap_{t \geq 0} \Psi^t({M})$, then $\mathcal{I}_x = \R$, and, by the compactness of $M$, the sequence $\Psi^{-n}(x)$, $n \in \N$,  subconverges to a critical point $z \in M$.
By a similar argument as in the proof of \ref{two}, this implies $x \in \mathscr{C}_{\leq k}$. 
\end{proof}

\begin{remark} 
In Lemma~\ref{morse_flow}, 
\begin{enumerate}[\myicon]
   \item part~\ref{one} may be well known, but we could not find a suitable reference;
   \item part~\ref{two} is independent of the Morse-Smale condition;
   \item for trivial $\Gamma$, part~\ref{three} can be derived from \cite{CE}*{Lemma~9.18}.
\end{enumerate}
\end{remark}

\begin{lemma}[Extending the extension]\label{strati}  
Let $\mathscr{C} \subset M$ be a closed  $\Gamma$-invariant subset and let $U \subset M$ be a $\Gamma$-invariant open neighborhood of $\mathscr{C}$. 
Let 
\[
    G_U  \colon D^m \to \mathscr{R}^{\Gamma}_{>0}(U) 
\]
be a continuous map with $ G_U(\xi)  = g(\xi)|_U$ for  $\xi \in \partial D^m$, where $\mathscr{R}^{\Gamma}_{>0}(U)$ is equipped with the weak $C^{\infty}$-topology.

Let $\mathscr{W} \subset \mathrm{int}(M)$ be a $\Gamma$-invariant smooth, not necessarily compact, embedded submanifold without boundary such that $\dim \mathscr{W} \leq \dim M - 1$.

Then there are $\Gamma$-invariant open neighborhoods $U' \subset U$ of $\mathscr{C}$ and $V \subset \mathrm{int}(M)$ of $\mathscr{W}$ and a continuous map 
\[
    G_{U' \cup V} \colon D^m \to \mathscr{R}^{\Gamma}_{>0}(U' \cup V) 
\]
satisfying $G_{U' \cup V}(\xi)  = g(\xi)|_{U' \cup V}$ for $\xi \in \partial D^m$. 
\end{lemma} 

\begin{proof} 
Since $\mathscr{C} \subset {M}$ is closed, hence compact, there exists an open $\Gamma$-invariant subset $U_1 \subset M$ with $\mathscr{C} \subset U_1\subset\bar{U}_1 \subset U$. 
Then $\bar{U}_1$ and $M \setminus U$ are disjoint $\Gamma$-invariant closed subsets of $M$.
Take a $\Gamma$-invariant smooth function $\chi_1 \colon M  \to [0,1]$ with 
\[
       \chi_1|_{U_1} \equiv 0 \quad\text{ and }\quad \chi_1|_{M \setminus U}  \equiv 1 .
\]

Let $\delta_1 \colon [0,1] \to [0,1]$ be a continuous function which is equal to $1$ near $0$ and equal to $0$ near $1$.
Denote by $\mathscr{R}^{\Gamma}(M)$ the space of $\Gamma$-invariant Riemannian metrics on $M$ equipped with the $C^{\infty}$-topology.
Choose some $g_0 \in \mathscr{R}^{\Gamma}(M)$ and define a continuous map $\bar G \colon D^m \to \mathscr{R}^{\Gamma}(M)$ by
\[
     \bar G(\xi) = \begin{cases} (1-\chi_1) G_U(\xi)  + \chi_1\big(\delta_1(|\xi|) g_0 + (1-\delta_1(|\xi|)) g\left( \nicefrac{\xi}{|\xi|}\right) \big), & \xi \neq 0 , \\
                                                              (1-\chi_1) G_U(\xi) + \chi_1  g_0 , & \xi = 0.
                         \end{cases}                  
\]
We have  $\bar G(\xi) = g(\xi)$ for $\xi \in \partial D^m$ and $\bar G(\xi)|_{U_1} = G_{U}(\xi)|_{U_1}$ for $\xi \in D^m$.
By a conformal change of $\bar G(\xi)$ near $\mathscr{W}$, leaving $\bar G(\xi)$ unchanged near $\mathscr{C}$, we will turn the scalar curvature of $\bar G(\xi)$ positive near $\mathscr{W}$.

There exists an open $\Gamma$-invariant subset $U_2 \subset M$ with $\mathscr{C} \subset U_2\subset\bar{U}_2 \subset U_1$. 
Then $\bar{U}_2$ and $M \setminus U_1$ are disjoint $\Gamma$-invariant closed subsets of $M$.
Take a $\Gamma$-invariant smooth function $\chi_2 \colon M  \to [0,1]$ with 
\[
       \chi_2|_{U_2} \equiv 0 \quad\text{ and }\quad \chi_2|_{M \setminus U_1}  \equiv 1 .
\]

Since being of positive scalar curvature is an open condition, there exists $0 < r_0 < 1$ such that $\bar G(\xi) \in \RR^{\Gamma}_{>0}(M)$ for $r_0 \leq |\xi| \leq 1$. 
Let $0 < r_0 < r_1 < 1$ and let $\delta_2 \colon [0,1] \to [0,1]$ be a continuous function equal to $1$ on $[0, r_0]$ and equal to $0$ on $[r_1, 1]$.

For $\xi \in D^m$ and $x \in M$, put $r_\xi(x) =  \mathrm{dist}_{\bar G(\xi)}(x, \mathscr{W})$.
By Lemma~\ref{lem.tub.nbhd}, there exists a $\Gamma$-invariant open neighborhood $V \subset \mathrm{int}(M)$ of $\mathscr{W}$ such that, for each $\xi \in D^m$, the function $r_{\xi}^2$ is smooth on $V$.

Choose a constant $\Lambda > 0$ so large that 
\[
\scal_{\bar G(\xi)}(x) + 4\Lambda (n-1)(n-\dim\mathscr{W})>0
\] 
for all $\xi\in D^m$ and all $x\in M$. 
Here $n=\dim M \geq 2$.
For $\xi \in D^m$, define the $\Gamma$-equivariant smooth map $\psi^{\Lambda}_{\xi} \colon V \to \R$ by 
\begin{equation}\label{largeLambda}
     \psi^{\Lambda}_{\xi}(x)  :=   - \Lambda \cdot  \chi_2 (x) \cdot \delta_2( |\xi|) \cdot r^2_{\xi}(x). 
\end{equation}
The family $(\psi_\xi^\Lambda)$ depends continuously on $\xi$ in the weak $C^{\infty}$-topology.

Denote the (non-negative) Laplace operator with respect to $\bar G(\xi)$ by $\Delta^{\xi}$.
The function $\psi^{\Lambda}_\xi$ has the following properties:
\begin{enumerate}[(i)] 
  \item \label{equivone} $\psi^{\Lambda}_{\xi}|_{\mathscr{W}} = 0$;
  \item \label{equivzero} $d \psi^{\Lambda}_{\xi}|_{\mathscr{W}} = 0$; 
  \item \label{unchanged} $\psi^{\Lambda}_{\xi} (x) = 0$, if $x\in U_2$ or $|\xi| \geq r_1$;
  \item \label{lessequal} $\Delta^{\xi} \psi^{\Lambda}_{\xi}  \geq 0$ along $\mathscr{W}$;
  \item \label{verynegative} $\Delta^{\xi} \psi^{\Lambda}_\xi(x) = 2\Lambda (n - \dim \mathscr{W})$, if $|\xi| \leq r_0$ and $x \in \mathscr{W} \cap (M \setminus U_1)$. 
 \end{enumerate} 
The last two statements follow from the fact that along $\mathscr{W}$ we have $\Delta^\xi\psi^\Lambda_\xi = -\Lambda\cdot \chi_2(x)\cdot \delta_2(|\xi|)\cdot \Delta^\xi(r^2_\xi)=2\Lambda\cdot \chi_2(x)\cdot \delta_2(|\xi|)\cdot (n - \dim \mathscr{W})$ because $r^2_\xi=0$ and $d(r^2_\xi)=0$ hold along $\mathscr{W}$.
 
We define the continuous map
\[
  G^\Lambda \colon D^m \to  \mathscr{R}^{\Gamma}(V),\quad G^{\Lambda}(\xi) := e^{2\psi^{\Lambda}_{\xi}} \cdot \bar G(\xi) .
\]
From \cite{Besse}*{Theorem~1.159~(f)} we recall the transformation formula for the scalar curvature under conformal change of metrics:
\[
\scal_{G^{\Lambda}(\xi)} = e^{-2\psi^{\Lambda}_{\xi}} \cdot \left( \scal_{\bar G(\xi)} + 2(n-1)\Delta^{\xi} \psi^{\Lambda}_{\xi} - (n-2)(n-1)\big|d\psi^\Lambda_\xi\big|^2\right) .
\]
Now we observe:
\begin{enumerate}[\myicon] 
\item For $\xi \in D^m$, since  $\bar{G}(\xi)|_{U_1} = G_U(\xi)|_{U_1} \in \mathscr{R}^{\Gamma}_{>0}(U_1)$, Properties~\ref{equivzero} and ~\ref{lessequal} imply that $G^\Lambda(\xi)$ has positive scalar curvature along $\mathscr{W} \cap U_1$.
\item For $r_0 \leq |\xi| \leq 1$, since  $\bar G(\xi) \in \mathscr{R}^{\Gamma}_{>0}(M)$,  Properties~\ref{equivzero} and  \ref{lessequal} imply that $G^\Lambda(\xi)$ has positive scalar curvature along $\mathscr{W}$.
\item For $0 \leq |\xi| \leq r_0$, Properties~\ref{equivzero} and  \ref{verynegative} together with \eqref{largeLambda} imply that $G^{\Lambda}(\xi)$ has positive scalar curvature along $\mathscr{W} \cap (M \setminus U_1)$.
\item Property~\ref{unchanged} implies that $G^\Lambda(\xi)|_{U_2} = \bar G(\xi)|_{U_2}$ for $\xi \in D^m$ and that $G^{\Lambda}(\xi)|_V = \bar G(\xi)|_V = g(\xi)|_V$ for $\xi \in \partial D^m$. 
\end{enumerate} 

Passing to a smaller $V$, we can therefore set $U' := U_2$ and define 
\[
G_{U' \cup V} \colon D^m \to \mathscr{R}_{>0}^{\Gamma} (U' \cup V)
\] 
by  
\[
  G_{U' \cup V}(\xi)(x) 
  := 
  \begin{cases} 
    G_U(\xi)(x), & x \in U' , \\ 
    G^{\Lambda}(\xi)(x), & x \in V . 
  \end{cases}  
\]
This concludes the proof of Lemma~\ref{strati}.
\end{proof} 

We will apply Lemma~\ref{strati} inductively to $\mathscr{W} := \mathscr{W}^u_i$, $1 \leq i \leq k$. 
The condition $\dim \mathscr{W}^u_i \leq \dim M-1$ will be ensured by Lemma~\ref{smalldim}, which is prepared by Lemmas~\ref{lem.checkZ.connected}, \ref{lem.compact.core} and \ref{lem.cancellation}. 

We work with the $\Gamma$-invariant Riemannian metric $\mu$ on $M$ introduced before Lemma~\ref{closed} and the associated gradient vector field $X=-\grad(f)\in C^\infty(TM)$.

Let $H < \Gamma$ be a closed subgroup such that $\Gamma/H$ is the maximal orbit type of $M$ and let $M_{\Reg} \subset M$ be the principal stratum.\footnote{Recall that $M_{(H)}$ is  the union of all $\Gamma$-orbits with isotropy group conjugate to $H$.}  
By \cite{Bredon}*{Theorem~IV.3.1}, $M_{\Reg}$ is open and dense in $M$.
Furthermore, the existence of $\Gamma$-invariant collar neighborhoods implies that $\Gamma/H$ is the maximal orbit type of $\partial M$ and that $M_{\Reg}$ is a manifold with boundary $\partial(M_{\Reg}) = (\partial M)_{\Reg}$.

Following \cite{Mayer}*{Definition~2.1}, we call the $\Gamma$-invariant Morse function $f$ {\em special} if for each critical orbit $\mathcal{O}_i$ of $f$, the isotropy group $H_i$ acts trivially on the unstable representation $W_i^u$.
In other words, the index of $f$ at $\mathcal{O}_i$ equals the index of $f|_{M_{(H_i)}}$ at $\mathcal{O}_i$.

\begin{lemma}[Connectivity of the principal quotient]\label{lem.checkZ.connected}
Assume that the $\Gamma$-invariant Morse function $f$ is special.
Let $c < 1$ be a regular value of $f$ such that the critical orbits of $f$ contained in
\[
 Z:=f^{-1}([c,1])
\]
are precisely the critical orbits contained in $M_{\Reg}$. 
Set
\[
 Z_{\Reg}:=Z\cap M_{\Reg},\qquad
 \partial_-Z:=f^{-1}(\{c\}),\qquad
 \partial_+Z:=f^{-1}(\{1\})=\partial M,
\]
and write $f_Z:=f|_Z$.
Then $\check Z:=Z_{\Reg}/\Gamma$ is path connected. 
Consequently, the inclusion $\partial_+\check Z\hookrightarrow \check Z$ is $0$-connected.
\end{lemma}

\begin{proof}
By the equivariant slice theorem, the orbit space $\check M := M_{\Reg}/\Gamma$ carries an induced structure of a smooth manifold.
This manifold, which is non-compact in general, has non-empty boundary $\partial \check M = (\partial M)_{\Reg}/\Gamma$.
Moreover, since $M$ is path connected, $\check M$ is path connected by \cite{Bredon}*{Theorem~IV.3.1}.

Consider the smooth manifold $\check Z := Z_{\Reg}/\Gamma$.
It is non-compact in general and has non-empty boundary $\partial \check Z = \partial_+ \check Z \cup \partial_- \check Z$, where
\[
   \partial \check Z = (\partial Z)_{\Reg}/\Gamma = \partial_+ \check Z \cup \partial_- \check Z,
   \quad \partial_{\pm}\check Z = (\partial_\pm Z)_{(H)}/\Gamma,
   \quad \partial_+ \check Z = \partial \check M.
\]

Let $c' < c$ be a regular value of $f$ such that $f^{-1}([c',c]) \subset M$ contains exactly one critical orbit of $f$.
This critical orbit is of type $\Gamma/H'$ for some closed subgroup $H' < \Gamma$ and, by the choice of $c$, is disjoint from $M_{\Reg}$.
The compact smooth $\Gamma$-invariant submanifold $Z':=f^{-1}([c',1]) \subset M$ is obtained from $Z$ by attaching a $\Gamma$-handle
\[
 \Gamma\times_{H'}\bigl(\bar B(W^u)\times \bar B(W^s)\bigr)
\]
along a $\Gamma$-equivariant smooth embedding
\[
 \Gamma\times_{H'}\bigl(\bar B(W^u)\times S(W^s)\bigr)\hookrightarrow \partial_-Z,
\]
where $H'$ acts trivially on $W^u$ because $f$ is special and $\bar B(-)$ denotes the closed unit ball.
Since $\Gamma/H'\times(\bar B(W^u)\times 0)=\Gamma\times_{H'}(\bar B(W^u)\times 0)$ does not contain orbits of type $\Gamma/H$, the inclusion
\[
 \left(\Gamma\times_{H'}\bigl(\bar B(W^u)\times S(W^s)\bigr)\right)_{(H)}
 \hookrightarrow
 \left(\Gamma\times_{H'}\bigl(\bar B(W^u)\times \bar B(W^s)\bigr)\right)_{(H)}
\]
is a strong $\Gamma$-deformation retract.
We conclude that $Z_{\Reg}\hookrightarrow Z'_{\Reg}$ is a $\Gamma$-homotopy equivalence.
By induction, the inclusion $Z_{\Reg}\hookrightarrow M_{\Reg}$ is a $\Gamma$-homotopy equivalence.
Thus $\check Z=Z_{\Reg}/\Gamma$ is path connected because $\check M=M_{\Reg}/\Gamma$ is path connected.
Since $\partial_+\check Z=\partial\check M\neq\emptyset$, the inclusion $\partial_+\check Z\hookrightarrow\check Z$ is $0$-connected.
\end{proof}

\begin{lemma}[A compact principal-orbit core]\label{lem.compact.core}
Under the assumptions and with the notation of Lemma~\ref{lem.checkZ.connected}, there exists a compact $\Gamma$-invariant smooth submanifold with corners $R\subset Z$ with the following properties:
\begin{enumerate}[(i)]
 \item all $\Gamma$-orbits in $R$ are of type $\Gamma/H$, and  all critical orbits of $f_Z$ are contained in the interior of $R$;
 \item the boundary decomposes into $\Gamma$-invariant smooth manifolds with boundary as $\partial R=\partial_+R\cup\partial_-R\cup\partial_\mathrm{vert}R$, where $\partial_{\pm}R\subset\partial_{\pm}Z$;
 \item \label{collar_in_R} the gradient field $X=-\grad(f)$ is tangential to $\partial_\mathrm{vert}R$.
\end{enumerate}
\end{lemma}

\begin{proof}
Let $Z_\mathrm{sing}:=Z\setminus Z_{(H)}$ be the union of singular orbits in $Z$.
This is a closed $\Gamma$-invariant subset of $Z$ which does not contain critical orbits of $f_Z$.
Hence the restriction of the gradient vector field $X$ of $-f$ to $Z_\mathrm{sing}$ does not have zeros.
There is therefore a $\Gamma$-invariant open neighborhood $U$ of $Z_\mathrm{sing}$ in $Z$ such that $X|_U$ does not have zeros.
Choose an open $\Gamma$-invariant subset $V\subset Z$ with $Z_\mathrm{sing}\subset V\subset\bar V\subset U$, take a $\Gamma$-invariant smooth cutoff function $\chi\colon Z\to[0,1]$ with $\chi|_V\equiv1$ and $\chi|_{Z\setminus U}\equiv0$, and define the smooth $\Gamma$-equivariant vector field $Y\in C^\infty(TZ)$ by
\begin{equation} \label{def_Y}
 Y(x):=
 \begin{cases}
  \chi(x)\dfrac{X(x)}{\langle X(x),X(x)\rangle},&x\in U,\\
  0,&x\in Z\setminus U.
 \end{cases}
\end{equation}
Let $\varphi^t\colon Z\to Z$ be the flow induced by $Y$.
Each $\varphi^t$ is $\Gamma$-equivariant and sends points in $Z_\mathrm{sing}$ to points in $Z_\mathrm{sing}$.
For $x\in\partial_+Z$, let $\mathcal I_x\subset[0,\infty)$ be the maximal interval such that $\varphi^t(x)$ is defined for all $t\in\mathcal I_x$.
By construction, $[0,1-c]\subset\mathcal I_x$ for all $x\in\partial_+Z$, and $\mathcal I_x=[0,1-c]$ for all $x\in(\partial_+Z)_\mathrm{sing} := \partial_+ Z \setminus (\partial_+ Z)_{(H)}= \partial_+ Z \cap Z_{\rm sing} $.

By continuous dependence of solutions of ODEs on initial values, there exists an open $\Gamma$-invariant neighborhood $W$ of $(\partial_+Z)_\mathrm{sing}$ in $\partial_+Z$ such that $\varphi^t(x)\in V$ for all $x \in W$ and $t\in[0,1-c]$.
In particular, $\varphi\colon W\times[0,1-c]\to Z$, $(x,t)\mapsto\varphi^t(x)$, is a $\Gamma$-equivariant smooth embedding and $\varphi(W\times\{1-c\})\subset\partial_-Z$.
Since, for $z \in Z_\mathrm{sing}$, we have $\varphi^{1-f(z)}(z) \in (\partial_+Z)_\mathrm{sing}$, we furthermore obtain that 
\begin{equation} \label{cover_all}
 Z_{\rm sing} = \varphi \big( (\partial_+Z)_\mathrm{sing} \times [0,1-c] \big) . 
 \end{equation}

Let $d\colon\partial_+Z\to\R_{\geq0}$ measure the distance to $(\partial_+Z)_\mathrm{sing}$ with respect to the Riemannian metric induced by $\mu$.
The function $d$ is $\Gamma$-invariant and $1$-Lipschitz.
There exists $\eps>0$ such that $d^{-1}([0,\eps])\subset W$.
Let $\delta\colon\partial_+Z\to\R_{\geq0}$ be a $\Gamma$-invariant smooth function satisfying $\|\delta-d\|_{C^0}<\frac\eps6$.
By Sard's theorem, there exists a regular value  $\xi\in(\frac\eps3,\frac{2\eps}3)$ of $\delta$.
Put
\begin{equation} \label{def:P} 
 P:=\delta^{-1}([0,\xi])\subset\partial_+Z. 
 \end{equation} 
Then $P$ is a $\Gamma$-invariant compact smooth submanifold of $\partial_+Z$ with smooth boundary such that $(\partial_+Z)_\mathrm{sing}\subset\mathrm{int}(P)$ and $P\subset W$.
In particular, $\partial P\subset(\partial_+Z)_{(H)}$.

Let $Q:=\varphi(P\times[0,1-c])\subset Z$.
Then $Q$ is a $\Gamma$-invariant compact smooth submanifold of $Z$ with corners which is diffeomorphic to $P\times[0,1-c]$.
Its boundary decomposes into compact smooth $\Gamma$-invariant manifolds with boundary as $\partial Q=\partial_+Q\cup\partial_-Q\cup\partial_\mathrm{vert}Q$, where
\[
 \partial_+Q:=\varphi(P\times\{0\}),\quad
 \partial_-Q:=\varphi(P\times\{1-c\}),\quad
 \partial_\mathrm{vert}Q:=\varphi(\partial P\times[0,1-c]).
\]
We have that $\partial_{\pm}Q\subset\partial_{\pm}Z$, $Q$ contains no critical orbits of $f_Z$, and $\partial_\mathrm{vert}Q\subset Z_{(H)}$.
Furthermore, $X$ is tangential to $\partial_\mathrm{vert}Q$ and, by \eqref{cover_all}, we have $Z_\mathrm{sing}\subset Q$.

Set $R:=Z\setminus(Q\setminus\partial_\mathrm{vert}Q)$.
Then $R$ is a $\Gamma$-invariant compact smooth submanifold of $Z$ with corners.
Its boundary decomposes into compact smooth $\Gamma$-invariant manifolds with boundary as $\partial R=\partial_+R\cup\partial_-R\cup\partial_\mathrm{vert}R$, where
\[
 \partial_+R=R\cap\partial_+Z,\qquad
 \partial_-R=R\cap\partial_-Z,\qquad
 \partial_\mathrm{vert}R=\partial_\mathrm{vert}Q.
\]
All $\Gamma$-orbits in $R$ are of type $\Gamma/H$,  all critical orbits of $f_Z$ are contained in the interior of $R$, and $X$ is tangential to $\partial_\mathrm{vert}R$.
\end{proof}

\begin{lemma}[Cancellation of the coindex-zero critical orbits]\label{lem.cancellation}
Under the assumptions and with the notation of Lemma~\ref{lem.checkZ.connected}, 
the function $f_Z$ can be replaced by a special $\Gamma$-invariant Morse function $\tilde f_Z\colon Z\to[c,1]$ with no critical orbits of coindex $0$ and without altering $f_Z$ near $\partial Z$.
\end{lemma}

\begin{proof}
The function $f_Z$ induces a Morse function $f_{\check Z}\colon\check Z\to[c,1]$ with finitely many critical points, which correspond to the critical orbits of $f_Z$.
In particular, the finitely many coindex-$0$ critical orbits of $f_Z$ correspond to finitely many coindex-$0$ critical points $p_1,\ldots,p_\ell\in\check Z$ of $f_{\check Z}$.
These points are precisely the local maxima of $f_{\check Z}$ in $\check Z\setminus\partial\check Z$.
Since, by Lemma~\ref{lem.checkZ.connected}, the inclusion $\partial_+\check Z\hookrightarrow\check Z$ is $0$-connected, there are smooth curves $\gamma_1,\ldots,\gamma_\ell\colon[0,1]\to\check Z$ such that $\gamma_i$ starts at $p_i$ and ends in $\partial_+\check Z$.

Let $R \subset Z$ be the submanifold constructed in Lemma~\ref{lem.compact.core}.
Since all $\Gamma$-orbits in $R$ are of type $\Gamma/H$, the quotient $\check R:=R/\Gamma$ has an induced structure of a smooth manifold with corners.
The boundary of $\check R$ decomposes into compact smooth manifolds with boundary as $\partial\check R=\partial_+\check R\cup\partial_-\check R\cup\partial_\mathrm{vert}\check R$, where
\[
 \partial_{\pm}\check R=(\partial_{\pm}R)/\Gamma,
 \qquad
 \partial_\mathrm{vert}\check R=(\partial_\mathrm{vert}R)/\Gamma.
\]
In other words, $\check R$ constitutes a bordism between the manifolds with boundary $\partial_- \check R$ and $\partial_+ \check R$.

The restriction of $f_Z$ to $R$ induces a Morse function $f_{\check R}\colon\check R\to[c,1]$ with no critical points on $\partial\check R$ and satisfying $f_{\check R}^{-1}(c) = \partial_- \check R$ and $f_{\check R}^{-1}(1) = \partial_+ \check R$.
The $\Gamma$-invariant Riemannian metric $\mu_R$ induced by $\mu$ descends to a Riemannian metric $\mu_{\check R}$ on $\check R$.
Furthermore, the gradient field $X_{\check R}$ of $-f_{\check R}$ with respect to $\mu_{\check R}$ is tangential to $\partial_\mathrm{vert}\check R$.

The critical points of coindex $0$ of $f_{\check R}$ are precisely $p_1,\ldots,p_\ell$.
Passing to a smaller $\eps > 0$ in the proof of Lemma~\ref{lem.compact.core}, we can assume that the curves $\gamma_i$ connecting $p_i$ with $\partial_+ \check R$ are contained in $\check R$.
Thus the inclusion $\partial_+\check R\hookrightarrow\check R$ is $0$-connected.

In this situation the coindex-$0$ critical points $p_1, \ldots, p_\ell$ can be cancelled against coindex-$1$ critical points of $f_{\check R}$. 
To achieve this, we will double the manifold $\check R$ along its vertical boundary $\partial_{\rm vert} \check R$ and use handle cancellation as described  in \cite{Milnor}.

We will use  the notation from the proof of Lemma~\ref{lem.compact.core}. 
Since $\partial_+Z$ is compact,  the set of regular values of $\delta$ is open and thus we find $ \xi < \xi'  < \tfrac{2\eps}{3}$ such that $[\xi, \xi']$ contains only regular values of $\delta$.
Then $\mathscr{U}  := \delta^{-1}([\xi, \xi')) \subset \partial_+ R$ is a  $\Gamma$-invariant collar neighborhood of the boundary $\partial (\partial_+ R) = \partial P$ in $\partial_+ R$ and $\delta$ induces a  $\Gamma$-equivariant collar structure $\mathscr{U} \approx \partial ( \partial_+ R) \times [\xi, \xi')$.

Set $\mathscr{V} := \varphi(\mathscr{U} \times [0,1-c]) \subset R$ where $\varphi^t$ is the flow of the vector field $Y$ defined in \eqref{def_Y}.
Note that $\varphi(\mathscr{U} \times \{1-c\})$ is a $\Gamma$-invariant collar neighborhood of $\partial(\partial_- R)$ in $\partial_- R$.
Altogether,  $\mathscr{V}$ is a $\Gamma$-invariant collar neighborhood of $\partial_{\rm vert}  R$ in  $R$ on which the Morse function $f_R$ takes the form $f_R = 1-t$.
Dividing out the $\Gamma$-action, we obtain a collar neighborhood $\check{\mathscr{V}} \approx \partial_{\rm vert} \check R \times [\xi, \xi')$ of $\partial_{\rm vert} \check R$ in $\check R$, inducing collar neighborhoods of $\partial(\partial_{\pm} \check R)$ in $\partial_{\pm} \check R$.
On $\check{\mathscr{V}}$, we have $f_{\check R} = 1-t$, which is independent of the collar coordinate.

These collar structures  induce smooth structures on the doubles 
\begin{align*} 
   \mathcal{D} \check R& := \check R \cup_{\partial_{\rm vert} \check R} \check R, \\
   \mathcal{D} \check{\mathscr{V}} &:=  \check{\mathscr{V}} \cup_{\partial_{\rm vert} \check R} \check{\mathscr{V}}, \\
   \mathcal{D} \partial_+ \check R & := \partial_+ \check R \cup_{\partial(\partial_+ \check R)} \partial_+ \check R, \\
   \mathcal{D} \partial_- \check R & := \partial_- \check R \cup_{\partial(\partial_- \check R)} \partial_- \check R.
   \end{align*}
Now, $(\mathcal{D} \check R; \mathcal{D} \partial_+ \check R, \mathcal{D} \partial_-\check R)$ is a smooth manifold triad in the sense of \cite{Milnor}*{Definition 1.3}.
Setting 
\[
   f_{\mathcal{D}\check R} := f_{\check R} \cup_{\partial_{\rm vert} \check R} f_{\check R} \colon \mathcal{D}\check R \to [c,1],
\]
we obtain a Morse function $-f_{\mathcal{D}\check R}$ on this manifold triad in the sense of \cite{Milnor}*{Definition 2.3}.
All critical points of $f_{\mathcal{D}\check R}$ lie outside of $\mathcal{D}\check{\mathscr{V}}$.
Put $\alpha := \xi' - \xi$.
We obtain an induced diffeomorphism  $\mathcal{D} \check{\mathscr{V}} \approx \partial_{\rm vert} \check R \times (-\alpha, \alpha)$, yielding a tubular neighborhood of $\partial_{\rm vert} \check R$ in $\mathcal{D} \check R$.

Let $r \in (-\alpha, \alpha)$ be the collar coordinate, let $h$ be some Riemannian metric on $\partial_{\rm vert} \check R$ and let  $\chi \colon \R \to [0,1]$ be a smooth cutoff function equal to $1$ on $[-\alpha/3, \alpha/3]$ and equal to $0$ on $(-\infty, - 2\alpha/3] \cup [2\alpha/3, \infty)$.
We obtain a smooth Riemannian metric on $\partial_{\rm vert} \check R \times (-\alpha, \alpha)$, defined by
\[
    \chi(r) \cdot (h + dr^2) + \big(1 - \chi(r)\big) \, ( \check \mu \cup  \check \mu).
\]
Extend this metric by $\check \mu \cup  \check \mu$ to a smooth Riemannian metric on $\mathcal{D} \check R$. 
By construction, the associated gradient field $X_{\mathcal{D}\check R}$ of $-f_{\mathcal{D}\check R}$ is tangential to $\partial_{\rm vert}\check R \times \{r\}$ for all $r \in (-\alpha/3, \alpha/3)$. 
In particular, all stable and unstable manifolds of critical points of $f_{\mathcal{D}\check R}$ have distance at least $\alpha/3$ to $\partial_{\rm vert}\check R$.

By construction, the inclusion $\mathcal{D} \partial_+ \check R \hookrightarrow \mathcal{D} \check R$ is $0$-connected. 
Applying \cite{Milnor}*{Theorem 8.1}, we cancel the critical points of index $0$ of  $- f_{\mathcal{D}\check R}$ against critical points of index $1$ of $-f_{\mathcal{D}\check R}$.
Note that the doubled function $- f_{\mathcal{D}\check R}$ has two copies of each $p_i$ as  critical points of index $0$.

We claim that we can assume that this cancellation process leaves  $f_{\mathcal{D}\check R}$ unaltered near $\partial_{\rm vert} \check R \, \cup  \,  \mathcal{D} \partial_{\pm} \check R  \subset \mathcal{D}\check R$.
First, rearrangement of critical points, as in \cite{Milnor}*{Theorem 4.1}, ensures that for critical points $q_1, q_2 \in \mathcal{D}\check R$ with ${\rm ind}(q_1) <  {\rm ind}(q_2)$, the value of $q_1$ is smaller than the one of $q_2$. 
This process alters $-f_{\mathcal{D}\check R}$ in an arbitrarily small neighborhood of the union of stable and unstable manifolds of critical points of $-f_{\mathcal{D}\check R}$ and away from $\mathcal{D} \partial_{\pm} \check R$.
Hence, this step can be assumed to leave  $f_{\mathcal{D}\check R}$ unaltered near $\mathcal{D} \partial_{\pm} \check R$ and on the $\alpha/6$-neighborhood of $\partial_{\rm vert} \check R$.
Second, as in the proof of \cite{Milnor}*{Theorem 8.1}, for any critical point $p$ of index $0$ of $-f_{\mathcal{D}\check R}$, we find a critical point $q$ of index $1$ of $-f_{\mathcal{D}\check R}$ such that $q$ is connected to $p$  by one and only one flow line of $X_{\mathcal{D}\check R}$. 
According to \cite{Milnor}*{Theorem 5.4}, the critical point $p$ can be cancelled against $q$ by altering $-f_{\mathcal{D}\check R}$ in an arbitrarily small neighborhood of this flow line. 
Since this flow line is disjoint from $\mathcal{D} \partial_{\pm} \check R$ and has distance at least $\alpha/6$ from $\partial_{\rm vert} \check R$, we can assume that this cancellation leaves $-f_{\mathcal{D}\check R}$ unaltered near  $\mathcal{D} \partial_{\pm} \check R$ and on the $\alpha/12$-neighborhood of $\partial_{\rm vert} \check R$.
By induction, we can eliminate all critical points of index $0$ of $-f_{\mathcal{D}\check R}$ without altering $f_{\mathcal{D}\check R}$ near $\partial_{\rm vert} \check R \cup \mathcal{D} \partial_{\pm} \check R$. 
This proves our claim.

The restriction of the resulting Morse function to $\check R$ coincides with $f_{\check R}$ near $\partial \check R=\partial_+\check R\cup\partial_-\check R\cup\partial_\mathrm{vert}\check R$ and does not contain critical points of coindex $0$.
We call this Morse function $f_{\check R}$ again.

Precomposing $f_{\check R}$ with the projection $R\to\check R$ defines a $\Gamma$-invariant Morse function $\tilde f_R \colon R \to [c,1]$ with no critical orbits of coindex $0$.
Since all orbits in $R$ are of type $\Gamma/H$, it is special. 
By construction, $\tilde f_R$ coincides with $f_{R}$ near $\partial R = \partial_\mathrm{vert} R \cup \partial_+ R\cup\partial_- R$.
Taking the union of $\tilde f_R$  with $f_Q$, we obtain the required special $\Gamma$-invariant Morse function $\tilde f_Z\colon Z\to[c,1]$ with no critical orbits of coindex $0$.
\end{proof}

\begin{lemma}[Dimension bound on unstable manifolds] \label{smalldim}
We can assume that $\dim \mathscr{W}^u_i \leq \dim M-1$ for all $1 \leq i \leq k$.
\end{lemma}

\begin{proof}
By \cite{Mayer}*{Satz~2.2} together with the remark at the end of Section~2 in \cite{Mayer}, we can assume that the $\Gamma$-invariant Morse function $f$ is special.
Therefore, every critical orbit outside $M_{\Reg}$ has $\dim W_i^s\geq1$.
Hence all critical orbits of coindex $0$, equivalently all local maxima of $f|_{\mathrm{int}(M)}$, lie in $M_{\Reg}$.

Using that $f$ is special, \cite{Hanke08}*{Lemma~12} applied to $G:=\Gamma$ and $H$ allows us, after rearranging the critical orbits, to choose a regular value $c<1$ such that the critical orbits in $Z:=f^{-1}([c,1])$ are precisely those contained in $M_{\Reg}$.
Lemmas~\ref{lem.checkZ.connected}, \ref{lem.compact.core}, and \ref{lem.cancellation} provide a special $\Gamma$-invariant Morse function $\tilde f_Z$ on $Z$ with no critical orbits of coindex $0$ which agrees with $f_Z$ near $\partial Z$.
Extending $\tilde f_Z$ by $f$ on $M\setminus Z$ eliminates all critical orbits of coindex $0$ without changing $f$ near $\partial M$.
We may therefore assume that $\dim W_i^s\geq1$ for all $i=1,\ldots,k$.

Finally,
\begin{align*}
\dim \mathscr{W}^u_i
&=
\dim F_i^u
\stackrel{\eqref{defF}}{=}
\dim W_i^u+\dim\Gamma-\dim H_i\\
&\stackrel{\eqref{dimsum}}{=}\dim M-\dim W_i^s
\leq
\dim M-1.
\qedhere
\end{align*}
\end{proof}

\begin{proof}[Conclusion of the proof of Theorem~\ref{mainA}]
Applying Lemma~\ref{strati} inductively to $\mathscr{W} := \mathscr{W}^u_i$ and $\mathscr{C} := \mathscr{C}_{\leq i-1}$ for $1 \leq i \leq k$ and using Lemma~\ref{smalldim} and Lemma~\ref{morse_flow}~\ref{one} and \ref{two}, we obtain an open $\Gamma$-invariant neighborhood $U \subset M$ of $\mathscr{C}_{\leq k}$ and  a continuous map 
\[
    G_U \colon D^m \to \RR^{\Gamma}_{>0}(U)
\]
with $G_U(\xi) = g(\xi)|_U$ for $\xi \in \partial D^m$. 

Let $t_0$ be as in Lemma~\ref{morse_flow}~\ref{three} for this $U$. 
Then $(\Psi^t)_{t \in [0 ,t_0]}$ yields a smooth isotopy with $\Psi^{t_0}(M)\subset U$.    
Now define the continuous map $G \colon D^m \to \RR^{\Gamma}_{>0} (M)$ by
\[
  G(\xi) := 
  \begin{cases}  
      (\Psi^{t_0})^*( G_U(2 \xi)), & \text{ for } 0 \leq |\xi| \leq \tfrac{1}{2} , \\
      (\Psi^{2t_0(1 - |\xi|) })^* \left( g\left(\nicefrac{\xi}{|\xi|}\right) \right), & \text{ for } \tfrac{1}{2} \leq |\xi| \leq 1 .
  \end{cases} 
\]
Then  $G|_{\partial D^m} = g$, concluding the proof of Theorem~\ref{mainA}. 
\end{proof}

\section{Scalar and mean curvature in Riemannian submersions}  
\label{sec.scalmeansub}

First we recall some basic facts about Riemannian submersions, largely following Section~9.B in \cite{Besse}.
Let $(M,g)$ and $(B, \check{g})$ be connected Riemannian manifolds, possibly with boundary, and let
\[
    \pi \colon (M, g) \to (B, \check{g}) 
\]
be a Riemannian submersion.
For $b \in B$, we denote by $M_b := \pi^{-1}(b) \subset M$ the fiber over $b$ and by $\hat g_b$ the Riemannian metric on $M_b$ induced by $g$.

Let  $\mathscr{V} = \ker d\pi \subset TM$ be the vertical subbundle and let  $\mathscr{H} = \mathscr{V}^{\perp_g} \subset TM$ be the horizontal subbundle of the submersion $\pi$.
We obtain an orthogonal decomposition
\[
    g = g_\mathscr{V} + g_\mathscr{H}. 
\]
We denote the orthogonal projections $TM\to \mathscr{H}$ and $TM\to \mathscr{V}$ by the same symbols $\mathscr{H}$ and $\mathscr{V}$, respectively.
Following \cite{Besse}*{Chapter~9}, we define the $(1,2)$-tensor field $A$ on $M$ by
\begin{align*}
A_XY &:= \mathscr{H}(\nabla^g_{\mathscr{H}X}\mathscr{V}Y)  + \mathscr{V}(\nabla^g_{\mathscr{H}X}\mathscr{H}Y).
\end{align*}
For horizontal vector fields $X, Y \in C^{\infty}(\mathscr{H})$, equation~(9.24) in \cite{Besse} says that 
\begin{equation} \label{tensorA} 
    A_X Y = \tfrac{1}{2}  \mathscr{V} [X,Y] . 
\end{equation} 
In our applications, we will mostly consider submersions with totally geodesic fibers. 
In this case, the O'Neill formula \cite{Besse}*{(9.70d)} says that the scalar curvature of $g$ is given by 
\[
     \scal_g = \scal_{\check g} \circ \pi + \scal_{\hat g} - |A|^2 . 
\]

Assume $\partial M \neq \emptyset$.
We only consider the cases where  the base or the fiber of the submersion $\pi$  have empty boundary.
Let $\mathcal{N} \in C^{\infty}(TM|_{\partial M})$ be the exterior unit normal of $\partial M \subset M$ with respect to $g$ and let $H_g \colon \partial M \to \R$ denote the (unnormalized) mean curvature of $\partial M \subset M$ with respect to $\mathcal{N}$.
Our sign convention is such that the boundary of the closed unit ball in $\R^n$ has mean curvature $n-1$.

\begin{proposition} \label{meancurv1}  
Let $\pi \colon (M, g) \to (B, \check{g})$ be a Riemannian submersion.
Assume $\partial B = \emptyset$ and let $\mathcal{N}$ be vertical, i.e., $\mathcal{N} \in C^{\infty}(\mathscr{V}|_{\partial M})$. 
Let $b \in B$ and let  $H_{\hat g_b} \colon \partial M_b \to \R$ denote the (unnormalized) mean curvature of $\partial M_b \subset (M_b, \hat g_b)$ with respect to the exterior unit normal.
Then we have 
\[
	H_{g}|_{\partial M_b}=  H_{\hat g_b} . 
\]
\end{proposition} 

\begin{proof}
By assumption, we have $\mathcal{N}|_{\partial M_b} \in C^{\infty}(TM_b|_{\partial M_b})$.
Since $\mathcal{N} \perp_g T \partial M_b$, this implies that $\mathcal{N}|_{\partial M_b}$ is the exterior unit normal of $\partial M_b \subset (M_b, \hat g_b)$.

Let $\II^g$ be the second fundamental form of $\partial M$ in $M$ and let $\II^{\hat g_b}$ be the second fundamental form of $\partial M_b$ in $M_b$, both with respect to $\mathcal{N}$.
Along $\partial M_b$, we find for vertical tangent vectors $X, Y \in T\partial M_b$ that
\[
\II^g(X,Y) = \langle \nabla^g_X\mathcal{N},Y\rangle_g = \langle \nabla^{\hat g_b}_X\mathcal{N},Y\rangle_{\hat g_b} = \II^{\hat g_b}(X,Y) .
\]
Using \cite{Besse}*{(9.21d)}, we find for $X, Y \in \mathscr{H}$ along $\partial M$ that
\[
\II^g(X,Y) 
= \langle \nabla^g_X\mathcal{N},Y\rangle_g 
= \langle \mathscr{H}(\nabla^g_X\mathcal{N}),Y\rangle_g 
= \langle A_X\mathcal{N},Y\rangle_g 
= - \langle A_XY,\mathcal{N}\rangle_g .
\]
By \eqref{tensorA}, $A_XY$ is skew-symmetric in $X$ and $Y$ while $\II^{g}(X,Y)$ is symmetric.
Hence, $\II^g(X,Y) = 0$.

Summarizing, with respect to the orthogonal splitting $T\partial M = T\partial M_b \oplus \mathscr{H}$, the second fundamental form takes the block form
\[
\II^g = \begin{pmatrix} \II^{\hat g_b} & * \\ * & 0 \end{pmatrix} .
\]
Taking traces concludes the proof.
\end{proof}

For $\tau > 0$ let 
\begin{equation} \label{can_var}
    g_\tau  = \tau  g_{\mathscr{V}} + g_{\mathscr{H}}  
\end{equation}
be the canonical variation of $g$, see \cite{Besse}*{Definition~9.67}.
Then $\pi \colon (M, g_{\tau}) \to (B, \check{g})$ is still a Riemannian submersion with the same vertical and horizontal subbundles as $(M,g) \to (B, \check{g})$. 

Tensorial quantities with respect to $g_\tau$ are indicated by a superscript $\tau$. 
Using the formulas in \cite{Besse}*{Section~9.G}, we obtain $|A^\tau|_{g_\tau}^2 = \tau \cdot |A|^2_g$.
Since $\scal_{\widehat{g_\tau}} = \frac{1}{\tau} \scal_{\hat g}$, we obtain for submersions with totally geodesic fibers that
\begin{equation} \label{oneill} 
     \scal_{g_\tau} = \scal_{\check g} \circ \pi +  \tfrac{1}{\tau} \scal_{\hat g} - \tau |A|^2 , 
\end{equation} 
cf.\ \cite{Besse}*{(9.70d)}.
In particular, if $\scal_{\hat g} \geq 0$, then for all $0 < \tau \leq 1$, we get $\scal_{g_\tau} \geq \scal_g$.

\begin{proposition} \label{meancurv2} 
Let $\pi \colon (M, g) \to (B, \check{g})$ be a Riemannian submersion.
Let $\partial M_b  =  \emptyset$ for all $b \in B$ and let $\mathcal{N}$ be horizontal, i.e., $\mathcal{N} \in C^{\infty}( \mathscr{H}|_{\partial M})$.
Then, for  $\tau > 0$, we have 
\[
    H_{g_\tau} = H_{g} .
\]
\end{proposition} 
 
\begin{proof}
Consider the Weingarten map $W^\tau = \nabla^{g_\tau}_{\bullet}\mathcal{N}\colon T\partial M\to T\partial M$ of $\partial M$ in $M$ with respect to $g_\tau$.
For vertical vector fields $X$ and $Y$ we have, by the Koszul formula for the Levi-Civita connection and the fact that $\mathcal{N}$ is horizontal and $[X,Y]$ is vertical, that 
\begin{align*}
2g_\tau (\nabla^{g_\tau}_X\mathcal{N},Y) 
&= 
X(g_\tau(\mathcal{N},Y)) + \mathcal{N}(g_\tau(X,Y)) - Y(g_\tau(X,\mathcal{N})) \\
&\quad + g_\tau ([X,\mathcal{N}],Y) - g_\tau([X,Y],\mathcal{N}) - g_\tau([\mathcal{N},Y],X) \\
&=
\tau \big(\mathcal{N}(g_1(X,Y))+ g_1 ([X,\mathcal{N}],Y) - g_1([\mathcal{N},Y],X) \big) .
\end{align*}
In particular, for $\tau=1$ we have 
\[
2g_1 (\nabla^{g_1}_X\mathcal{N},Y) 
= 
\mathcal{N}(g_1(X,Y))+ g_1 ([X,\mathcal{N}],Y) - g_1([\mathcal{N},Y],X) .
\]
Inserting this back into the previous equation, we find that 
\[
g_\tau (\nabla^{g_\tau}_X\mathcal{N},Y) 
= 
\tau g_1 (\nabla^{g_1}_X\mathcal{N},Y)
= 
g_\tau (\nabla^{g_1}_X\mathcal{N},Y).
\]
Therefore $\mathscr{V}\circ W^{\tau}\circ\mathscr{V}=\mathscr{V}\circ W^{1}\circ\mathscr{V}$.
For $X\in \mathscr{H}\cap T\partial M$ we find
\[
(\mathscr{H}\circ W^\tau) X = \mathscr{H}(\nabla^{g_\tau}_X\mathcal{N}) = \mathscr{H}(\nabla^{g}_X\mathcal{N}) 
\]
because $d\pi(\nabla^{g_\tau}_X\mathcal{N}) = \nabla^{\check g}_{d\pi(X)}d\pi(\mathcal{N})$ is independent of $\tau$.
Hence, with respect to the splitting $T\partial M = \mathscr{V}\oplus (\mathscr{H}\cap T\partial M)$, the Weingarten map $W^\tau$ takes the block form
\[
W^\tau 
= 
\begin{pmatrix} \mathscr{V}\circ W^{1}\circ\mathscr{V} & * \\ * & \mathscr{H}(\nabla^{g}_\bullet\mathcal{N}) \end{pmatrix} 
\]
with diagonal blocks independent of $\tau$.
Taking the trace concludes the proof.
\end{proof}
 
We describe an example of a Riemannian submersion which will be of importance later on.

\begin{example} \label{ex.vectorbundle}
Let $\pi\colon \nu\to B$ be a real vector bundle equipped with a Riemannian fiber metric and a compatible connection.

Given a smooth curve $c\colon [0,1]\to B$ and an element $\eta_0\in \nu_{c(0)}$, parallel transport (with respect to the connection) of $\eta_0$ along~$c$ yields a curve $\eta\colon [0,1]\to \nu$ with $\eta(0)=\eta_0$.
The tangent vectors $\dot{\eta}(0)$ obtained in this manner form a subspace of $T_{\eta_0}\nu$ which we denote by $\mathscr{H}_{\eta_0}$.

Denote the zero-section of $\nu$ by $B_0$.
Since parallel translates of a zero-vector are zero-vectors, $\mathscr{H}$ coincides with the tangent bundle of $B_0$ along $B_0$.
Since the connection is metric, parallel sections have constant length and hence $\mathscr{H}$ is always tangential to the $r$-sphere subbundles of $\nu$ where $r>0$.
Since the connection is linear, the distribution $\mathscr{H}$ is invariant under dilations $\nu\to\nu$, $\eta\mapsto r\eta$.

Denoting the tangent spaces of the fibers of $\pi$ by $\mathscr{V}$, we have $T\nu = \mathscr{V}\oplus \mathscr{H}$.
Thus $\pi\colon\nu\to B$ is a submersion with horizontal distribution $\mathscr{H}$.
In order to turn it into a Riemannian submersion, we equip the total space $\nu$ with a Riemannian metric $g$ such that $\mathscr{H}$ is orthogonal to $\mathscr{V}$.
The metric on $\mathscr{H}$ needs to be chosen exactly such that $d\pi\colon \mathscr{H}\to TB$ is a linear isometry at each point of $\nu$.
Hence, it is uniquely determined by the metric $\check{g}$ on $B$.

A canonical choice for the metric on $\mathscr{V}$ is to take the metric induced from the Euclidean vector space structure of the fibers of $\nu$ via the canonical isomorphism $\mathscr{V}_\eta = T_\eta(\nu_{\pi(\eta)}) \cong \nu_{\pi(\eta)}$.
More generally, we can fix any $\O(k)$-invariant Riemannian metric $g_0$ on $\R^k$ where $k$ is the rank of $\nu$.
For $p\in B$ we choose a linear isometry $\nu_p \to \R^k$ and define the metric on the fiber $\nu_p$ to be the pullback of $g_0$ via this isometry.
By $\O(k)$-invariance, this Riemannian metric on $\nu_p$ is independent of the choice of the linear isometry.

Since the connection is metric, parallel transport along any curve in $B$ yields a linear isometry between the fibers of $\nu$ w.r.t.\ the Euclidean metric and hence also a Riemannian isometry w.r.t.\ the metric induced by $g_0$.
This implies, by an argument due to Vilms (see \cite{Vilms}*{p.~78}), that the fibers $\nu_p$ are totally geodesic submanifolds of $\nu$.

If the bundle $\nu$ is oriented, then the group $\O(k)$ can be replaced by the group $\SO(k)$ throughout this discussion.
\end{example}

In our applications, the situation described in the previous example arises as follows:

\begin{example}  \label{ex.normalbundle}

Let $X$ be a manifold and let $B\subset X$ be a submanifold.
Consider the \emph{abstract normal bundle} $\pi\colon\nu := (TX|_B)/TB \to B$.

Given a Riemannian metric $g_X$ on $X$, we obtain an induced metric $g_B$ on $B$ and the \emph{geometric normal bundle} $\nu^g$ whose fibers are the orthogonal complements of $TB$ in $TX|_B$ with respect to $g_X$.
The composition $\nu^g \hookrightarrow TX|_B \twoheadrightarrow \nu$ is a vector bundle isomorphism.
Via this isomorphism, $\nu$ inherits a fiberwise scalar product turning it into a Riemannian vector bundle.
The Levi-Civita connection of $g_X$ induces a metric connection on $\nu^g$ and hence on $\nu$, called the \emph{normal connection}.
\end{example}

For oriented real vector bundles of rank two, there is the following useful construction of linear connections and Riemannian submersions with totally geodesic fibers.

\begin{example} \label{ex.circlebundle}
Let $\nu\to B$ be a real vector bundle of rank $2$ equipped with a Riemannian fiber metric. 
Furthermore, we assume that the bundle $\nu$ is oriented.

Any rotation matrix $\begin{pmatrix}\cos\theta & -\sin\theta \\ \sin\theta & \cos\theta \end{pmatrix} \in \SO(2)$ acts on the fibers $\nu_p$ of $\nu$ by rotating about the angle $\theta$ in the direction determined by the orientation.
This yields an  $\SO(2)$-action by vector bundle automorphisms on $\nu$ which is free outside the zero section.

Fix $r > 0$ and let $S_r(\nu) \to B$ be the $r$-sphere bundle of $\nu$.
The $\SO(2)$-action restricts to a free, fiber-transitive action on $S_r(\nu)$ turning $S_r(\nu)$ into an $\SO(2)$-principal bundle.
The map $S_r(\nu)\times\R^2 \to \nu$ given by $(v,(x_1,x_2))\mapsto \frac{1}{r}(x_1v+x_2R_{\pi/2}v)$ induces an isomorphism of oriented Riemannian vector bundles between the associated vector bundle induced by the standard representation of $\SO(2)$ on $\R^2$ and the vector bundle~$\nu$.
Now let 
\[
    \mathscr{H}  \subset TS_r(\nu)
\]
be an $\SO(2)$-invariant distribution complementary to the vertical distribution.
This defines a principal bundle connection on the $\SO(2)$-principal bundle $S_r(\nu)$ and hence induces a metric linear connection on the associated vector bundle $\nu$.
The horizontal distribution of this linear connection is given by the image of $\mathscr{H} \times T\R^2$ under the differential of the projection $S_r(\nu) \times \R^2 \to (S_r(\nu) \times \R^2)/\SO(2) = \nu$, see \cite{Besse}*{Section~9.55}.
In particular, along $S_r(\nu)$ it coincides with the given distribution $\mathscr{H}$.

Now any $\SO(2)$-invariant Riemannian metric on $\R^2$ induces a Riemannian metric on the total space $\nu$ such that the bundle map $\nu\to B$ is a Riemannian submersion with totally geodesic fibers as described in Example~\ref{ex.vectorbundle}.
\end{example}

\section{Scalar curvature under deformations of invariant metrics} 
\label{sec.Killing}

Let $M$ be a smooth, not necessarily compact $S^1$-manifold, possibly with boundary.
Let $g$ be an $S^1$-invariant Riemannian metric on $M$.
Furthermore, let $\X \in C^{\infty}(TM)$ be the Killing field of the $S^1$-action and let  $\omega^{g}$ be the corresponding $1$-form via the musical isomorphism with respect to $g$.

Let $\alpha, \beta\colon I\subset\R\to\R$ be smooth functions where the domain $I$ contains the range of the function $|{\X}|^2$ on $M$.
We obtain two smooth functions $\alpha\circ |{\X}|^2,\beta\circ|{\X}|^2\colon M\to\R$ which we denote by $\alpha^g$ and $\beta^g$, respectively.
If the metric $g$ is obvious from the context, we drop the superscripts from $\omega^g$, $\alpha^g$ and $\beta^g$.
It will always be clear from the context whether $\alpha$ and $\beta$ are regarded as functions on $\R$ or on $M$.
The same applies to the derivatives $\dot{\alpha}$ and  $\ddot{\alpha}$ with respect to $t$ and similarly for $\beta$.

We start with an auxiliary lemma.
\begin{lemma}
\label{lem:DeltaA}
We have
\begin{gather*}
\grad\,\alpha
=
-2\dot{\alpha}\cdot\nabla_{\X}{\X} ,\\
\div(\nabla_{\X}{\X})
=
\ric({\X},{\X}) - |\nabla {\X}|^2 ,\\
\Delta \alpha 
= 
2 \dot{\alpha}\cdot(\ric({\X},{\X}) - |\nabla {\X}|^2) -4 \ddot{\alpha}\cdot|\nabla_{\X}{\X}|^2 ,
\end{gather*}
and similarly for $\beta$.
\end{lemma}

\begin{proof}
Since ${\X}$ is a Killing field, its covariant differential $\nabla_\bullet {\X}$ is skew-symmetric.
For the computation we use an orthonormal frame $e_1,\dots,e_n$ which is synchronous at the point under consideration.
For the gradient we find
\begin{align*}
\grad\,\alpha 
&=
\dot{\alpha}\cdot\grad \big(|{\X}|^2\big) \\
&=
\dot{\alpha}\cdot\sum_i \partial_{e_i} \big(|{\X}|^2\big) e_i\\
&=
2\dot{\alpha}\cdot\sum_i \<\nabla_{e_i}{\X},{\X}\> e_i\\
&=
-2\dot{\alpha}\cdot\sum_i \<\nabla_{\X}{\X},e_i\> e_i\\
&=
-2\dot{\alpha}\cdot\nabla_{\X}{\X} .
\end{align*}

For the Laplacian we get
\begin{align}
\Delta \alpha
&=
-\div\,\grad\,\alpha\notag\\
&=
2\, \div(\dot{\alpha}\cdot\nabla_{\X}{\X}) \notag\\
&=
2\, \<\grad\,\dot{\alpha},\nabla_{\X}{\X}\> + 2\, \dot{\alpha}\,\div(\nabla_{\X}{\X}) \notag\\
&=
-4\, \ddot{\alpha}\,|\nabla_{\X}{\X}|^2 + 2\, \dot{\alpha}\,\div(\nabla_{\X}{\X}) .
\label{eq:DeltaA1}
\end{align}

Now observe
\begin{align}
\div(\nabla_{\X}{\X})
&=
\sum_i \<\nabla_{e_i}\nabla_{\X}{\X},e_i\> \notag\\
&=
\sum_i \<R(e_i,{\X}){\X}+\nabla_{\X}\nabla_{e_i}{\X} + \nabla_{[e_i,{\X}]}{\X},e_i\> \notag\\
&=
\ric({\X},{\X}) + \sum_i \big(\partial_{\X}\<\nabla_{e_i}{\X},e_i\> + \<\nabla_{\nabla_{e_i}{\X}}{\X},e_i\>\big) \notag\\
&=
\ric({\X},{\X}) + 0 - \sum_i \<\nabla_{e_i}{\X},\nabla_{e_i}{\X}\> \notag\\
&=
\ric({\X},{\X}) - |\nabla {\X}|^2.
\label{eq:DeltaA3}
\end{align}
Inserting \eqref{eq:DeltaA3} into \eqref{eq:DeltaA1} concludes the proof.
\end{proof}

\begin{example}
The choice $\alpha(t)=t$ yields
\begin{align*}
\grad|{\X}|^2
&=
-2\nabla_{\X}{\X} ,\\
\Delta |{\X}|^2
&=
2(\ric({\X},{\X})-|\nabla {\X}|^2).
\end{align*}
\end{example}

We are interested in deformations of the metric of the form 
\[
    g^\lambda = g + \lambda (\alpha g + \beta  \omega\otimes\omega) . 
\]
We will tacitly assume that for $|\lambda|$ small enough $g^\lambda$ is positive definite. 
This is automatic if $M$ is compact.

\begin{remark} \label{reinterpret} Assume that the $S^1$-action on $M$ is free and let $\check g$ be the induced metric on the orbit space $M/S^1$. 
The orbit map induces a Riemannian submersion
\[
          (M, g) \to (M/S^1, \check{g}) .
 \]
 From this point of view, the canonical variation $g_\tau$ defined in \eqref{can_var} is equal to  $g^{\lambda}$ where 
\[
      \alpha = 0 , \quad \beta(t)  = \frac{1}{t}, \quad \lambda = \tau - 1 . 
\]
\end{remark}

\begin{lemma} \label{S1inv} 
Each metric $g^\lambda$ is $S^1$-invariant.
\end{lemma}

\begin{proof}
We check that the Lie derivative $\LL_{\X} g^\lambda$ vanishes.
Since $|{\X}|$ is constant along the $S^1$-orbits, so are $\alpha$ and $\beta $.
Thus, 
\[
\LL_{\X}(\alpha g) = \partial_{\X} \alpha \cdot g + \alpha\cdot \LL_{\X} g = 0\cdot g + \alpha \cdot 0 = 0
\]
and 
\[
\LL_{\X}(\beta \omega\otimes\omega)
=
\partial_{\X}\beta \cdot\omega\otimes\omega + \beta\cdot (\LL_{\X}\omega)\otimes\omega + \beta \omega\otimes(\LL_{\X}\omega)
=
0
\]
because
\begin{align*}
\big(\LL_{\X}\omega\big)(Y)
&=
\partial_{\X}(\omega(Y))-\omega(\LL_{\X}Y)\\
&=
\partial_{\X}\<{\X},Y\>-\<{\X},[{\X},Y]\>\\
&=
\<\nabla_{\X}{\X},Y\>+\<{\X},\nabla_{\X}Y\>-\<{\X},\nabla_{\X}Y\>+\<{\X},\nabla_Y{\X}\>\\
&=
\<\nabla_{\X}{\X},Y\>+\<\nabla_Y{\X},{\X}\>
=
0.
\end{align*}
This concludes the proof.
\end{proof}

\begin{remark} \label{Gamma_inv} 
Suppose $\Gamma$ is a compact Lie group containing $S^1$ as a normal subgroup such that  the given $S^1$-action on $M$ extends to a smooth $\Gamma$-action.
Suppose that $g$ is $\Gamma$-invariant.
Then,  for each $\gamma \in \Gamma$, we have $\gamma^*(\X) = \X$ or $\gamma^*(\X) = - \X$.
So $\omega^g \otimes \omega^g$ is $\Gamma$-invariant and all the metrics $g^{\lambda}$ are $\Gamma$-invariant.
\end{remark}

Next we study the behavior of scalar curvature under this kind of metric variation.
We start with the case $\beta =0$.

\begin{lemma}
\label{lem:B=0}
Let $\beta =0$. 
Then
\[
\tfrac{d}{d\lambda}\big|_{\lambda=0} \scal_{g^\lambda}
= 
-\alpha\, \scal_g+(n-1)\Delta \alpha .
\]
\end{lemma}
\begin{proof}
In this case, $g^\lambda = (1+\lambda \alpha )g$ is conformally equivalent to $g$.
Using \cite{Besse}*{Theorem~1.159~(f)} we find, putting $e^{2f}=1+\lambda \alpha $,
\begin{align*}
\scal_{g^\lambda}
&=
e^{-2f}(\scal_g + 2(n-1)\Delta f -(n-1)(n-2)|df|^2) \\
&=
\frac{1}{1+\lambda \alpha }\Big(\scal_g + 2(n-1)\Delta \big(\tfrac12\log(1+\lambda \alpha )\big) \\
&
\phantom{\frac{1}{1+\lambda \alpha }\Big(}\quad -(n-1)(n-2)\big|d\big(\tfrac12\log(1+\lambda \alpha )\big)\big|^2\Big) \\
&=
\big(1-\lambda \alpha  + \OO(\lambda^2)\big)\big(\scal_g + (n-1)\Delta (\lambda \alpha + \OO(\lambda^2))+ \OO(\lambda^2)\big) \\
&=
(1-\lambda \alpha )(\scal_g + (n-1)\lambda\Delta \alpha )  + \OO(\lambda^2)\\
&=
\scal_g + \lambda(-\alpha \scal_g+ (n-1) \Delta \alpha )+ \OO(\lambda^2).
\end{align*}
This concludes the proof.
\end{proof}

Now we consider the case $\alpha =0$.

\begin{lemma}
\label{lem:A=0}
Let $\alpha =0$.
Then 
\begin{align*}
\tfrac{d}{d\lambda}\big|_{\lambda=0} \scal_{g^\lambda} 
&= 
2\big(\dot{\beta}|{\X}|^2 + \beta\big)\ric_g({\X},{\X}) 
-\big(2\dot{\beta}|{\X}|^2 + 3\beta\big)|\nabla {\X}|^2 \\
&\quad-\big(4\ddot{\beta}|{\X}|^2+10\dot{\beta}\big)|\nabla_{\X}{\X}|^2 .
\end{align*}
\end{lemma}

\begin{proof}
We use the general variation formula for scalar curvature under the metric perturbation
\begin{equation}
\tfrac{d}{d\lambda}\big|_{\lambda=0} \scal_{g+\lambda h} 
= 
\Delta(\tr_gh) + \delta\delta h - \langle \ric_g,h\rangle ,
\label{eq:ScalV_VarForm}
\end{equation}
where $h=\beta \omega\otimes\omega$, see \cite{Besse}*{Theorem~1.174~(e)}.
We observe
\begin{align*}
\tr_gh = \beta |{\X}|^2
\end{align*}
and hence, using Lemma~\ref{lem:DeltaA},
\begin{align}
\Delta \tr_gh
&=
\Delta \beta  \cdot |{\X}|^2 + \beta  \cdot \Delta |{\X}|^2 -2 \<\grad \beta ,\grad |{\X}|^2\> \notag\\
&=
(2 \dot{\beta}(\ric_g({\X},{\X}) - |\nabla {\X}|^2) -4 \ddot{\beta}|\nabla_{\X}{\X}|^2)\cdot |{\X}|^2 \notag\\
&\quad
+\beta\cdot(2(\ric_g({\X},{\X})-|\nabla {\X}|^2)) 
-2 \<-2\dot{\beta}\cdot\nabla_{\X}{\X},-2\nabla_{\X}{\X}\> \notag\\
&=
2(\dot{\beta}|{\X}|^2+\beta)(\ric_g({\X},{\X})-|\nabla {\X}|^2) - 4(\ddot{\beta}|{\X}|^2+2\dot{\beta})|\nabla_{\X}{\X}|^2 .
\label{eq:DeltaB1}
\end{align}
Moreover, we compute, using a synchronous frame,
\begin{align*}
\delta h
&=
\sum_i(\nabla_{e_i}h)(e_i,\bullet) \\
&=
\sum_i(\nabla_{e_i}(\beta \omega\otimes\omega))(e_i,\bullet) \\
&=
\sum_i(\partial_{e_i}\beta \cdot \omega\otimes\omega + \beta \cdot(\nabla_{e_i}\omega)\otimes\omega + \beta \omega\otimes\nabla_{e_i}\omega)(e_i,\bullet) \\
&=
\partial_{\X}\beta \cdot\omega + \beta  \sum_i \<\nabla_{e_i}{\X},e_i\>\omega + \beta \sum_i \omega(e_i)\nabla_{e_i}\omega \\
&=
0 + 0 + \beta  \nabla_{\X}\omega 
\end{align*}
and
\begin{align}
\delta\delta h
&=
\delta(\beta  \nabla_{\X}\omega) \notag\\
&=
\sum_i \nabla_{e_i}(\beta  \nabla_{\X}\omega)(e_i) \notag\\
&=
\sum_i \partial_{e_i}(\beta  (\nabla_{\X}\omega)(e_i))  \notag\\
&=
\sum_i \partial_{e_i}(\beta  \<\nabla_{\X}{\X},e_i\>) \notag\\
&=
\sum_i (\partial_{e_i}\beta  \<\nabla_{\X}{\X},e_i\> + \beta \<\nabla_{e_i}\nabla_{\X}{\X},e_i\>) \notag\\
&=
\<\nabla_{\X}{\X},\grad_g \beta \> + \beta\, \div_g (\nabla_{\X}{\X}) \notag\\
&=
\<\nabla_{\X}{\X},-2\dot{\beta}\nabla_{\X}{\X}\> + \beta (\ric_g({\X},{\X})-|\nabla {\X}|^2) \notag\\
&=
-2\dot{\beta}|\nabla_{\X}{\X}|^2 + \<\ric_g,h\> - \beta |\nabla {\X}|^2.
\label{eq:DeltaB7}
\end{align}
Inserting \eqref{eq:DeltaB1} and \eqref{eq:DeltaB7} into \eqref{eq:ScalV_VarForm} yields
\begin{align*}
\tfrac{d}{d\lambda}\big|_{\lambda=0} \scal_{g+\lambda h} 
&= 
2(\dot{\beta}|{\X}|^2 + \beta)(\ric_g({\X},{\X}) - |\nabla {\X}|^2) 
-4(\ddot{\beta}|{\X}|^2+2\dot{\beta})|\nabla_{\X}{\X}|^2 \\
&\quad
-2\dot{\beta}|\nabla_{\X}{\X}|^2 - \beta |\nabla {\X}|^2 \\
&=
2(\dot{\beta}|{\X}|^2 + \beta)\ric_g({\X},{\X}) 
-(2\dot{\beta}|{\X}|^2 + 3\beta)|\nabla {\X}|^2 \\
&\quad -(4\ddot{\beta}|{\X}|^2+10\dot{\beta})|\nabla_{\X}{\X}|^2 .
\end{align*}
This concludes the proof.
\end{proof}

Combining Lemmas~\ref{lem:DeltaA}, \ref{lem:B=0} and \ref{lem:A=0}, we get

\begin{proposition}
\label{prop:ScalVar}
The scalar curvature under the variation $g^\lambda = g + \lambda (\alpha  g + \beta  \omega\otimes\omega)$ satisfies
\begin{align*}
\tfrac{d}{d\lambda}\big|_{\lambda=0} \scal_{g^\lambda}
&= 
-\alpha\,\scal_g +2((n-1)\dot{\alpha}+\dot{\beta}|{\X}|^2 + \beta)\ric_g({\X},{\X}) \\
&\quad
-(2(n-1)\dot{\alpha}+2\dot{\beta}|{\X}|^2 + 3\beta)|\nabla {\X}|^2 \\
&\quad
-(4(n-1)\ddot{\alpha}+4\ddot{\beta}|{\X}|^2+10\dot{\beta})|\nabla_{\X}{\X}|^2 .
\QED
\end{align*}
\end{proposition}

\begin{example}
\label{ex:a0b1}
Consider the constant functions $\alpha=0$ and $\beta=1$.
Then Proposition~\ref{prop:ScalVar} becomes
\[
\tfrac{d}{d\lambda}\big|_{\lambda=0} \scal_{g^\lambda}
=
2\,\ric_g({\X},{\X}) - 3\,|\nabla {\X}|^2 .
\]
\end{example}

\begin{lemma}
\label{lem:KillingEst}
For any Killing vector field ${\X}$ we have
\[
|{\X}|^2|\nabla {\X}|^2 \ge 2 |\nabla_{\X}{\X}|^2.
\]
\end{lemma}

The point of the lemma is the factor $2$ in the estimate.
Without it, the inequality holds true for all vector fields.

\begin{proof}[Proof of Lemma~\ref{lem:KillingEst}]
At points where $\nabla_{\X}{\X}$ vanishes the assertion holds trivially.
So assume $\nabla_{\X}{\X}\neq 0$ at the point under consideration.
In particular, ${\X}\neq0$ at that point.
We put
\begin{align*}
e_1 := \frac{\X}{|{\X}|} \text{ and }   
e_2 := \frac{\nabla_{\X}{\X}}{|\nabla_{\X}{\X}|} .
\end{align*}
Observe that $e_1$ and $e_2$ are perpendicular because $\nabla {\X}$ is skew-symmetric.
We can complete to an orthonormal basis $e_1,e_2,\dots,e_n$ of the tangent space under consideration.
Now
\begin{align*}
|\nabla_{e_2}{\X}|^2
&=
\<\nabla_{e_2}{\X},e_1\>^2 + \sum_{i\ge3}\<\nabla_{e_2}{\X},e_i\>^2 \\
&\ge
\<\nabla_{e_2}{\X},e_1\>^2 \\
&=
\frac{1}{|\nabla_{\X}{\X}|^2|{\X}|^2}\<\nabla_{\nabla_{\X}{\X}}{\X},{\X}\>^2\\
&=
\frac{1}{|\nabla_{\X}{\X}|^2|{\X}|^2}\<\nabla_{\X}{\X},\nabla_{\X}{\X}\>^2\\
&=
\frac{|\nabla_{\X}{\X}|^2}{|{\X}|^2} .
\end{align*}
Therefore
\begin{align*}
|\nabla {\X}|^2
&=
\sum_i |\nabla_{e_i}{\X}|^2 \\
&\ge 
|\nabla_{e_1}{\X}|^2 + |\nabla_{e_2}{\X}|^2 \\
&\ge 
2 \frac{|\nabla_{\X}{\X}|^2}{|{\X}|^2}.
\qedhere
\end{align*}
\end{proof}

\begin{example}
Consider $\alpha=0$ and $\beta(t)=\frac{1}{t}$.
This choice is possible only if ${\X}$ has no zeros, i.e., if the $S^1$-action is fixed-point free.
Then $\dot{\beta}(t)=-\frac{1}{t^2}$ and $\ddot{\beta}(t)=\frac{2}{t^3}$ and Proposition~\ref{prop:ScalVar} becomes
\begin{align*}
\tfrac{d}{d\lambda}\big|_{\lambda=0} \scal_{g^\lambda}
&=
- |{\X}|^{-2} |\nabla {\X}|^2 + 2\, |{\X}|^{-4}|\nabla_{\X}{\X}|^2 .
\end{align*}
By Lemma~\ref{lem:KillingEst}, $\tfrac{d}{d\lambda}\big|_{\lambda=0} \scal_{g^\lambda}\le0$ in this case.
\end{example}

We generalize this example in the following corollary.

\begin{corollary}
\label{cor:RicciWeg}
Assume $(n-1)\dot{\alpha}+\dot{\beta}t+\beta=0$ and $\dot{\beta}\le0$.
Then
\[
\tfrac{d}{d\lambda}\big|_{\lambda=0} \scal_{g^\lambda}
\le
-\alpha\,\scal_g + (n-1)\dot{\alpha}|\nabla {\X}|^2.
\]
\end{corollary}

\begin{proof}
By Proposition~\ref{prop:ScalVar} we have
\begin{align}
\tfrac{d}{d\lambda}\big|_{\lambda=0} \scal_{g^\lambda}
&=
-\alpha\,\scal_g - (2(n-1)\dot{\alpha}+2\dot{\beta}|{\X}|^2 + 3\beta)|\nabla {\X}|^2 \notag\\
&\quad
-(4(n-1)\ddot{\alpha}+4\ddot{\beta}|{\X}|^2+10\dot{\beta})|\nabla_{\X}{\X}|^2 .
\label{eq:SpecialChoice1}
\end{align}
Now we observe
\begin{align}
2(n-1)\dot{\alpha}+2\dot{\beta}t + 3\beta
&=
-2(\dot{\beta}t+\beta)+2\dot{\beta}t + 3\beta
=
\beta
\label{eq:SpecialChoice2}
\end{align}
and
\begin{align}
4(n-1)\ddot{\alpha}+4\ddot{\beta}t+10\dot{\beta}
&=
-4(2\dot{\beta}+\ddot{\beta}t)+4\ddot{\beta}t+10\dot{\beta}
=
2\dot{\beta}.
\label{eq:SpecialChoice3}
\end{align}
The last equation uses $(n-1)\ddot{\alpha}+2\dot{\beta}+\ddot{\beta}t=0$ which is obtained by differentiating $(n-1)\dot{\alpha}+\dot{\beta}t+\beta=0$.
Inserting \eqref{eq:SpecialChoice2} and \eqref{eq:SpecialChoice3} into \eqref{eq:SpecialChoice1} yields, using Lemma~\ref{lem:KillingEst},
\begin{align*}
\tfrac{d}{d\lambda}\big|_{\lambda=0} \scal_{g^\lambda}
&=
-\alpha\,\scal_g - \beta |\nabla {\X}|^2 -2\dot{\beta}|\nabla_{\X}{\X}|^2 \\
&\le
-\alpha\,\scal_g - \beta |\nabla {\X}|^2 -\dot{\beta}|{\X}|^2|\nabla {\X}|^2 \\
&=
-\alpha\,\scal_g + (n-1)\dot{\alpha}|\nabla {\X}|^2.
\qedhere
\end{align*} 
\end{proof}

\begin{example}
Generalizing Example~\ref{ex:a0b1}, choose a constant $C \in \R$ and consider $\alpha(t)=C$ and $\beta(t)=\frac{1}{t}$.
Still, this choice is only possible if ${\X}$ has no zeros because $\beta(0)$ is not defined.
Corollary~\ref{cor:RicciWeg} yields
\[
\tfrac{d}{d\lambda}\big|_{\lambda=0} \scal_{g^\lambda}
\le -C \,\scal_g .
\]
\end{example}

\begin{example} \label{ex:withoutzero} 
In order to allow zeros of ${\X}$, that is, if $M^{S^1} \neq \emptyset$, we fix constants $C \in\R$ and $\eps >0$.
We put $\alpha(t)= C + \frac{\eps}{n-1}\frac{1}{t+\eps}$ and $\beta(t)=\frac{1}{t+\eps}$.
Then $\dot{\alpha} = -\frac{\eps}{n-1}\frac{1}{(t+\eps)^2}$ and $\dot{\beta}=-\frac{1}{(t+\eps)^2}$.
The assumptions in Corollary~\ref{cor:RicciWeg} are still satisfied and we get
\begin{equation} \label{scaldeform}
\tfrac{d}{d\lambda}\big|_{\lambda=0} \scal_{g^\lambda}
\le
-\alpha\,\scal_g + (n-1)\dot{\alpha}|\nabla {\X}|^2
\le
-\Big(C + \frac{\eps}{n-1}\frac{1}{|{\X}|^2+\eps}\Big)\cdot\scal_g .
\end{equation}
If $\scal_g \ge0$ this implies,
\begin{equation}\label{scaldiffest}
\tfrac{d}{d\lambda}\big|_{\lambda=0} \scal_{g^\lambda}
\le
-C\,\scal_g .
\end{equation}
\end{example}

\begin{lemma} \label{lem:ODE} 
Let $\eps \geq 0$. 
If ${\X}$ has zeros, then we assume $\eps > 0$.  
Let 
\begin{equation} \label{defab}
    \alpha(t) = \frac{\eps}{n-1}\frac{1}{t+\eps} , \quad \beta(t) = \frac{1}{t+\eps}
\end{equation}
be defined as in Example~\ref{ex:withoutzero} for $C = 0$.

Let $\RR^{S^1}(M)$ be the space of Riemannian metrics on $M$, equipped with the weak $C^{\infty}$-topology.
Then, for each $g\in\RR^{S^1}(M)$, there is a unique smooth path  $\mathcal{P}_{\eps} \colon [0,\infty)$ $\to  \RR^{S^1}(M)$ satisfying the ODE 
\begin{equation} \label{system} 
\begin{split} 
            \mathcal{P}_{\eps} (0)   & = g ; \\
           \mathcal{P}_{\eps}'(s)    & = - \Big( \alpha^{\mathcal{P}_{\eps}(s)}  \cdot \mathcal{P}_{\eps}(s) + \beta^{\mathcal{P}_{\eps}(s)} \cdot \big( \omega^{\mathcal{P}_{\eps}(s)} \otimes \omega^{\mathcal{P}_{\eps}(s)}  \big) \Big) .
\end{split} 
\end{equation}     
The paths $\mathcal{P}_{\eps}$ depend continuously (w.r.t.\ the compact-open topology on the space of paths) on $\eps$ and on $g$.

Furthermore, if $\scal_g > 0$, then $\tfrac{d}{ds} \scal_{\mathcal{P}_{\eps}(s)} \geq 0$ on $M$ for all $s$. 
In particular, each $\mathcal{P}_{\eps}(s)$ has positive scalar curvature.
\end{lemma} 

\begin{proof} 
At each point $p\in M$, \eqref{system} constitutes an ODE on the space of symmetric $(0,2)$-tensors on $T_pM$.
Thus there is a unique maximal solution in the open subset of positive definite symmetric $(0,2)$-tensors.
By smooth dependence on the initial values, the solutions depend smoothly on $g$, and by smooth dependence on parameters, they depend smoothly on $p$ and $\eps$ wherever they exist.

To see that the solution is defined on all of $[0,\infty)$, let $p \in M$.
First assume $\X(p) \neq 0$.
Let $\kappa := |\X(p)|_g$, $V := \mathrm{span} \{ \X(p) \} \subset T_p M$ and $H := V^{\perp} \subset T_p M$.
Let $g_p = g_V + g_H$ be the induced splitting.
To determine $\mathcal{P}_{\eps}(s)$ at $p$, we work with  the ansatz
\[
    \mathcal{P}_{\eps}(s) = a(s) \cdot g_V + b(s) \cdot g_H 
\]
where $a, b$ are smooth functions with values in the positive reals.
Note that each $\mathcal{P}_{\eps}(s)$ is positive definite.
We obtain 
\begin{align*}
     \omega^{\mathcal{P}_{\eps}(s)} \otimes \omega^{\mathcal{P}_{\eps}(s)} & = \kappa^2 \cdot a(s)^2 \cdot g_V , \\
      | \X(p)|^2_{\mathcal{P}_{\eps}(s)} & = \kappa^2 \cdot a(s) . 
\end{align*}
Hence \eqref{system} is equivalent to the system
\begin{align} 
 \label{fora}  a(0) = 1, \quad  a' & = - \frac{\eps}{n-1} \frac{a}{\kappa^2 a + \eps}   - \frac{\kappa^2 a^2}{ \kappa^2 a  + \eps} \, , \\
 \label{forb}   b(0) = 1, \quad b' & = - \frac{\eps}{n-1} \frac{b}{\kappa^2 a + \eps} \, . 
 \end{align}  
For $\eps = 0$, this simplifies to $a(0) = b(0) =1$, $a' = - a$, $b' = 0$ with solution $a(s) = \exp(-s)$ and $b \equiv 1$ on $[0,\infty)$.
For $\eps > 0$, the maximal solution  of \eqref{fora} is monotonically  decreasing and positive, hence it is defined on $[0,\infty)$. 
Plugging this into \eqref{forb}, the maximal solution for $b$ is also monotonically decreasing and positive, hence also defined on $[0,\infty)$. 

If $\X(p) = 0$, we work with the ansatz
\[
    \mathcal{P}_{\eps}(s) = b(s) \cdot g .
\]
In particular, $\omega^{\mathcal{P}_{\eps}(s)} = 0$ for all $s$ and  \eqref{system}  is equivalent  to 
\begin{align} 
 \label{forb_new}   b(0) = 1, \quad b' & = - \frac{1}{n-1} b .
 \end{align}  
This is solved by $b(s) = \exp(-\tfrac{s}{n-1})$ on $[0,\infty)$.

Denoting by $\RR(M)$ the space of Riemannian metrics on $M$ with the weak $C^{\infty}$-topology, we thus get a smooth path $\mathcal{P}_{\eps} \colon [0,\infty) \to \RR(M)$ which solves \eqref{system}.
Since $\X$ and $g$ are $S^1$-invariant, the path $\mathcal{P}_{\eps}$ takes values in $\RR^{S^1}(M)$.

Let $s_0 \in [0,\infty)$, and consider the deformation $(\mathcal{P}_{\eps}(s_0))^{\lambda}$ for $\alpha$ and $\beta$ defined in \eqref{defab}.
Then, using \eqref{system},
\[
     \tfrac{d}{ds}\big|_{s=s_0} \scal_{\mathcal{P}_{\eps}(s)} = - \tfrac{d}{d\lambda}\big|_{\lambda  =0} \scal_{(\mathcal{P}_{\eps}(s_0))^{\lambda}} .
\]
By \eqref{scaldiffest} with $C = 0$, the right-hand side is non-negative if $\scal_{\mathcal{P}_{\eps}(s_0)} \geq 0$.
Hence, if $\scal_g > 0$ then $\scal_{\mathcal{P}_{\eps}(s)} > 0$ for all $s \in [0,\infty)$.
This concludes the proof of Lemma~\ref{lem:ODE}. 
\end{proof} 

The path $\mathcal{P}_\eps$ is illustrated in Figure~\ref{fig.deform} for $M=S^2$ with the standard metric and the standard $S^1$-action by rotation.
\begin{figure}[ht]
\parbox{.06\textwidth}{\phantom{X}}
\parbox{.24\textwidth}{
\begin{overpic}[scale=0.11]{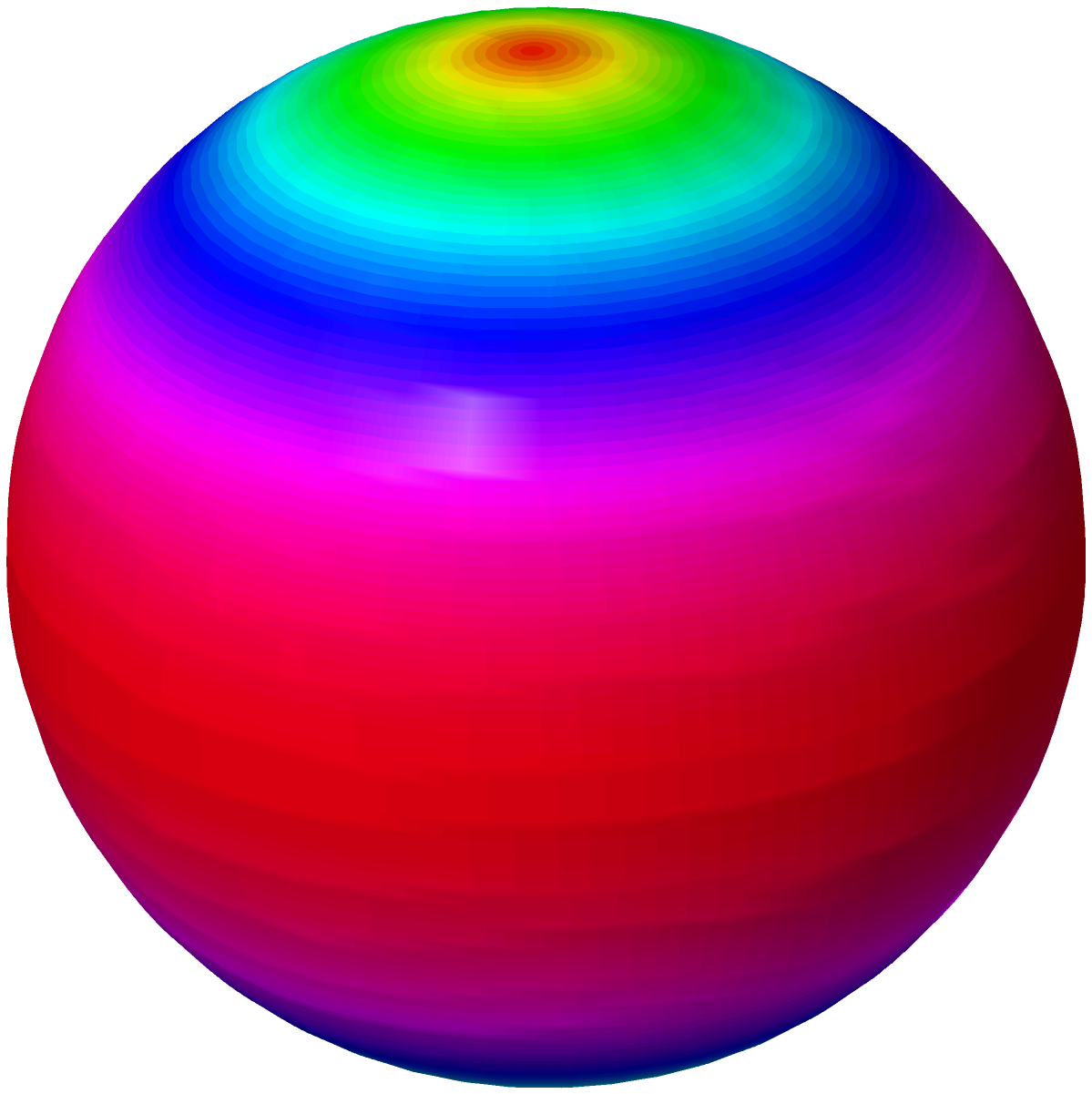}
\put(35,-19){$s=0$}
\end{overpic}

\vspace{10mm}
\includegraphics[scale=0.11]{Bilder/Sphere1.png}
}
\parbox{.22\textwidth}{
\vspace{-2mm}
\begin{overpic}[scale=0.1, trim=-20mm 0 0 0]{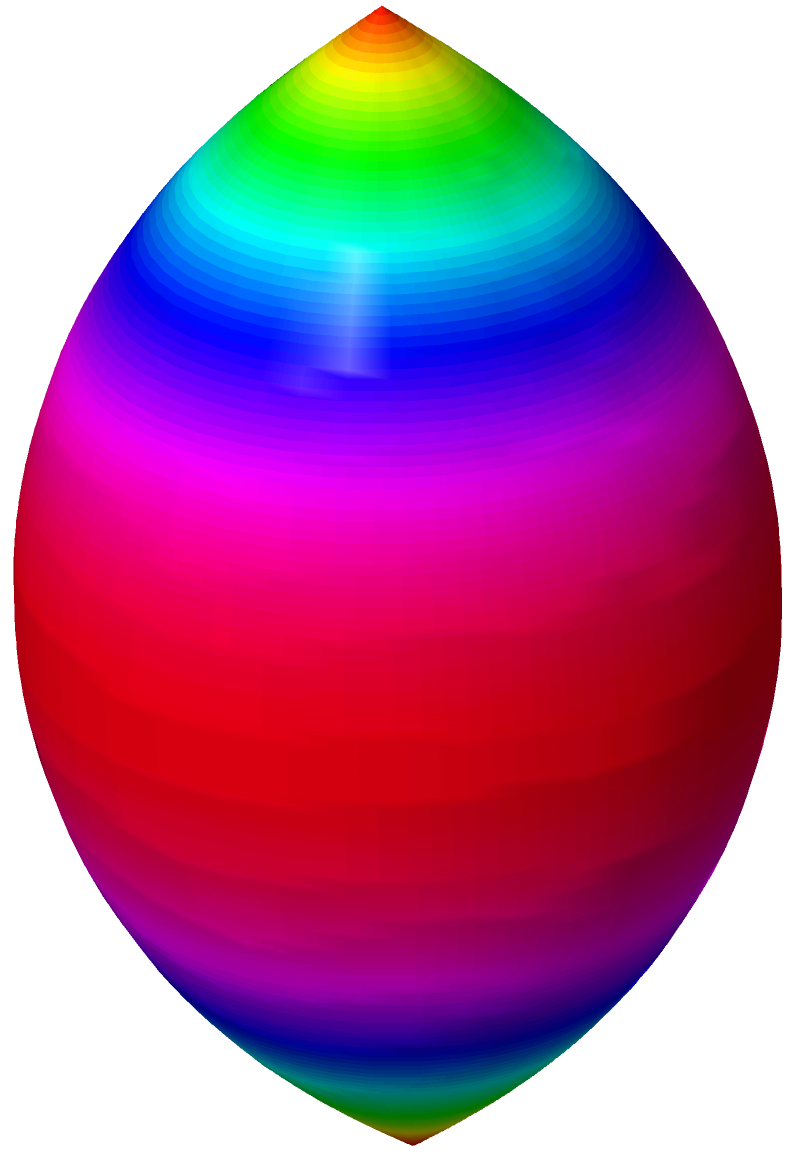}
\put(28,-22){$s=0.5$}
\end{overpic}

\vspace{10mm}
\includegraphics[scale=0.1]{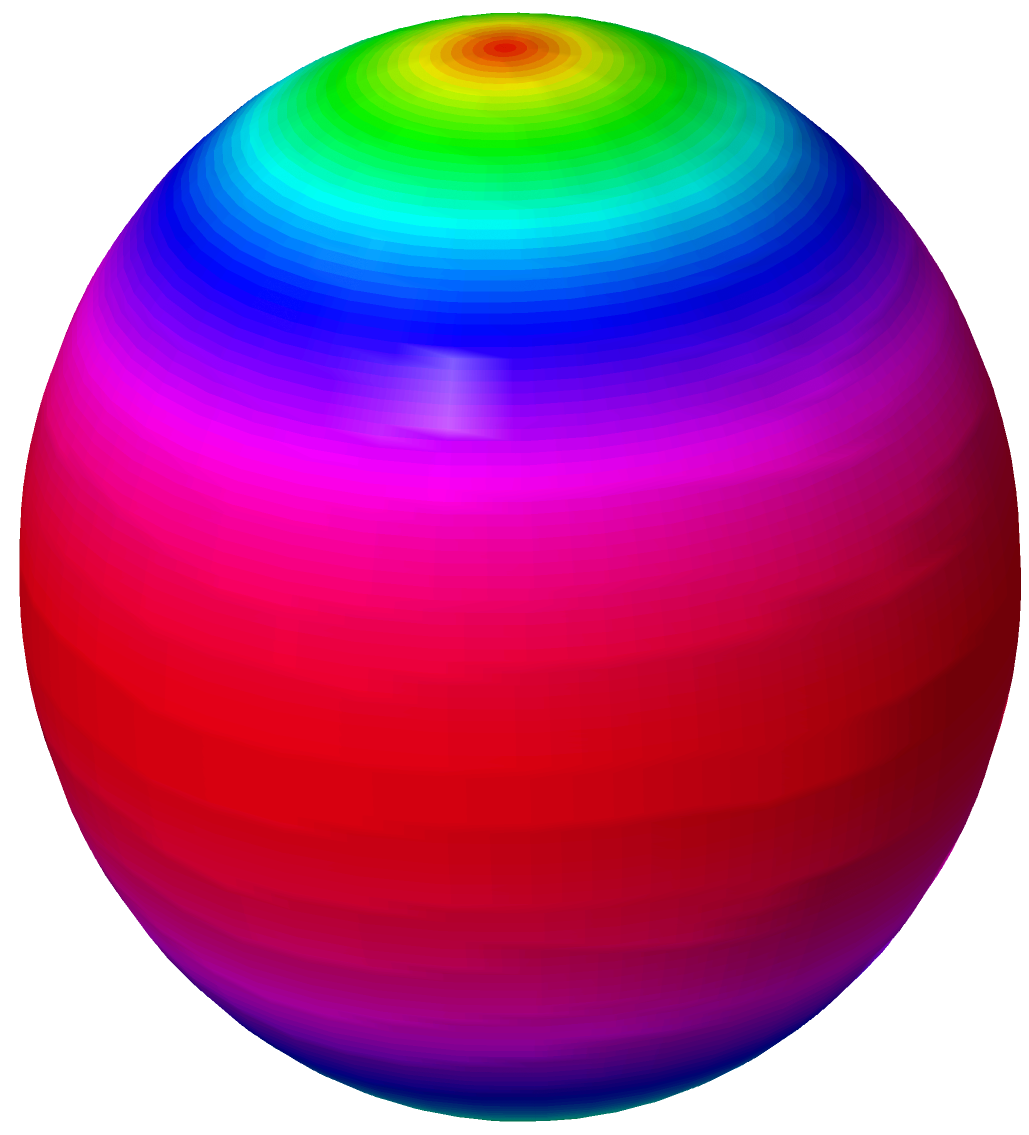}
}
\parbox{.22\textwidth}{
\vspace{-4mm}
\begin{overpic}[scale=0.1, trim=-30mm 0 0 -6mm]{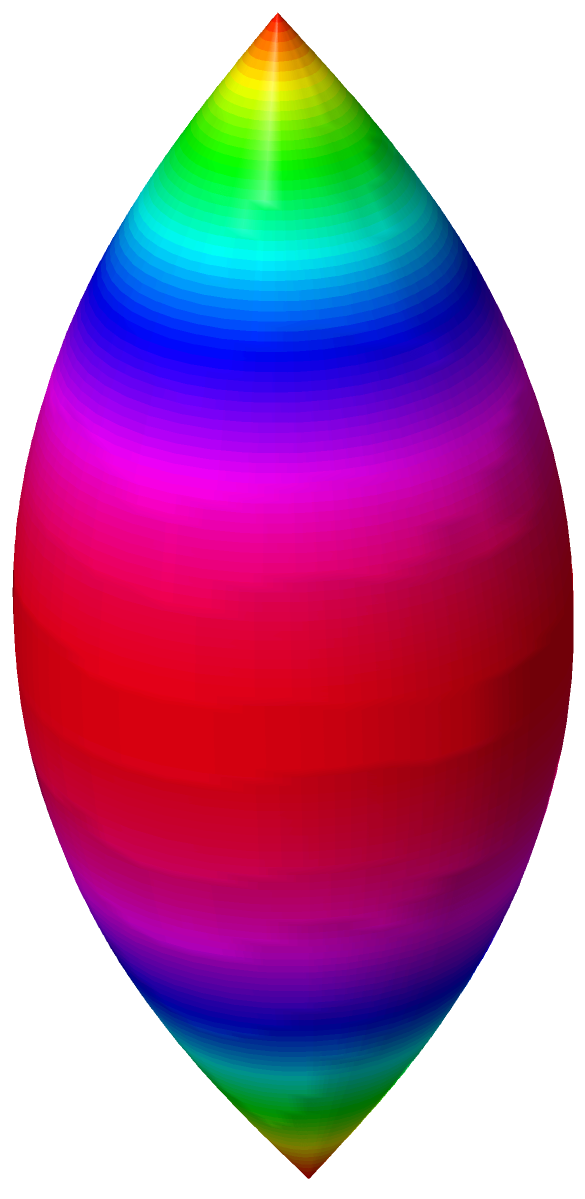}
\put(23,-18){$s=1$}
\end{overpic}

\vspace{10mm}
\includegraphics[scale=0.1]{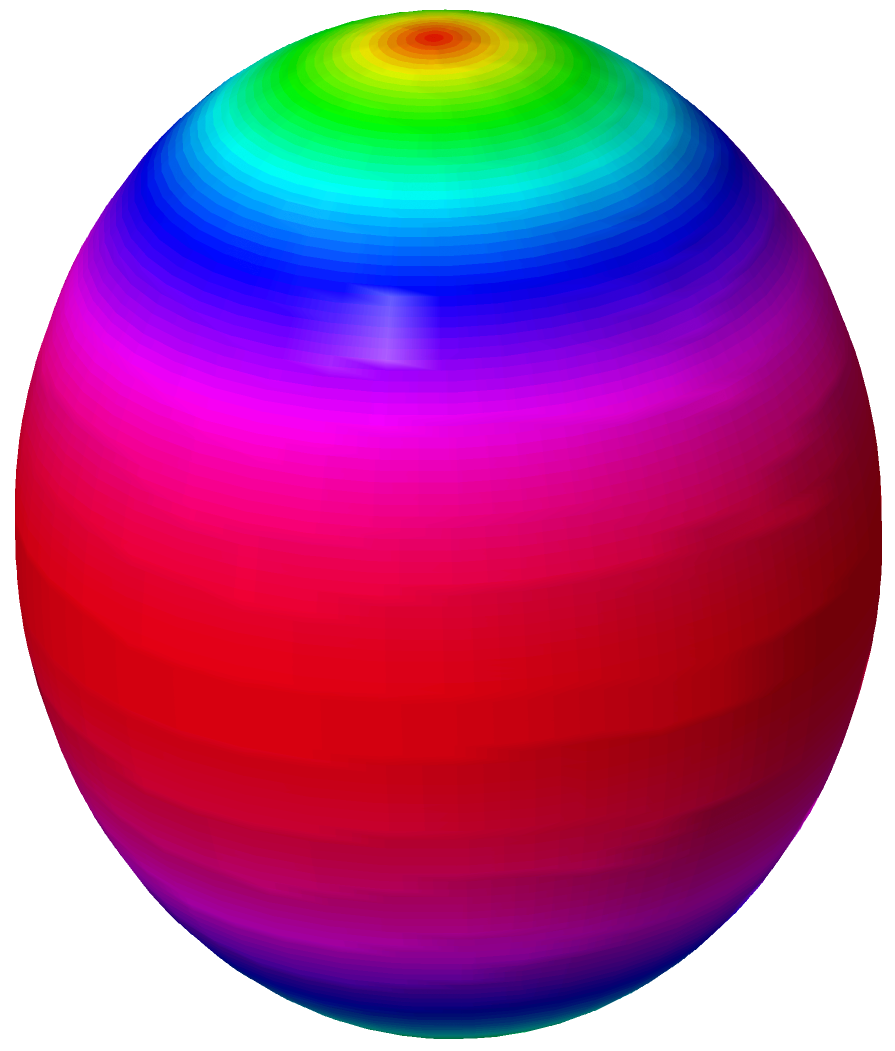}
}
\parbox{.22\textwidth}{
\vspace{-5mm}
\begin{overpic}[scale=0.1, trim=-30mm 0 0 0]{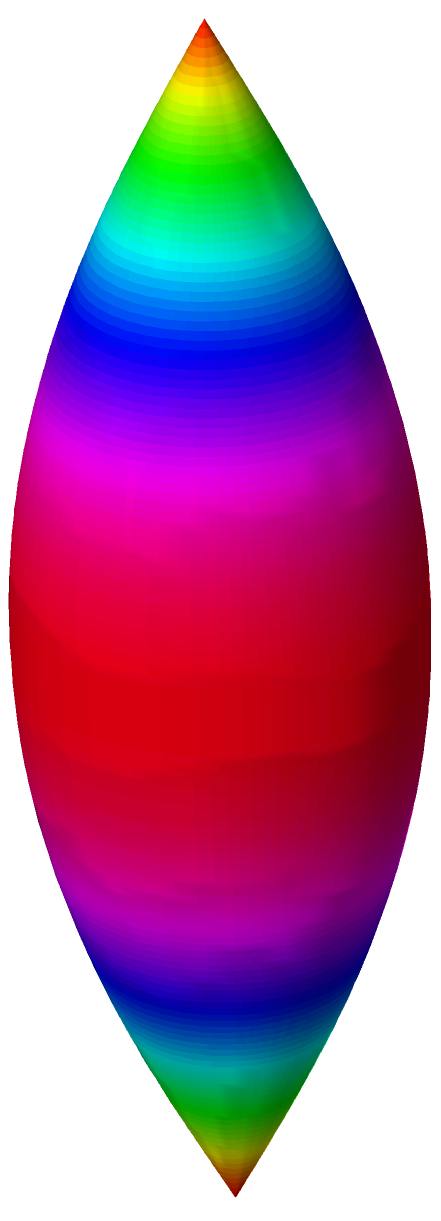}
\put(15,-16){$s=1.5$}
\end{overpic}

\vspace{10mm}
\includegraphics[scale=0.1]{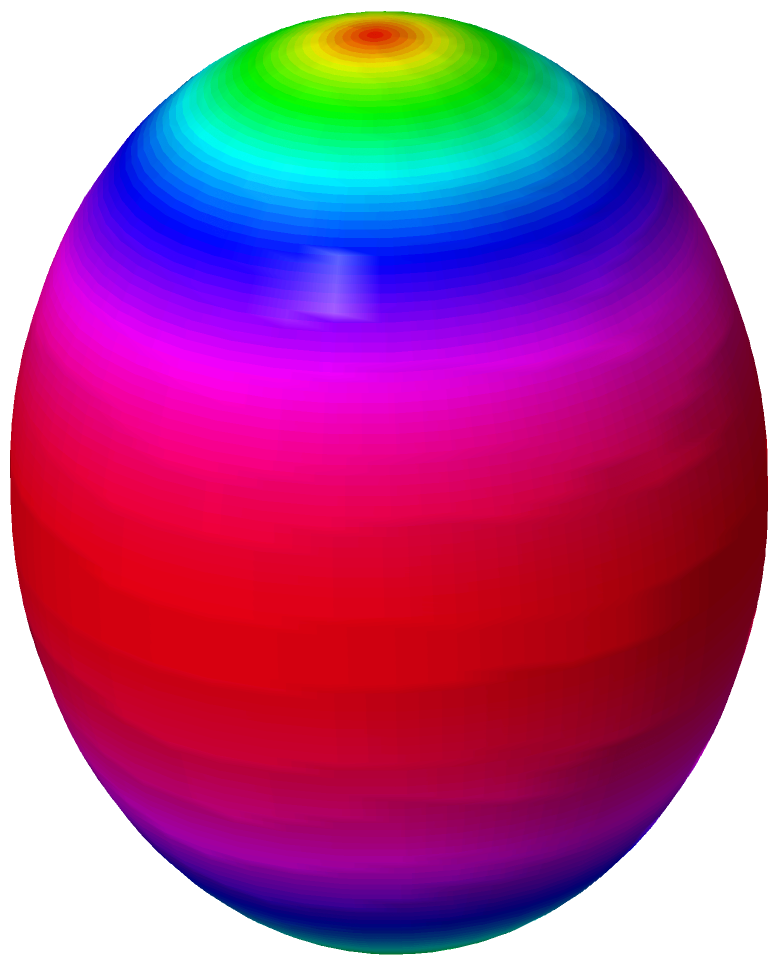}
}
\caption{The path $\mathcal{P}_\eps$ for $M=S^2$ with $\eps=0$ (top) and $\eps=1$ (bottom).}
\label{fig.deform}
\end{figure}

One observes that for $\eps=0$ the metric instantaneously develops singularities at the fixed points while it remains smooth for $\eps>0$.

\begin{addendum} \label{add_Gamma} 
Assume the setting of Lemma~\ref{lem:ODE}.
Let $\Phi \colon M \to M$ be a diffeomorphism satisfying $\Phi^*(\X) =\X$ or $\Phi^*(\X) = - \X$ and let $\mathcal{P}_{\eps}$ be a solution of \eqref{system}.
Then the path $s \mapsto \Phi^*(\mathcal{P}_{\eps}(s))$  is a solution of   \eqref{system} with initial condition $\Phi^*(g)$.

This implies the following:
Suppose $\Gamma$ is a compact Lie group containing $S^1$ as a normal subgroup such that  the given $S^1$-action on $M$ extends to a smooth $\Gamma$-action.
In particular,  for each $\gamma \in \Gamma$, we have $\gamma^*(\X) = \X$ or $\gamma^*(\X) = - \X$.
Suppose that $g$ is $\Gamma$-invariant.
Then each $\mathcal{P}_{\eps}(s)$ is $\Gamma$-invariant.
\end{addendum}

\begin{remark} \label{explicitsol} Assume that the $S^1$-action on $M$ is fixed-point free and that $\eps = 0$.
Let $\mathscr{V}= \mathrm{span} \, \X \subset TM$, let $\mathscr{H} :=\mathscr{V}^{\perp_g} \subset TM$ and let $g = g_{\mathscr{V}} + g_{\mathscr{H}}$ be the resulting splitting of $g$.
Then  the path $\mathcal{P}_0$ is given by
\[
    \mathcal{P}_0(s) = \exp(-s) g_\mathscr{V} + g_{\mathscr{H}} . 
 \]
If the $S^1$-action is free, this coincides with the canonical variation $g_{\exp(-s)}$ for the  Riemannian submersion $(M,g) \to (M/S^1, \check g)$, compare Remark~\ref{reinterpret}.
\end{remark}

The following result will be important for the proof of Theorem~\ref{mainB}.

\begin{proposition} \label{maindeform}  
Let $\Gamma$ be a compact Lie group containing $S^1$ as a normal subgroup.
Let $M$ be a compact smooth $\Gamma$-manifold, possibly with boundary.
Let $U \subset M$ be a $\Gamma$-invariant open subset such that the action restricted to $S^1$ is fixed-point free on $\bar U$.
Let $K$ be a compact Hausdorff space and let $g \colon K \to \mathscr{R}^{\Gamma}_{>0}(M)$ be a continuous map.

Then for each $s_0 > 0$ there is a continuous map 
 \[
    \mathcal{P} \colon K \to C^{\infty}([0,s_0], \mathscr{R}^{\Gamma}_{>0}(M))
 \]
such that, for $s \in [0,s_0]$ and $\xi \in K$, we have
 \[
      \mathcal{P}(\xi)(s)|_U =\big(\exp(-s) \, g(\xi)_{\mathscr{V}} + g(\xi)_{\mathscr{H}} \big)|_U .
 \]
 Here, the $\Gamma$-invariant subbundles $\mathscr{V}$ and $\mathscr{H}$ of $TU$ are defined as in Remark~\ref{explicitsol} with respect to $g(\xi)$ and the given $S^1$-action.
 \end{proposition}

\begin{proof} 
Since the $S^1$-action on $\bar U$ is fixed-point free, there exists a smooth $\Gamma$-invariant cut-off function $\chi \colon M \to [0,1]$ which is $0$ on a $\Gamma$-invariant open neighborhood $V$ of $M^{S^1}$ in $M$ and equal to $1$ on $U$. 
Let $\eps > 0$ and $s_0 > 0$.
For $\xi \in K$ consider the paths
\begin{enumerate}[\myicon]
     \item $\mathcal{P}(\xi)_{\eps} \colon [0,s_0] \to \mathscr{R}^{\Gamma}_{>0}(M)$, 
     \item $\mathcal{P}(\xi)_0 \colon [0,s_0] \to \mathscr{R}^{\Gamma}_{>0}(M \setminus M^{S^1})$
\end{enumerate} 
 constructed in Lemma~\ref{lem:ODE} in combination with Addendum~\ref{add_Gamma} with initial values $g(\xi)$.
 These paths depend continuously on $\xi$.

 Now, for $s \geq 0$ we set
 \[
     \mathcal{P}(\xi)(s):= (1-\chi) \cdot \mathcal{P}(\xi)_{\eps}(s) + \chi \cdot \mathcal{P}(\xi)_0(s) \in \mathscr{R}^{\Gamma}(M) . 
  \]
 Then $\mathcal{P}(\xi)(s)|_V = \mathcal{P}(\xi)_{\eps}(s)|_V$  for all $s$ and $\xi$.
 On the compact subset $M \setminus V \subset M$, we have
 \[
    \lim_{\eps \to 0} \mathcal{P}(\xi)_{\eps} = \mathcal{P}(\xi)_0
 \]
 in $C^0([0,s_0], \mathscr{R}^{\Gamma}(M\setminus V))$ uniformly in $\xi$.
Hence, on $M \setminus V$, we have
 \[
     \lim_{\eps \to 0} \mathcal{P}(\xi)  =  \mathcal{P}(\xi)_0
 \]
in $C^0([0,s_0], \mathscr{R}^{\Gamma}(M \setminus V))$ uniformly in $\xi$.
Since the paths $\mathcal{P}(\xi)_0$ consist of positive scalar curvature metrics, choosing $\eps$ small enough, each $\mathcal{P}(\xi)(s)|_{M \setminus V}$ has  positive scalar curvature.
 
The formula for $\mathcal{P}(\xi)(s)|_U$ follows from Remark~\ref{explicitsol}.
\end{proof}

\section{Uniformization of invariant tubular neighborhoods} 
\label{sec.uniformization}

Let $\Gamma$ be a compact Lie group,  let $M$ be a smooth $\Gamma$-manifold without boundary and let $S \subset M$ be a compact $\Gamma$-invariant smooth submanifold without boundary.
Let $ \RR^\Gamma(M) $ be the space of $\Gamma$-invariant Riemannian metrics on $M$, equipped with the weak $C^{\infty}$-topology.
Furthermore, let $\Diff^\Gamma(M,S)$ be the space of $\Gamma$-equivariant diffeomorphisms of $M$ which act as the identity on $S$, equipped with the weak $C^{\infty}$-topology.

Let $K$ be a compact Hausdorff space and let 
\[
  g \colon K \to \RR^\Gamma(M) 
\]
be a continuous family of $\Gamma$-invariant Riemannian metrics on $M$. 
Let 
\[
    \nu = (TM|_S) / TS \to S 
\]
be the normal bundle of $S$ in $M$. 
It is equipped with an induced $\Gamma$-action.
With this definition, $\nu$ is independent of any choice of Riemannian metric on $M$.

Every Riemannian metric $g(\xi)$ identifies $\nu$ with the geometric normal bundle $\nu^{g(\xi)}\subset TM|_S$ and induces a fiberwise scalar product on $\nu$, see Example~\ref{ex.normalbundle}.
Furthermore, the normal exponential map for $g(\xi)$ along $S$ maps some $\Gamma$-invariant open neighborhood of the zero section of $\nu$ diffeomorphically and $\Gamma$-equivariantly onto a $\Gamma$-invariant open neighborhood of $S$ in $M$.
By the uniqueness of tubular neighborhoods, see \cite{Hirsch94}*{Chapter~4, Theorem~5.3}, any two such embeddings for $g(\xi_1)$ and $g(\xi_2)$ are smoothly isotopic.
We will formulate and prove this result in an equivariant and family version.
Our exposition, which is based on parallel transport with respect to generalized cylinder metrics, is somewhat different from the one used in \cite{Hirsch94}.

\begin{proposition} \label{uniformnormal}
Let $g_\mathrm{ref}$  be a $\Gamma$-invariant Riemannian metric on $M$.
Let $\mathscr{W} \subset M$ be a $\Gamma$-invariant open neighborhood of $S$.
Then there exist open $\Gamma$-invariant neighborhoods  $U  \subset  \nu$ of the zero-section and $V \subset \mathscr{W}$ of $S$ and a continuous map 
\[
     \Phi\colon K  \to  C^{\infty}\big( [0,1], \mathrm{Diff}^{\Gamma}(M,S)\big)
\]
such that the normal exponential map along $S$ with respect to  $g_\mathrm{ref}$  induces a $\Gamma$-equivariant diffeomorphism $U \stackrel{\approx}{\longrightarrow} V$, and the following assertions hold for every $\xi \in K$:
\begin{enumerate}[(i)] 
\item $\Phi(\xi)(0) = \id_M$;
\item $\Phi(\xi)(s)|_{M \setminus \mathscr{W}}  = \id_{M \setminus \mathscr{W}}$ for all $s \in [0,1]$;
\item  the normal exponential map along $S$ with respect to $\Phi(\xi)(1)^*(g(\xi))$ induces a $\Gamma$-equivariant diffeomorphism $U \stackrel{\approx}{\longrightarrow} V$, and this  diffeomorphism coincides with the one induced by the normal exponential map with respect to $g_\mathrm{ref}$.
In particular, we get
\[
      \nu^{\Phi(\xi)(1)^*(g(\xi))} = \nu^{g_\mathrm{ref}} \subset TM|_S.
\]
\item The bundle metrics on $\nu$ induced by $\Phi(\xi)(1)^*(g(\xi))$ and by $g_\mathrm{ref}$ coincide.
\end{enumerate} 
\end{proposition} 

\begin{proof} 
Define the continuous map $\bar g \colon K \to C^{\infty}([0,1], \RR^{\Gamma}(M))$ by
\begin{equation} \label{straight_path}
    \bar g(\xi)(s) := (1-s)g_\mathrm{ref} + s g(\xi). 
\end{equation}
For $\xi \in K$ and $s \in [0,1]$, let $h(\xi,s)$  be the $\Gamma$-invariant vector bundle metric on $\nu$ induced by $\bar g(\xi)(s)$. 

Consider the $\Gamma$-invariant smooth generalized cylinder metric $ds^2 + \bar g(\xi)(s)$ on $[0,1] \times M$.
It induces a $\Gamma$-equivariant connection on the normal bundle of $[0,1] \times S$ in $[0,1] \times M$.
Parallel transport with respect to this connection induces a $\Gamma$-equivariant vector bundle isometry
\begin{equation} \label{bundle_isometry}
         \mathcal{P}(\xi)(s)  \colon  (\nu , h(\xi,0)) \to (\nu , h(\xi,s))
\end{equation}
for each $\xi \in K$ and $s \in [0,1]$.
Furthermore, $\mathcal{P}(\xi)(0) = \id_\nu$. 
Here we identify $\nu$ with the orthogonal complements of $TS \subset TM|_S$ with respect to $\bar g(\xi)(0)$ and $\bar g(\xi)(s)$, respectively.
The isometry $\mathcal{P}(\xi)(s)$ depends continuously on $\xi$ and smoothly on $s$.

Let $\mathscr{U} \subset \nu$ be an open $\Gamma$-invariant neighborhood of the zero-section. 
For $\xi \in K$ and $s \in [0,1]$, we obtain an open $\Gamma$-invariant neighborhood of the zero-section by
\[
     U_{\xi,s} := \mathcal{P}(\xi)(s)(\mathscr{U}) \subset \nu .
\]
By choosing  $\mathscr{U}$ small enough, we  can assume that for every $\xi \in K$ and $s \in [0,1]$, there exists an open $\Gamma$-invariant neighborhood $V_{\xi,s} \subset M$ of $S$ such that the normal exponential map of $\bar g (\xi)(s)$ along $S$ induces a $\Gamma$-equivariant diffeomorphism 
\[
     \exp^{\perp}_{\xi,s}  \colon U_{\xi,s} \xrightarrow{\approx} V_{\xi,s}.
\]
Note that $\exp^{\perp}_{\xi,1}$ is the normal exponential map for  $g(\xi)$ and that $\exp^{\perp}_{\xi,0}$ is the one for  $g_\mathrm{ref}$. 

We observe that $V_{\xi,0}$ does not depend on $\xi$. 
Put $\mathscr{V} := V_{\xi, 0}$.
For each $\xi \in K$, we obtain  a smooth $\Gamma$-equivariant isotopy $\varphi(\xi)(s)$, $s \in [0,1]$, such that the following diagram commutes:
\begin{equation} \label{compatible} 
\vcenter{\hbox{
\xymatrixcolsep{4pc}
\xymatrix{ 
       \mathscr{U}   \ar[r]^{ \exp^{\perp}_{\xi, 0} } \ar[d]_{\mathcal{P}(\xi)(s)} & \mathscr{V}  \ar[d]^{\varphi(\xi)(s)}  \\
       U_{\xi,s}   \ar[r]^{  \exp^{\perp}_{\xi,s}}    & V_{\xi,s}  
}}}
\end{equation} 
These isotopies satisfy $\varphi(\xi)(0)=\id_{\mathscr{V}}$, $\varphi(\xi)(s)|_S=\id_S$, and they depend continuously on $\xi$.

By  the isotopy extension theorem  \cite{Hirsch94}*{Chapter~8, Theorem~1.4}, which immediately generalizes to the $\Gamma$-equivariant and family situation (for the equivariant situation, see \cite{Kan07}*{Theorem 8.3}), there is a $\Gamma$-invariant open neighborhood $V \subset \mathscr{W}\cap \mathscr{V}\subset M$ of $S$ and  a continuous map 
\[
     \Phi \colon K \to C^{\infty}\big([0,1] , \mathrm{Diff}^{\Gamma}(M,S)\big) 
\]
 such that, for $\xi \in K$ and $s \in [0,1]$, 
\begin{enumerate}[\myicon] 
  \item $\Phi(\xi)(0) = \id_M$; 
  \item $\Phi(\xi)(s)|_{M \setminus \mathscr{W}}  = \id_{M \setminus \mathscr{W}}$; 
  \item $ \Phi(\xi)(s)|_{V}  =  \varphi(\xi)(s)|_{V}$.
 \end{enumerate} 
Now Proposition~\ref{uniformnormal} holds with this $V$ and with $U := (\exp^{\perp}_{g_\mathrm{ref}})^{-1}(V)$.
\end{proof}

For later use we formulate the following variant of Proposition~\ref{uniformnormal} for manifolds with boundary.
Let $\Gamma$ be a compact Lie group and  let $M$ be a smooth $\Gamma$-manifold with compact boundary $\partial M$. 
Set $S := \partial M$. 
Let $ \RR^\Gamma(M) $ be the space of $\Gamma$-invariant Riemannian metrics on $M$, equipped with the weak $C^{\infty}$-topology.
Furthermore, let $\Diff^\Gamma(M,S)$ be the space of $\Gamma$-equivariant diffeomorphisms of $M$ which act as the identity on $S$, equipped with the weak $C^{\infty}$-topology.

Let $K$ be a compact Hausdorff space and let 
\[
  g \colon K \to \RR^\Gamma(M) 
\]
be a continuous family of $\Gamma$-invariant Riemannian metrics on $M$. 
Let 
\[
    \nu = (TM|_S) / TS \to S 
\]
be the normal bundle of $S$ in $M$. 
Choose a $\Gamma$-equivariant vector bundle trivialization $\nu \cong S \times \R$, where, for each $x \in S$, the vector $(x,1)$ points into $M$.

For every Riemannian metric $g(\xi)$, the normal exponential map along $S$ maps some $\Gamma$-invariant open neighborhood of the zero section $S \times \{0\}$ in $S \times [0,\infty)$ diffeomorphically and $\Gamma$-equivariantly onto a $\Gamma$-invariant open neighborhood of $S$ in $M$.

\begin{proposition} \label{uniformnormal_bis}
Let $g_\mathrm{ref}$  be a $\Gamma$-invariant Riemannian metric on $M$.
Let $\mathscr{W} \subset M$ be a $\Gamma$-invariant open neighborhood of $S$.
Then there exist open $\Gamma$-invariant neighborhoods  $U  \subset  S \times [0,\infty)$ of  $S \times \{0\}$ and $V \subset \mathscr{W}$ of $S$ and a continuous map 
\[
     \Phi\colon K  \to  C^{\infty}\big( [0,1], \mathrm{Diff}^{\Gamma}(M,S)\big)
\]
such that the normal exponential map along $S$ with respect to  $g_\mathrm{ref}$  induces a $\Gamma$-equivariant diffeomorphism $U \stackrel{\approx}{\longrightarrow} V$, and the following assertions hold for every $\xi \in K$:
\begin{enumerate}[(i)] 
\item $\Phi(\xi)(0) = \id_M$;
\item $\Phi(\xi)(s)|_{M \setminus \mathscr{W}}  = \id_{M \setminus \mathscr{W}}$ for all $s \in [0,1]$;
\item  the normal exponential map along $S$ with respect to $\Phi(\xi)(1)^*(g(\xi))$ induces a $\Gamma$-equivariant diffeomorphism $U \stackrel{\approx}{\longrightarrow} V$, and this  diffeomorphism coincides with the one induced by the normal exponential map with respect to $g_\mathrm{ref}$.
\end{enumerate} 
\end{proposition} 

\begin{proof} 
The proof is the same as the one for Proposition~\ref{uniformnormal}, with the only difference being that $\mathscr{U} \subset S \times [0,\infty)$ is an open $\Gamma$-invariant neighborhood of the zero section $S \times \{0\}$ of $\nu$.
\end{proof}

\section{Equivariant smoothing of mean-convex singularities} 
\label{sec.smoothing}

Let $\Gamma$ be a compact Lie group, let $\hat M$ be a smooth $\Gamma$-manifold without boundary and let $\Sigma \subset \hat M$ be a compact $\Gamma$-invariant smooth hypersurface without boundary. 
Assume that $\hat M = M_1 \cup_\Sigma M_2$ where $M_1$ and $M_2$ are smooth $\Gamma$-manifolds with compact boundaries $\Sigma$.
Consider the disjoint union 
\[
   \mathcal{M} = M_1 \sqcup M_2. 
\]
This is a smooth $\Gamma$-manifold with compact boundary $\Sigma \sqcup \Sigma$. 
If $g_j$ is a $\Gamma$-invariant Riemannian metric on $M_j$, $j = 1,2$, we denote by $(g_j)_0$ the induced metric on $\partial M_j = \Sigma$.

Let $H_{g_j} \colon \partial M_j \to \R$ be the unnormalized mean curvature of $g_j$ with respect to the exterior unit normal.
Recall that with this convention, which we already used in Section~\ref{sec.scalmeansub}, the mean curvature of the closed unit ball in $\R^n$ is equal to $n-1$.

Consider
\begin{equation} \label{def_sing}
      \RR_{> 0}^{\Gamma, \Sigma}(\hat M)  := \{ g_1 \sqcup  g_2   \in \RR^{\Gamma}_{> 0} (\mathcal{M}) \mid (g_1)_0 = (g_2)_0 ,\, H_{g_1} + H_{g_2} \geq 0\} , 
\end{equation}
the space of $\Gamma$-invariant Riemannian metrics on $\hat M$ of positive scalar curvature  and  ``mean convex singularity'' along $\Sigma$. 
This space is equipped with the subspace topology of $\RR^{\Gamma}_{>0}(\mathcal{M})$. 

Each smooth $g \in \RR^{\Gamma}_{> 0}(\hat M)$ induces $ g_1 \sqcup g_2 \in \RR_{> 0}^{\Gamma, \Sigma}(\hat M)$ where $g_j = g|_{M_j}$. 
In \cite{BaerHanke2}*{Theorem~4.11} we proved that for $\Gamma=\{1\}$, a converse holds up to homotopy.
In Proposition~\ref{desingularise} below, we obtain an equivariant analog.

Let  $g_1 \sqcup g_2 \in \RR^{\Gamma}(\mathcal{M})$. 
For sufficiently small  $\eps > 0$, the normal exponential map along $\partial \mathcal{M}$ with respect to $g_1 \sqcup g_2$ induces a $\Gamma$-equivariant diffeomorphism 
\begin{equation} \label{collar}
     [0,\eps) \times \partial \mathcal{M} = \big( [0,\eps) \times  \partial M_1 \big) \, \sqcup \, \big( [0,\eps) \times \partial M_2 \big) \stackrel{\approx}{\longrightarrow} U^{g_1}_{\eps}  \sqcup U^{g_2}_{\eps} 
\end{equation}
where $U_\eps^{g_j}$ is the open $\Gamma$-invariant $\eps$-neighborhood of $\partial M_j$ in $M_j$ with respect to $g_j$.
Note that \eqref{collar} induces a smooth structure on the open $\Gamma$-invariant neighborhood $U^{g_1}_\eps \cup U^{g_2}_{\eps} \subset \hat M$ of $\Sigma$.
However, this does not necessarily coincide with the given smooth structure on $\hat M$.

Let $t$ be the canonical coordinate on $[0,\eps)$. 
By the generalized Gauss lemma and using the diffeomorphism \eqref{collar}, each metric $g_j$ takes the form
\begin{equation} \label{Gauss}
     g_j = dt^2 + (g_j)_t
\end{equation}
on $U^{g_j}_\eps$ where $(g_j)_t$, $t \in [0,\eps)$, is a smooth family of $\Gamma$-invariant  metrics on $\partial M_j$.
For $\ell \geq 0$, we define smooth symmetric $\Gamma$-invariant $(0,2)$-tensor fields on $\partial M_j$ by 
\[
      (g_j)_0^{(\ell)}(X,Y) := \tfrac{d^\ell}{dt^\ell}\big|_{ t = 0} ( (g_j)_t(X,Y)). 
\]
In particular, $(g_j)_0^{(0)} = (g_j)_0$ is the metric on $\partial M_j$ induced by $g_j$ and $(g_j)_0^{(1)} = (\dot g_j)_0$ is the $t$-derivative of $(g_j)_t$ at $t=0$.

\begin{lemma} \label{charglatt} Let $g_1 \sqcup g_2 \in \RR^{\Gamma}(\mathcal{M})$. 
Then $g_1 \sqcup g_2$ defines a smooth metric on $\hat M$ if and only if the following two properties are satisfied:
\begin{enumerate}[(i)]
\item \label{erstens} 
The smooth structure on $U^{g_1}_\eps \cup U^{g_2}_\eps \subset \hat M$ induced by the collars~\eqref{collar} coincides with the one induced by the smooth structure on $\hat M$.
  \item \label{zweitens} We have  $(g_1)_0^{(\ell)} = (-1)^{\ell} (g_2)_0^{(\ell)}$ for all $\ell \geq 0$.
 \end{enumerate}
 \end{lemma}
 
\begin{proof} 
Clearly, if $g_1 \sqcup g_2$ defines a smooth metric on $\hat M$, then \ref{erstens} and \ref{zweitens} hold.
Conversely, if \ref{erstens} holds, then 
\[
U^{g_1}_{\eps}  \sqcup U^{g_2}_{\eps} \stackrel{\eqref{collar}}{\approx} [0,\eps)  \times \partial \mathcal{M}  \to (-\eps, \eps) \times \Sigma, 
\quad 
(t,x) \mapsto 
\begin{cases} 
  (-t,x) & \textrm{ for } x \in \partial M_1 , \\
  (+ t,x) & \textrm{ for } x \in \partial M_2,
\end{cases}
\]
induces a diffeomorphism $U^{g_1}_{\eps}  \cup U^{g_2}_{\eps} \xrightarrow{\approx}  (-\eps, \eps) \times \Sigma$ on the open neighborhood $U^{g_1}_{\eps}  \cup U^{g_2}_{\eps}$ of $\Sigma$  in $ \hat M$.
Hence, if both \ref{erstens} and  \ref{zweitens} are satisfied, then $j^{k} g_1|_\Sigma = j^{k} g_2|_\Sigma$ for all $k$-jets, $k \geq 0$. 
That is, $g_1 \sqcup g_2$ defines a smooth metric on $\hat M$.
\end{proof}

\begin{proposition} \label{desingularise} 
The inclusion $\RR^{\Gamma}_{> 0} (\hat M) \hookrightarrow \RR_{> 0}^{\Gamma, \Sigma}(\hat M)$ is a weak homotopy equivalence. 
\end{proposition} 

\begin{proof} 
Let ${m } \geq 0$ and 
\begin{equation} \label{start_g}
    g = g_1 \sqcup g_2  \colon  D^{m }  \to \RR_{> 0}^{\Gamma, \Sigma}(\hat M)
\end{equation}
be continuous such that $g( \partial D^{m } ) \subset \RR^{\Gamma}_{> 0} (\hat M)$.  
For $\xi \in D^m$ and $j = 1,2$ let 
\[
      \II_{g_j(\xi)} = - \tfrac{1}{2} (\dot g_j)_0 
\]
be the second fundamental form of $g_j(\xi)$ with respect to the exterior normal $-\frac{\partial}{\partial t}$ along $\partial M_j$.
By definition, we have $H_{g_j(\xi)} =  \tr_{g_j(\xi)_0} \II_{g_j(\xi)}$, where $(g_1(\xi))_0 = (g_2(\xi))_0$ by \eqref{def_sing}.

We apply \cite{BaerHanke2}*{Theorem~3.7}\footnote{In that reference, the normalized mean curvature is used.} to $\mathcal{M}$ with the symmetric tensor field 
\[
  k(\xi) :=  - \II_{g_2(\xi)} \sqcup \II_{g_2(\xi)}\in C^{\infty} ( T^* \partial \mathcal{M}  \otimes T^* \partial \mathcal{M} )
\]
and $\sigma = 0$.
Note that \cite{BaerHanke2}*{Theorem~3.7} applies since
\[
 \tr_{g(\xi)_0}(k) =   \tr_{g_1(\xi)_0} ( -  \II_{g_2(\xi)} ) +  \tr_{g_2(\xi)_0} (  \II_{g_2(\xi)} ) = 0 \leq H_{g_1(\xi)} + H_{g_2(\xi)}.
\]
In our situation, the proof of  \cite{BaerHanke2}*{Theorem~3.7} yields a map 
\begin{equation}\label{desing}
     f = f_1 \sqcup f_2 \colon  D^{m}  \times [0,1] \to \RR_{> 0}^{ \Gamma, \Sigma}(\hat M) 
\end{equation}
satisfying $f_j(\xi)(0) = g_j(\xi)$.
To justify the $\Gamma$-equivariance of $f(\xi,s)$, which is not discussed in \cite{BaerHanke2}, in the proof of \cite{BaerHanke2}*{Proposition~3.3}, instead of citing the local flexibility lemma \cite{BaerHanke1}*{Addendum~3.4}, one can use the $\Gamma$-equivariant local flexibility lemma, Theorem~\ref{thm:FlexLem1}.
The other parts of the construction of $f$ in \cite{BaerHanke2} remain unchanged.

Using parts (c), (d), (f), and (g) of \cite{BaerHanke2}*{Theorem~3.7}, we get
\begin{enumerate}[(I)]
   \item \label{aa} $f(\xi,s) \in \RR_{> 0}^{\Gamma}(\hat M)$ for $(\xi,s)  \in \partial D^m \times [0,1]$;
   \item \label{bb} $f_1(\xi,1)^{(\ell)}_0 = (-1)^{\ell} f_2 (\xi,1)_0^{(\ell)}$ for $\xi \in D^m$ and $\ell \geq 0$.
\end{enumerate}

After applying the homotopy $f$ in \eqref{desing}, which starts at $g$ and satisfies \ref{aa}, we obtain a new map $g \colon  D^{m }  \to \RR_{> 0}^{\Gamma, \Sigma}(\hat M)$,
\[
   g(\xi) :=g_1(\xi) \sqcup g_2(\xi) \in  \mathscr{R}_{>0}^{\Gamma, \Sigma} ( \hat M), \qquad g_j(\xi) := f_j(\xi,1), \quad j =1,2.
\]  
Using \ref{bb}, the metric $g_1(\xi) \sqcup g_2(\xi)$ satisfies property~\ref{zweitens} of Lemma~\ref{charglatt} for each $\xi \in D^m$.
Through a further homotopy, we will ensure that $g_1(\xi) \sqcup g_2(\xi)$ also satisfies property~\ref{erstens} of Lemma~\ref{charglatt}.

For this aim, we apply Proposition~\ref{uniformnormal_bis} for $M:= \mathcal{M}$, for some $\Gamma$-invariant  Riemannian metric $g_\mathrm{ref} \in \RR^{\Gamma}(\mathcal{M})$ which is induced by a $\Gamma$-invariant smooth Riemannian metric on $\hat M$, for  the family $g \colon D^m \to \RR^{\Gamma, \Sigma}_{>0}(\hat M)$, for $\mathscr{W} := \mathcal{M}$  and for $S := \partial \mathcal{M} = \Sigma \sqcup \Sigma$.
Let  $\Phi \colon D^m \to C^{\infty}([0,1], \mathrm{Diff}^{\Gamma}(\mathcal{M},S))$ be the resulting map. 
It has the following properties: 
\begin{enumerate}[(A)]
   \item for all $\xi \in D^m$, we have $\Phi(\xi)(0) = \id_{\mathcal{M}}$;
   \item \label{Be} for small enough $\eps$, the $\Gamma$-equivariant diffeomorphisms~\eqref{collar} with respect to $(\Phi(\xi)(1))^*(g(\xi))$ are independent of $\xi \in D^m$.
   In particular, all of them define the same smooth structure on $\hat M$, and this coincides with the given smooth structure on $\hat M$.
\end{enumerate}

Let $\xi \in \partial D^m$.
Then $g(\xi) \in \RR^{\Gamma}(\hat M)$, and therefore, for all $s \in [0,1]$, we have $\bar g(\xi)(s) = (1-s)g_\mathrm{ref} + s g(\xi) \in \RR^{\Gamma}(\hat M)$ in \eqref{straight_path}.
Now recall that around $S = \partial \mathcal{M} = \Sigma \cup \Sigma \subset \mathcal{M}$, the diffeomorphisms  $\Phi(\xi)(s) \in \Diff^{\Gamma}(\mathcal{M},S)$ constructed in the proof of Proposition~\ref{uniformnormal_bis} are induced by the vector bundle isometries from equation \eqref{bundle_isometry} and the normal exponential maps along $\Sigma \subset \hat M$ for the metrics $g_\mathrm{ref}$ and $\bar g(\xi)(s)$, see diagram \eqref{compatible} and the preceding remarks. 
Since these vector bundle isometries and metrics are smooth (with respect to the given smooth structure on $\hat M$), the two diffeomorphisms defined by $\Phi(\xi)(s)$ on the two sides of $\Sigma \subset \hat M$ glue to a smooth diffeomorphism  in $\Diff^{\Gamma}(\hat M, \Sigma)$.
Combining this with property \ref{aa}, we conclude that for $\xi \in \partial D^m$ and $s \in [0,1]$, we have $(\Phi(\xi)(s))^*(g(\xi)) \in \RR_{> 0}^{\Gamma}(\hat M)$.

Now let $\xi \in D^m$.
Recall that the metric $g(\xi) = g_1(\xi) \sqcup g_2(\xi) \in \RR_{>0}^{\Gamma, \Sigma}(\hat M)$ satisfies property~\ref{zweitens} of Lemma~\ref{charglatt}. 
We claim that for all $s \in [0,1]$, we have that
\[
    \Phi(\xi)(s)^*(g_1(\xi))\,  \sqcup \, \Phi(\xi)(s)^*(g_2(\xi)) \in \RR^{\Gamma, \Sigma}_{>0}(\hat M) 
\]
and that this metric satisfies property~\ref{zweitens} of Lemma~\ref{charglatt}.
The first assertion holds because the defining properties   in \eqref{def_sing} are invariant under the pull back of metrics along diffeomorphisms of $\mathcal{M}$ fixing $S$.

In order to verify property~\ref{zweitens} of Lemma~\ref{charglatt}, let $\eps > 0$ such that the normal exponential map along $\partial \mathcal{M}$ with respect to $g_1(\xi) \sqcup g_2(\xi)$ induces a $\Gamma$-equivariant diffeomorphism
\[
\Psi \colon   [0,\eps) \times \partial \mathcal{M} \stackrel{\approx}{\rightarrow} U^{g_1(\xi)}_\eps \sqcup U^{g_2(\xi)}_\eps
\]
as in \eqref{collar}.
Let $s \in [0,1]$. 
Then the normal exponential map along $\partial \mathcal{M}$ with respect to $ \Phi(\xi)(s)^*(g_1(\xi))\,  \sqcup \, \Phi(\xi)(s)^*(g_2(\xi))$ induces a $\Gamma$-equivariant diffeomorphism
\[
   \Psi_s \colon  [0,\eps) \times \partial \mathcal{M}  \stackrel{\approx}{\longrightarrow} U^{\Phi(\xi)(s)^*(g_1(\xi))}_{\eps} \,  \sqcup \, U^{\Phi(\xi)(s)^*(g_2(\xi))}_{\eps} .
\]
We have $\Psi_0 = \Psi$.
Furthermore, since $\Phi(\xi)(s)$ induces an isometry 
\begin{align*}
   \big( U^{\Phi(\xi)(s)^*(g_1(\xi))}_{\eps} &\,   \sqcup \,   U^{\Phi(\xi)(s)^*(g_2(\xi))}_{\eps} ,  \Phi(\xi)(s)^*(g_1(\xi))\,   \sqcup \, \Phi(\xi)(s)^*(g_2(\xi)) \big) \\
                & \cong \big( U^{g_1(\xi)}_\eps \sqcup U^{g_2(\xi)}_\eps , g_1(\xi) \sqcup g_2(\xi) \big),
\end{align*}
we have
\begin{equation} \label{covariance}
        \Phi(\xi)(s) \circ \Psi_s = \Psi.
\end{equation}
As in \eqref{Gauss}, we write, using the diffeomorphisms $\Psi$ and $\Psi_s$,
\[
   g_j (\xi)= dt^2 + \big(g_j(\xi)\big)_t , \qquad  \Phi(\xi)(s)^*(g_j(\xi)) = dt^2 +  \big(\Phi(\xi)(s)^*(g_j(\xi))\big)_t \, , 
 \]
 where $\big(g_j(\xi)\big)_t$ and  $(\big(\Phi(\xi)(s)^*(g_j(\xi))\big)_t$, $t \in [0, \eps)$, are smooth families of $\Gamma$-invariant metrics on $\partial M_j$.
 Using \eqref{covariance}, we obtain
 \[
    (g_j)_t = \big(\Phi(\xi)(s)^*(g_j(\xi))\big)_t \quad \textrm{for } t \in [0,\eps). 
 \]
Consequently the parity relation  required in property \ref{zweitens} of Lemma~\ref{charglatt} is indeed satisfied for $\Phi(\xi)(s)^*(g_1(\xi))\,  \sqcup \, \Phi(\xi)(s)^*(g_2(\xi))$. 

Now set $\mathcal{P} \colon D^m \times [0,1] \to \RR^{\Gamma, \Sigma}_{>0}(\hat M)$, $\mathcal{P}(\xi, s) = (\Phi(\xi)(s))^*(g(\xi))$. 
Then $\mathcal{P}(\xi, 0) = g(\xi)$ for all $\xi \in D^m$ and $\mathcal{P}(\xi,s) \in \RR^{\Gamma}_{>0}(\hat M)$ for all $(\xi,s) \in \partial D^m \times [0,1]$.
Furthermore for all $\xi \in D^m$, the metric $\mathcal{P}(\xi,1)$ satisfies properties \ref{erstens} and \ref{zweitens} of Lemma \ref{charglatt}, where property \ref{erstens} is ensured by property \ref{Be} above and property \ref{zweitens} was discussed in the previous paragraph.
In particular, $\mathcal{P}(\xi,1) \in \RR^{\Gamma}_{>0}(\hat M)$.

The concatenation of $\mathcal{P}$ with the homotopy $f$ defines a homotopy $D^m \times [0,1] \to \mathbb{R}_{>0}^{\Gamma, \Sigma}(\hat M)$ starting at the map $g$ in \eqref{start_g}  and sending $(\partial D^m \times [0,1]) \cup (D^m \times \{1\})$ to $\mathbb{R}^{\Gamma}_{>0}(\hat M)$.
This concludes the proof of Proposition~\ref{desingularise}.
\end{proof}
 
\section{Proof of Theorem~\ref{mainB}}
\label{sec.proofmainB}

In this section we carry out the proof of Theorem~\ref{mainB}.
Let $\Gamma$ be a compact Lie group and let $M$ be a closed connected smooth $\Gamma$-manifold.
Suppose that $\Gamma$ contains a normal subgroup $H$ isomorphic to $S^1$ such that $M^H$ contains at least one fixed-point component of codimension $2$.
In particular, we have $\dim M \geq 2$.
In the following, we identify $H$ with $S^1$ by a fixed isomorphism.

Let $\Stab <\Gamma$ be the kernel of the given $\Gamma$-action on $M$, that is, the subgroup of elements acting trivially on $M$.
This is a closed and normal subgroup of $\Gamma$.
Now $S^1 / (\Stab  \cap S^1)$ acts effectively on $M$, it is a normal subgroup of $\Gamma/ \Stab $, and $M^{S^1} = M^{S^1/( \Stab  \cap S^1)}$ contains components of codimension $2$.
Thus, by replacing  $S^1$ with $S^1/(\Stab   \cap S^1)$ and $\Gamma$ with $\Gamma/\Stab $, we can assume without loss of generality that $S^1$ acts effectively on $M$. 

According to Remark~\ref{weaklyetc}, it suffices to show that $\mathscr{R}^{\Gamma}_{>0}(M)$ is weakly contractible.
To do this, let $m \geq 0$ and let
\[
   g \colon \partial D^m  \to   \RR^{\Gamma}_{>0}(M) 
\]
be a continuous map. 
This assumption is empty for $m = 0$.
We  have  to show that $g$ extends to a continuous map 
\[
   G \colon D^m \to \RR^{\Gamma}_{>0}(M). 
\]
For $m = 0$, this means that $\RR^{\Gamma}_{>0}(M)$ is non-empty.
Let $S \subset M^{S^1}$ be the union of fixed-point components of codimension $2$.
This is a $\Gamma$-invariant submanifold of $M$.
Let 
 \[
    \nu = (TM|_S) / TS \to S
 \]
be the (abstract) normal bundle of $S$.
The group $\Gamma$ acts by vector bundle automorphisms on $\nu$ and the subgroup $S^1$ by fiber-preserving vector bundle automorphisms.
Since the $S^1$-action on $\nu$ is non-trivial, it induces an orientation on $\nu$ by declaring the pairs of vectors $\big(v,\frac{d}{dt}|_{t=0}(e^{2\pi it}\cdot v)\big)$ to be positively oriented for $v\in\nu$, $v\ne0$.
Note that, in general, the induced action of $\Gamma$ on $\nu$ need not be orientation-preserving.

Let $g_\mathrm{ref}$ be a $\Gamma$-invariant reference Riemannian metric on $M$ and let
\[
      \Phi \colon \partial D^m \to C^{\infty}\big([0,1], \mathrm{Diff}^{\Gamma}(M,S)\big) 
\]
be the map constructed in Proposition~\ref{uniformnormal} for $g_\mathrm{ref}$ and for $\mathscr{W} := M$.
Define 
\[
\mathcal{P} \colon \partial D^m \to C^{\infty}\big([0,1], \mathscr{R}^{\Gamma}_{>0}(M)\big)
\]
 by
\[
       \mathcal{P}(\xi)(t) :=   \Phi (\xi)(t)^*(g(\xi)).
\]
We have $\mathcal{P}(\xi)(0) = g(\xi)$ for all $\xi$. 
Replacing $g(\xi)$ by $\mathcal{P}(\xi)(1)$, we can assume that there exist open $\Gamma$-invariant neighborhoods $U \subset \nu$ of the zero-section and $V \subset M$ of $S$ such that the normal exponential maps for $g(\xi)$  induce $\Gamma$-equivariant diffeomorphisms $U \stackrel{\approx}{\longrightarrow} V$ which are independent of $\xi$ and coincide with the diffeomorphism induced by the normal exponential map for  $g_\mathrm{ref}$.
Furthermore, we can assume that the bundle  metric on $\nu$ induced by $g(\xi)$ is independent of $\xi\in\partial D^m$ and coincides with the one induced by $g_\mathrm{ref}$.

The oriented bundle $\nu$ with this bundle metric carries an induced $\SO(2)$-action as in Example~\ref{ex.circlebundle}.

\begin{lemma} \label{lem:free} Under the standard identification 
\[
   S^1 \cong \SO(2), \qquad \exp(i \theta) \mapsto \begin{pmatrix}\cos\theta & -\sin\theta \\ \sin\theta & \cos\theta \end{pmatrix},
\]
the $S^1$-action on $\nu$ induced by the $\Gamma$-action on $M$ coincides with this $\SO(2)$-action.
\end{lemma}

\begin{proof}
Let $H < S^1$ is the closed subgroup for which $S^1/H$ is the maximal orbit type of the $S^1$-action on $M$ induced by the given $\Gamma$-action.
Then $H$ acts trivially on an open and dense subset of $M$, and therefore on the whole of $M$ as $M$ is connected.
This  implies $H=\{1\}$, using effectiveness of the $S^1$-action on $M$.
Hence, $S^1$ acts freely on an open and dense subset of $M$.
Its intersection with $V$ is an open and dense subset of $V$.
This is only possible if the $S^1$-action on $\nu$ has weight $\pm 1$.
By the choice of the orientation of $\nu$, it has weight $+1$.
\end{proof}

Denote the fiberwise norm for the bundle metric induced by $g_\mathrm{ref}$ on $\nu$ by $| \bullet |$.
For $\rho > 0$, let 
\[
    B_\rho(S)  = \{ | \eta | <  \rho \}  \subset \nu 
    \text{ and } 
    \bar B_\rho(S)   = \{ | \eta| \leq \rho \}  \subset \nu
\]
be the open and closed $\rho$-disk bundles of $\nu$ with respect to $| \bullet |$, respectively.
Similarly, let
\[
     B_\rho := \{ | \zeta | < \rho\} \subset \R^2
     \text{ and } 
     \bar B_\rho := \{ | \zeta | \leq  \rho\} \subset \R^2
\]
be the open and closed $\rho$-disks centered at the origin, respectively.

Let $\rho_0 > 0$ such that $B_{\rho_0}( S)  \subset U$.
We identify  $B_{\rho_0} (S)$ with the open $\rho_0$-neighborhood of $S$ in $V\subset M$ by the normal exponential map for $g_\mathrm{ref}$. 
For each $\xi \in \partial D^m$ and each $\eta \in \nu$ with $|\eta|=1$, the radial line $[0,\rho_0) \ni t \mapsto t \eta \in B_{\rho_0}(S)$ is a unit speed geodesic with respect to $g(\xi)$.

\begin{notation} \label{bez}
For  $\xi \in \partial D^m$, the metric $g(\xi)$ induces 
\begin{enumerate}[\myicon]
\item  
a $\Gamma$-invariant Riemannian metric $g_S(\xi)$ on $S \subset M$,
\item 
a distribution $\mathscr{H}(\xi)$ on the total space $\nu$ corresponding to the normal connection on $\nu$ for $g(\xi)$, see Examples~\ref{ex.vectorbundle} and \ref{ex.normalbundle}.
\end{enumerate}
Note that $\mathscr{H}(\xi)$ is $\Gamma$-invariant and tangential to $\partial B_{\rho}(S)$ for all $\rho >0$.
Along $S$, the distribution $\mathscr{H}(\xi)$ coincides with $TS$.
In particular, $\mathscr{H}(\xi)$ is a horizontal distribution for the submersion $\nu \to S$.
\end{notation}

Equip  $\bar B_{\rho} \subset \R^2$ with the usual $\O(2)$-action and $\mathscr{R}^{\O(2)}(\bar B_\rho)$ with the trivial $\Gamma$-action. 
For any $\xi \in \partial D^m$, any $\rho > 0$ and any smooth $\Gamma$-invariant map 
\[
   \vartheta \colon S \to \mathscr{R}^{\O(2)}(\bar B_\rho),
\]
we can construct a new metric $\Theta(\xi)$ on $\bar B_\rho(S)$ as follows.
We make $\mathscr{H}(\xi)$ and the fibers of $\nu$ orthogonal.
On $\mathscr{H}(\xi)$, we take the metric making $d\pi|_{\mathscr{H}(\xi)_p}\colon\mathscr{H}(\xi)_p\to T_{\pi(p)}S$ an isometry at each point $p\in \bar B_{\rho}(S)$.
To define the metric along the fibers of $\nu$, we choose a linear isometry $T_p\nu_{\pi(p)}\cong\nu_{\pi(p)} \to \R^2$ and pull back the metric $\vartheta(\pi(p))$.
Since $\vartheta(\pi(p))$ is $\O(2)$-invariant, this definition does not depend on the choice of the isometry.
And since $\vartheta$ and $\mathscr{H}(\xi)$ are $\Gamma$-invariant, the metric $\Theta(\xi)$ is $\Gamma$-invariant.
In this way, we have constructed a $\Gamma$-equivariant Riemannian submersion
\begin{equation} \label{induced_sub}
  (\bar B_\rho(S)  , \Theta(\xi)) \to (S, g_S(\xi)) 
\end{equation}
whose fiber over $x \in S$ is isometric to $(\bar B_\rho, \vartheta(x))$.
This submersion has totally geodesic fibers, see Example~\ref{ex.vectorbundle}.

\begin{figure}[ht]
\begin{overpic}[scale=0.15, trim=-20mm 0 0 0]{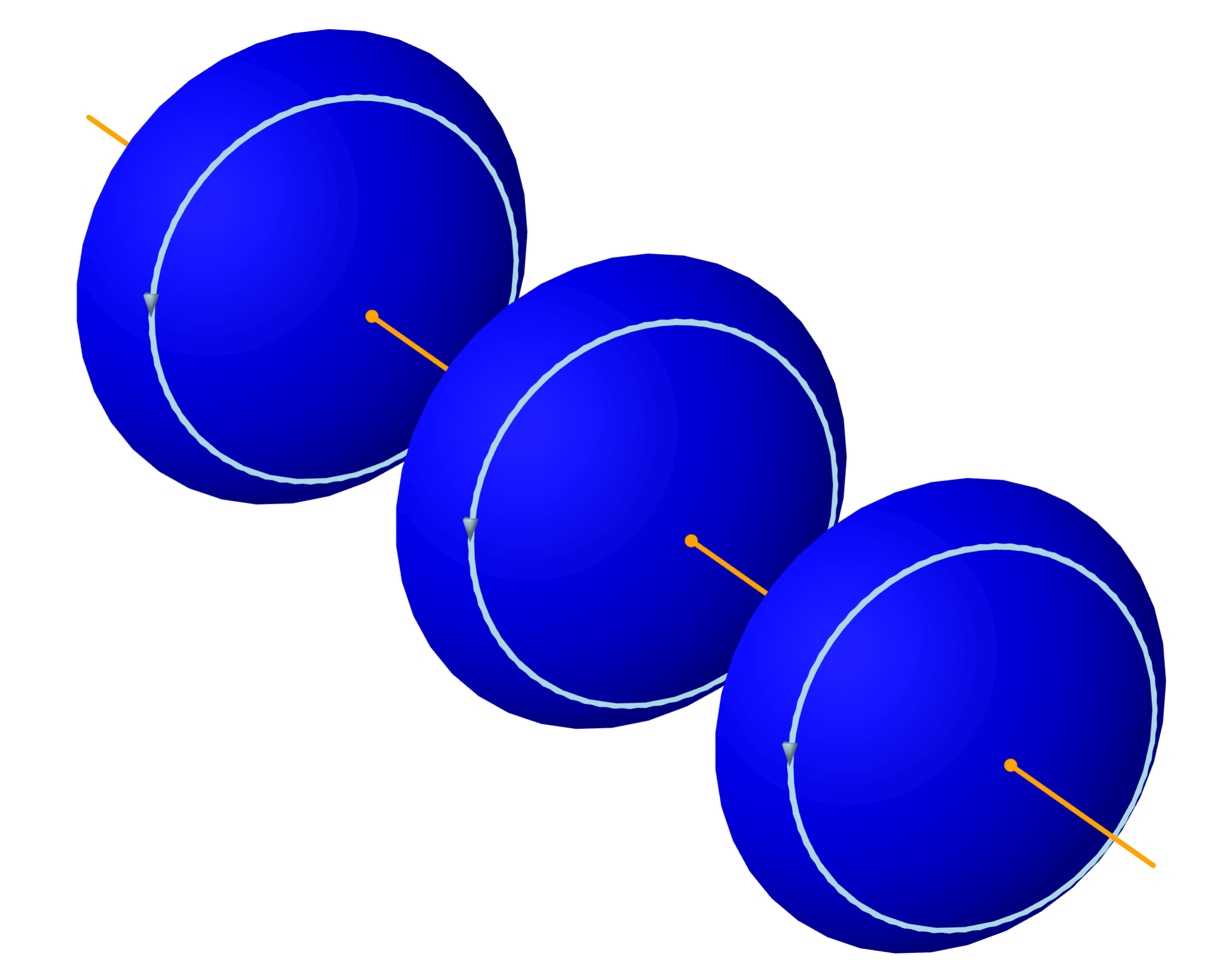}
\put(97,8){$\textcolor{gelb}{S}$}
\end{overpic}
\caption{Illustration of the metric $\Theta_{\sigma}(\xi)$ on $\bar B_{\rho}(S)$ for some fixed $\xi \in \partial D^m$.
The fixed-point component $S$ is indicated by the yellow axis.
The blue caps are fibers of the submersion~\eqref{subF}.
Three orbits of $\Gamma=S^1$ are shown in light blue.}

\label{fig:start}
\end{figure}

\begin{example} \label{ex.specialcase}
Let $\sigma > 0$ and let $0 < \rho < \sigma \pi$.
In polar coordinates $(r,\varphi)$ on $\R^2$, let $\vartheta_{ \sigma} \in \RR^{\O(2)}_{>0} (\bar B_{\rho})$ be given by  
\[
     \vartheta_{\sigma} = dr^2 + \sigma^2 \sin\left( \tfrac{r}{\sigma} \right)^2 d\varphi^2.
\]
In other words, $(\bar B_{\rho}, \vartheta_{\sigma})$ is isometric to the closed $\rho$-disk in the $2$-sphere of radius $\sigma$.
The boundary of this disk has mean curvature $\cot(\nicefrac{\rho}{\sigma})/\sigma$ which is positive if $\rho<\frac{\sigma \pi}{2}$.
For $\vartheta=\vartheta_{\sigma}$, the submersion~\eqref{induced_sub} specializes to a Riemannian submersion denoted
\begin{equation} \label{subF} 
  (\bar B_\rho(S)  , \Theta_{\sigma}(\xi)) \to (S, g_S(\xi)).
\end{equation} 
It has horizontal distribution $\mathscr{H}(\xi)$ and fibers isometric to  $(\bar B_\rho  , \vartheta_{\sigma})$. 
For $\Gamma = S^1$, the situation is depicted in Figure~\ref{fig:start}.
\end{example}

Next, we bring the metrics $g(\xi)$ into a standard form near $S$.
We work with the identification of $B_{\rho_0}(S) \subset \nu$ with the open $\rho_0$-neighborhood of  $S$ in $M$ with respect to $g_\mathrm{ref}$ (hence with respect to all $g(\xi)$), as explained before Notation \ref{bez}.

\begin{proposition} \label{deform2} 
There exists  $\sigma  > 0$ and  $0 < \rho <\min\{  \frac{\sigma \pi}{2}, \rho_0\}$ and a continuous map $\mathcal{P} \colon \partial D^m \to  C^{\infty}\big([0,1] ,  \mathscr{R}^{\Gamma}_{>0}(M)\big)$ such that
\begin{enumerate}[(i)] 
    \item \label{submersion1} $\mathcal{P}(\xi)(0) = g(\xi)$;
   \item  \label{submersion2}  $\mathcal{P}(\xi)(1)|_{\bar B_\rho(S)  } =\Theta_{\sigma}(\xi)$. 
 \end{enumerate} 
 \end{proposition} 

\begin{proof} 
For $\sigma > 0$,  the homothety $B_{\nicefrac{\pi}{2} } \to B_{\nicefrac{\sigma \pi}{2}}$, $x \mapsto \sigma x$, induces an $\O(2)$-equivariant isometry  $(B_{\nicefrac{\pi}{2} }  , \sigma^2 \vartheta_{1}) \cong (B_{\nicefrac{\sigma \pi}{2}}, \vartheta_{\sigma})$ and hence a $\Gamma$-equivariant isometry $(B_{\nicefrac{\pi}{2}}(S),  \Theta_{1}(\xi)_{\sigma^2}) \cong (B_{\nicefrac{\sigma \pi}{2}}(S), \Theta_{\sigma}(\xi))$.
Here the subscript $\sigma^2$ refers to the canonical variation~\eqref{can_var}.
Since
\[
      \lim_{\sigma \to 0} \scal_{\vartheta_{\sigma}} = \lim_{\sigma \to 0}  2 \sigma^{-2} = \infty 
\]
it follows from  \eqref{oneill} that there exists $\sigma  > 0$ such that for all $0 < \rho < \frac{\sigma \pi}{2}$ and $\xi \in \partial D^m$, the submersion metric $\Theta_{ \sigma}(\xi)$ on $B_\rho(S) $ has positive scalar curvature.
Let  $0 < \rho < \min\{  \frac{\sigma \pi}{2}, \rho_0\}$.

In order to apply the equivariant local flexibility Theorem~\ref{thm:FlexLem1}, we claim that for all  $\xi \in \partial D^m$ the $1$-jets of $g(\xi)$ and $\Theta_{\sigma}(\xi)$ coincide along $S$, i.e., 
\begin{equation} \label{1jet}
         j^1 g(\xi) |_{S}  = j^1  \Theta_{ \sigma}(\xi) |_{S}  . 
\end{equation}
Indeed, for each $\xi \in \partial D^m$, the metrics $g(\xi)$ and $\Theta_{ \sigma}(\xi)$ induce the same metric on $S$, and the normal exponential maps for $g(\xi)$ and $\Theta_{\sigma}(\xi)$ along $S$, which are equal to the identity on $B_{\rho}(S)$, coincide on $B_{\rho}(S)$.
Moreover, by the choice of the metric $\vartheta_{\sigma}$ in Example  \ref{ex.specialcase},  the bundle metrics on the abstract normal bundle $\nu=(TB_{\rho}(S)|_S)/TS \to S$ induced by $\Theta_{\sigma}(\xi)$ and $g(\xi)$ coincide.
Moreover, $S$ consists of fixed-point components of an $S^1$-action which is by isometries for both $g(\xi)$ and $\Theta_{\sigma}(\xi)$.
Hence, $S$ is totally geodesic for both metrics.
Finally, the normal connections on $\nu$ along $S$ for $g(\xi)$ and $\Theta_{\sigma}(\xi)$ coincide because they have the same horizontal distribution $\mathscr{H}(\xi)$.
Thus, Lemma~\ref{lem.jet1g} implies \eqref{1jet}.

Let $\mathrm{Sym}^2 (B_\rho(S) ) \to B_\rho(S) $ denote the bundle of symmetric $(0,2)$-tensors over $B_\rho(S) $.  
Consider the continuous map
\[
\bar g \colon \partial D^m \to C^{\infty}\big( [0,1],  \mathrm{Sym}^{2}(B_\rho(S) )\big),
\quad
\bar g(\xi)(t) := (1-t) g(\xi) + t \Theta_{\sigma}(\xi). 
\]
For sufficiently small $\rho$, each $\bar g(\xi)(t)$ is positive definite, i.e., a Riemannian metric.
Both $g(\xi)$ and $\Theta_{\sigma}(\xi)$ have positive scalar curvature and the $1$-jets coincide along $S$ by \eqref{1jet}.
Since scalar curvature depends affine linearly on the second derivatives of the metric, each $\bar g(\xi)(t)$ has positive scalar curvature along $S$ for all $\xi \in \partial D^m$ and all $t \in [0,1]$.
Hence, we get $\bar g(\xi)(t) \in  \RR^{\Gamma}_{>0}( B_{\rho}(S) )$ for all $\xi \in \partial D^m$ and $t \in [0,1]$, possibly after passing to a still smaller $\rho$.

We apply Theorem~\ref{thm:FlexLem1} to $E = \mathrm{Sym}^2 (B_\rho(S) )$ and the open partial differential relation of being a Riemannian metric of positive scalar curvature.
After possibly shrinking $\rho$ once more, we find a continuous map $\mathcal{P} \colon \partial D^m \to  C^{\infty}\big([0,1] ,  \mathscr{R}^{\Gamma}_{>0}(M)\big)$ with  properties~\ref{submersion1} and~\ref{submersion2}. 
\end{proof} 

By Proposition~\ref{deform2}, we can henceforth work under the following assumption on $g$: 

\begin{assumption} \label{ass} 
There exist $\sigma > 0$ and $0 < \rho <  \frac{\sigma \pi}{2}$ such that for all $\xi \in \partial D^m$, we have 
\begin{equation} \label{annahme} 
     g(\xi)|_{\bar B_{ \rho}(S) } = \Theta_{\sigma}(\xi) .
\end{equation}
\end{assumption}

We fix $\sigma$ and $\rho$ as in this assumption (if $m=0$, then we choose $\sigma > 0$ and $0 < \rho < \frac{\sigma \pi}{2}$ arbitrarily), and set
\[
   \breve{M} := M \setminus B_\rho(S)  . 
\]
This is a compact connected $\Gamma$-manifold with non-empty boundary $\partial \breve{M} = \partial \bar B_\rho(S) $. 
By Lemma~\ref{lem:free}, the induced $S^1$-action on $\partial \breve{M}= \partial \bar B_\rho(S)$ is free.

For $\xi \in \partial D^m$, the metric $g(\xi)$ restricts to a metric $\breve g(\xi)$ on $\breve M$ and induces a metric $\breve g^{\partial}(\xi)$ on $\partial \breve{M}$.
The metric $\breve g^{\partial}(\xi)$ coincides with the one induced on $\partial \bar B_\rho(S)$ by the submersion metric $\Theta_{\sigma}(\xi)$ from  \eqref{subF}.

By Theorem~\ref{mainA}, we can extend $\breve g$ to a continuous map $\breve G \colon D^m \to \mathscr{R}^{\Gamma}_{>0}(\breve M)$ satisfying $\breve G|_{\partial D^m}= \breve g$.
For $m = 0$, this amounts to the existence of some metric in $\mathscr{R}^{\Gamma}_{>0}(\breve M)$.
Without loss of generality, we can assume that for $\tfrac{1}{2} \leq |\xi| \leq 1$, the metric $\breve G(\xi)$ depends only on $\tfrac{\xi}{|\xi|}$, i.e., $\breve G(\xi)=\breve g(\nicefrac{\xi}{|\xi|})=g(\nicefrac{\xi}{|\xi|})|_{\breve M}$.

For $\xi \in D^m$ and $x\in\partial\breve{M}$, let $\ell_\xi(x)$ denote the length of the orbit $S^1\cdot x\subset \partial\breve{M}$ with respect to the metric $\breve G(\xi)$.
Then $\ell_\xi$ is a smooth $\Gamma$-invariant function on $\partial \breve{M}$ and the map $D^m\to C^{\infty,\Gamma}(\partial\breve{M})$, $\xi\mapsto \ell_\xi$ is continuous.

Assume temporarily $n=\dim(M)\ge3$ and let $\varphi_\xi$ be the unique solution to the boundary value problem
\begin{equation}
\begin{cases}
\Delta_{\breve G(\xi)} \varphi_\xi = 0 & \text{on } \breve M,\\
\varphi_\xi = \Big(\frac{2\pi\sigma\sin(\frac{\rho}{\sigma})}{\ell_\xi}\Big)^{\frac{n-2}{2}} & \text{on } \partial \breve M.
\end{cases}
\label{BVP}
\end{equation}
By elliptic regularity, $\varphi_\xi$ is a smooth $\Gamma$-invariant function on $\breve M$ and it depends continuously on $\xi$ with respect to the $C^\infty$-topology.
Since $\ell_\xi$ is positive, the strong maximum principle implies that $\varphi_\xi$ is positive on $\breve M$.
For the conformally related metric $\hat{G}(\xi)=\varphi_\xi^{\frac{4}{n-2}}\breve G(\xi)$, we have
\[
\scal_{\hat{G}(\xi)}\cdot \varphi_\xi^{\frac{n+2}{n-2}}
=
\big(4\tfrac{n-1}{n-2}\Delta_{\breve G(\xi)} + \scal_{\breve{G}(\xi)}\big)\,\varphi_\xi
=
\scal_{\breve{G}(\xi)}\cdot\varphi_\xi ,
\]
see e.g.\ \cite{LP}*{Eq.~(1.1)}.
Since $\scal_{\breve{G}(\xi)}$ is positive, so is $\scal_{\hat{G}(\xi)}$.
The length of the $S^1$-orbits in $\partial\breve{M}$ with respect to $\hat{G}(\xi)$ is equal to $2\pi\sigma\sin(\frac{\rho}{\sigma})$ for all $\xi\in D^m$.
If $|\xi|\in[\frac12,1]$, then $\varphi_\xi=1$ on $\partial\breve{M}$.
Thus, the solution $\varphi_\xi$ to the boundary value problem~\eqref{BVP} is given by $\varphi_\xi=1$ on $\breve{M}$ and we have $\hat{G}(\xi)=\breve{G}(\xi)$ in this case.

If $n=2$, then replace \eqref{BVP} by the boundary value problem
\begin{equation*}
\begin{cases}
\Delta_{\breve G(\xi)} \varphi_\xi = 0 & \text{on } \breve M,\\
\varphi_\xi = \log\Big(\frac{\ell_\xi}{2\pi\sigma\sin(\frac{\rho}{\sigma})}\Big) & \text{on } \partial \breve M,
\end{cases}
\end{equation*}
and consider metric $\hat{G}(\xi)=e^{-2\varphi_\xi}\cdot\breve G(\xi)$.
We then have $\scal_{\hat{G}(\xi)}=e^{2\varphi_\xi}\cdot\scal_{\breve{G}(\xi)}$ and we can argue as for $n\ge3$.

Replacing $\breve{G}$ by $\hat{G}$, we can make the following assumption in addition to Assumption~\ref{ass}:

\begin{assumption}\label{ass2}
For every $\xi\in D^m$, the lengths of the $S^1$-orbits in $\partial\breve{M}$ with respect to $\breve{G}(\xi)$ are equal to $2\pi\sigma\sin(\frac{\rho}{\sigma})$.
Moreover, for $\tfrac{1}{2} \leq |\xi| \leq 1$, the metric $\breve G(\xi)$ depends only on $\tfrac{\xi}{|\xi|}$ and is hence equal to $g(\nicefrac{\xi}{|\xi|})|_{\breve M}$.
\end{assumption}

The  map $\breve G$ induces a continuous map 
\begin{equation} \label{defbound} 
    \breve G^{\partial} \colon D^m \to \mathscr{R}^{\Gamma}(\partial \breve M) . 
\end{equation} 
Moreover, we obtain a $\Gamma$-invariant distribution $\mathscr{H}(\xi)\subset T\partial \breve{M}$ by taking the orthogonal complements (w.r.t.\ $\breve G(\xi)$) of the $S^1$-orbits.
The projection $\pi\colon \nu\to S$ restricts to a $\Gamma$-equivariant submersion 
\begin{equation}
\partial \breve{M}=\partial \bar B_\rho(S) \to S.
\label{bdysub}
\end{equation}
Since $\mathscr{H}(\xi)$ and $\breve G(\xi)$ are $\Gamma$-invariant, we get a uniquely defined $\Gamma$-invariant metric $G_S(\xi)$ on $S$ such that \eqref{bdysub} is a $\Gamma$-equivariant Riemannian submersion with horizontal distribution $\mathscr{H}(\xi)$.
For $\xi \in \partial D^m$ we have $G_S(\xi) = g_S(\xi)$, as introduced in Notation~\ref{bez}.

For $\xi \in D^m$, the distribution $\mathscr{H}(\xi)  \subset T\partial \breve{M} = T\partial \bar B_\rho(S)$ induces a metric linear connection on $\nu$, see Example~\ref{ex.circlebundle}. 
Since $\mathscr{H}(\xi)$ is $\Gamma$-invariant, this metric linear connection on $\nu$ is $\Gamma$-equivariant.
Together with the metric $G_S(\xi)$ on $S$ and the metric $\vartheta_{\sigma}$ on $\bar B_\rho$ we obtain a Riemannian submersion with totally geodesic fibers
\begin{equation} \label{nongeod}
\left( \bar B_\rho ,\vartheta_{\sigma} \right)  
\hookrightarrow 
\left(\bar B_\rho(S) ,  \Theta_{\sigma}(\xi)\right)  
\stackrel{\pi}{\twoheadrightarrow} 
(S, G_S(\xi)), 
\end{equation}
see Example~\ref{ex.vectorbundle}.
For $\xi \in \partial D^m$, this is the submersion \eqref{subF}.
By construction, the metrics $\Theta_{\sigma}(\xi)$ and $G_S(\xi)$ are $\Gamma$-invariant and depend continuously on $\xi$.
Furthermore, the submersion map $\pi$ is $\Gamma$-equivariant.

For $\tfrac{1}{2} \leq |\xi| \leq 1$, the metric $\Theta_{\sigma}(\xi)$ depends only on $\tfrac{\xi}{|\xi|}$, i.e., $\Theta_{\sigma}(\xi)=\Theta_{\sigma}(\xi / |\xi|) $, and for every $\xi \in D^m$, the metrics on  $\partial \breve M = \partial \bar B_\rho(S) $ induced by $\breve G(\xi)$ and by $\Theta_{\sigma}(\xi)$ coincide.
However, for $|\xi| < \frac{1}{2}$, the metric $\Theta_{\sigma}(\xi)$ may not be of positive scalar curvature.
For $\Gamma = S^1$, the situation near $\bar B_{\rho}(S)$ is depicted in Figure~\ref{fig.deformation1}.

\begin{figure}[ht]
\begin{overpic}[scale=0.2, trim=-20mm 0 0 0]{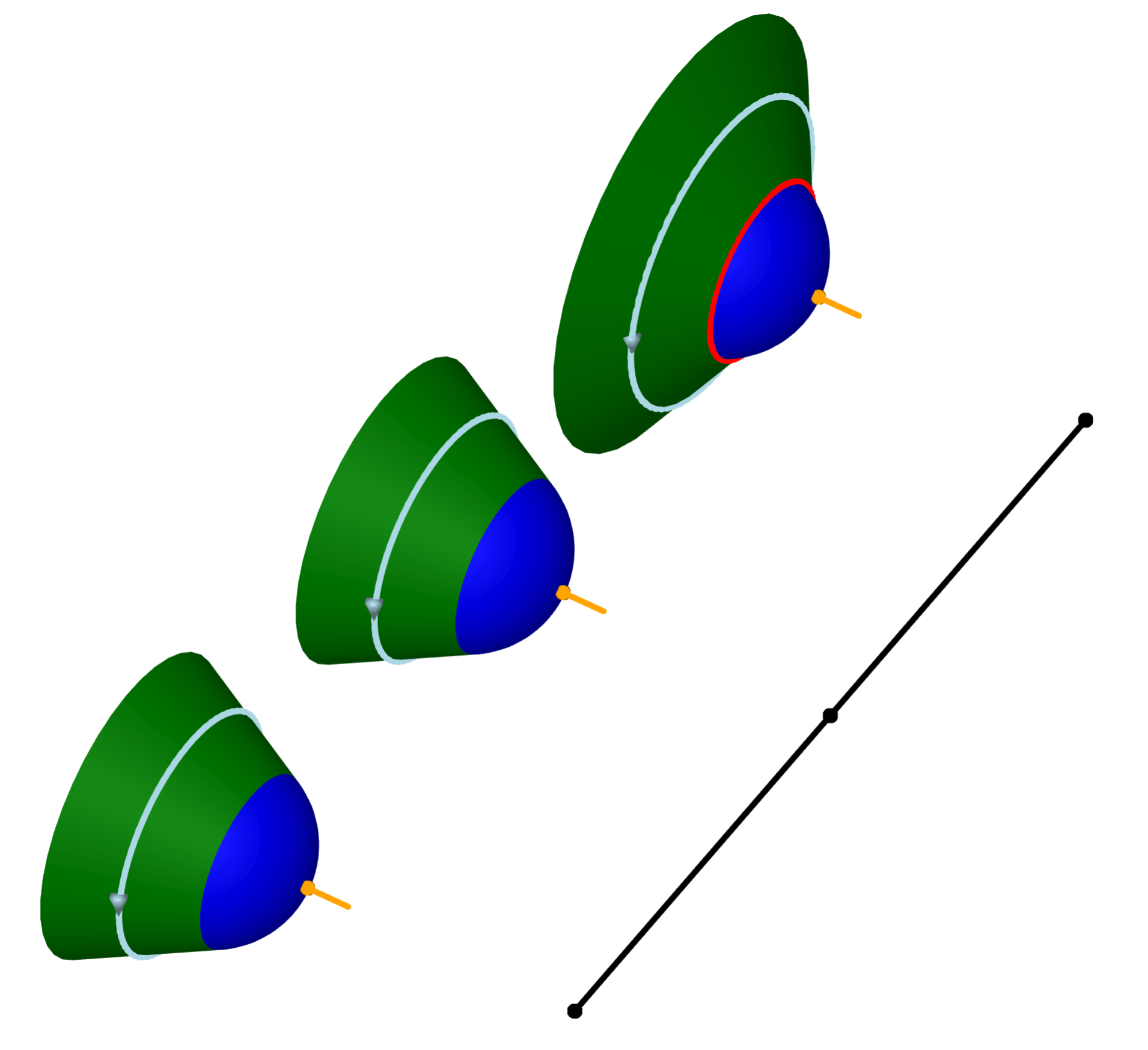}
\put(34,10){$\textcolor{gelb}{S}$}
\put(55,35){$\textcolor{gelb}{S}$}
\put(77,60){$\textcolor{gelb}{S}$}
\put(97,51){$|\xi|=0$}
\put(76,27){$|\xi|=\frac12$}
\put(55,2){$|\xi|=1$}
\end{overpic}
\caption{Illustration of the metrics $\Theta_{\sigma}(\xi) \cup \breve G(\xi)$ near $\bar B_{\rho}(S)$ for $|\xi| = 1$, $\tfrac{1}{2}$ and $0$.
The blue regions indicate $(\bar B_{\rho}(S), \Theta_{\sigma}(\xi))$ and the green regions indicate  $(\breve{M}, \breve {G}(\xi))$ near $\bar B_\rho(S)$.
Metrics in green are of positive scalar curvature. 
For $\xi = 0$,  the metric in blue may not be of positive scalar curvature and the sum of the mean curvatures of $\Theta_{\sigma}(\xi)$ and of $\breve G(\xi)$ with respect to the exterior normal may be negative (indicated in red).
One orbit of $\Gamma = S^1$ is shown in light blue.}

\label{fig.deformation1}
\end{figure}

\begin{proposition} \label{goodstuff}
For $\xi \in D^m$ and $s \geq 0$, consider the canonical variation (cf.\ \eqref{can_var})
\[
     \varphi(\xi , s)  := \Theta_\sigma(\xi)_{\exp(-s)} \in \mathscr{R}^{\Gamma}(\bar B_\rho(S) ). 
 \]
Then  the following assertions hold:
\begin{enumerate}[(i)]
\item \label{due} 
If $s \geq 0$ and $|\xi|\in[ \tfrac{1}{2} ,1]$, then the metric $ \varphi(\xi , s)$ has positive scalar curvature.
\item \label{tre} 
There exists $s_0 > 0$ such that for each $s \geq s_0$ and each $\xi \in D^m$, the metric $ \varphi(\xi , s)$ has positive scalar curvature.
\item \label{quattro} 
For every $\xi\in D^m$, the mean curvature of $\partial \bar B_\rho(S) $ with respect to  $\varphi(\xi , s)$ is given by $\exp(\nicefrac{s}{2})\cot(\nicefrac{\rho}{\sigma})/\sigma$.

In particular, it is monotonically increasing in $s$ and for every $H \in \R$ there exists $s_0 > 0$ such that it is bounded below by $H$ for every $s \geq s_0$ and $\xi \in D^m$.
\end{enumerate}
\end{proposition}

\begin{proof} 
For  $\xi \in D^m$ the scalar curvature of the metric $ \vartheta_{\sigma}$ on $\bar B_\rho $  is equal to $\frac{2}{\sigma^2}$.
Since $\exp(-s)$ is monotonically decreasing in $s$,  the O'Neill formula~\eqref{oneill} for Riemannian submersions with totally geodesic fibers implies that the scalar curvature of $\varphi(\xi , s)$ is monotonically increasing in $s$.

Furthermore, for $\tfrac{1}{2} \leq |\xi| \leq 1$, the metric
\[
\varphi(\xi , 0) = \Theta_{\sigma}(\xi) = g(\nicefrac{\xi}{|\xi|})|_{\bar B_{ \rho}(S) } 
\]
has positive scalar curvature.
This shows assertion~\ref{due}.

For $s \geq 0$, the metric $\exp(-s) \cdot \vartheta_{\sigma}$ on $\bar B_\rho$ has scalar curvature $\nicefrac{2 \exp(s)}{\sigma^2}$.
Since $\lim_{s \to \infty} \nicefrac{2 \exp(s)}{\sigma^2} = \infty$, the O'Neill formula~\eqref{oneill} implies \ref{tre}.

Next, the mean curvature of $\partial \bar B_\rho$ with respect to the metric $\vartheta_\sigma$ is given by $\cot(\nicefrac{\rho}{\sigma})/\sigma>0$.
Here we use the assumption $0 < \rho < \frac{\sigma \pi}{2}$.
Hence, the mean curvature with respect to $\exp(-s)\vartheta_\sigma$ is given by $\exp(\nicefrac{s}{2})\cot(\nicefrac{\rho}{\sigma})/\sigma$. 
Furthermore, the normal bundle of $\partial \bar B_\rho(S) $ with respect to $\Theta_\sigma(\xi)$ is vertical.
Thus,  assertion~\ref{quattro} follows from Proposition~\ref{meancurv1}.
\end{proof}

\begin{proof}[Conclusion of the proof of Theorem~\ref{mainB}]
Since the $S^1$-action on $\partial \breve M$ is free, there is an open $\Gamma$-invariant neighborhood $U$ of $\partial \breve M$ in $\breve M$ such that the $S^1$-action on $\bar U$ is free.
The set $U$ carries the structure of a smooth, non-compact $\Gamma$-manifold with boundary $\partial U = \partial \breve M$.
The metric $ \breve G(\xi)|_U$ induces a metric $\check h(\xi)$ on the orbit space $U / S^1$ such that we have a $\Gamma$-equivariant Riemannian submersion 
\begin{equation} \label{sub_inner}
       (U, \breve G(\xi)|_U)\to (U/S^1, \check h(\xi)) . 
\end{equation}

Since the $S^1$-action on $U$ restricts to an $S^1$-action on $\partial U$, the normal bundle of $\partial U \subset U$ with respect to $\breve G(\xi)|_U$ is horizontal. 
Hence Proposition~\ref{meancurv2} implies that for $\tau > 0$, the mean curvature of $\partial U \subset U$ with respect to the canonical variation $(\breve G(\xi)|_U)_{\tau}$ is independent of $\tau$.

Let $H_{\breve{G}(\xi)}$ be the mean curvature of $\partial U \subset U$ with respect to $\breve{G}(\xi)$ and set 
\[
    H :=  \max_{x\in\partial \breve M,\, \xi \in D^m} | H_{\breve{G}(\xi)}(x)| \in \R .
\]
Choose  $s_0 > 0$ large enough so that Proposition~\ref{goodstuff}~\ref{tre} and \ref{quattro} apply with this $H$.

Let $\chi \colon [0, 1] \to [0,1]$ be a smooth function equal to $1$ on $[0,\tfrac{1}{2}]$ and equal to $0$ near $1$.
Define the continuous map  $\Psi_{\bar B(S)} \colon D^m \to \RR^{\Gamma}_{>0}( \bar B_\rho(S) )$ by
\[
     \Psi_{\bar B(S)}(\xi) := \varphi(\xi, \chi(|\xi|) \cdot s_0) .
\]
Positivity of the scalar curvature of $\Psi_{\bar B(S)}(\xi)$ is implied by parts~\ref{due} and \ref{tre} of Proposition~\ref{goodstuff}.
 
Next, define $\Psi_{\breve M} \colon D^m \to \RR^{\Gamma}_{>0}(\breve M)$
\[
\Psi_{\breve M}(\xi) = \mathcal{P}(\xi)(\chi(|\xi| )\cdot s_0) ,
\]
where $\mathcal{P}$ is taken from Proposition~\ref{maindeform} with $\mathcal{P}(\xi)(0) = \breve G(\xi)$.
The families $\Psi_{\bar B(S)}$ and $\Psi_{\breve M}$ have the following properties:
\begin{enumerate}[\myicon]
\item
For every $\xi \in \partial D^m$, we have 
\begin{equation} \label{boundaryval}
     \Psi_{\bar B(S)}(\xi) = \Theta_{\sigma}(\xi) = g(\xi)|_{\bar B_\rho(S) } , \quad \Psi_{\breve M}(\xi) = \breve G(\xi) = g(\xi)|_{\breve M} . 
\end{equation}
\item
For every $\xi \in D^m$, the metrics on $\partial \bar B_\rho(S)  = \partial \breve M$ induced by $\Psi_{\bar B(S)}(\xi)$ and $\Psi_{\breve M}(\xi)$ coincide.
\item
For every $\xi \in D^m$, we get 
\begin{equation} \label{inequality}
H_{\Psi_{\bar B(S)}(\xi)} \geq - H_{\Psi_{\breve M}(\xi)}.
\end{equation}
\end{enumerate}
The first two properties are immediate. 
To show \eqref{inequality}, we distinguish the following two cases:
\begin{enumerate}[(I)]
\item 
Let $\tfrac{1}{2} \leq |\xi| \leq 1$. 
Then $\breve G(\xi) = g(\nicefrac{\xi}{|\xi|})|_{\breve M}$ by the construction of $\breve G$.
Hence, 
\[
H_{\Psi_{\bar B(S)}(\xi)} 
\geq  
H_{\Theta_{\sigma}(\xi)} 
=  
- H_{\breve G(\xi)} 
=  
- H_{\Psi_{\breve M}(\xi)}.
 \]
The inequality follows from Proposition~\ref{goodstuff}~\ref{quattro}, the first equality holds since $g(\nicefrac{\xi}{|\xi|})$ is a smooth metric on $M$, and the last equality follows from Proposition~\ref{meancurv2}.
\item 
Let $0 \leq |\xi| \leq \tfrac{1}{2}$.
Then, for each $y \in \partial \breve M = \partial \bar B_{\rho}(S)$, we get 
\begin{align*}
     H_{\Psi_{\bar B(S)}(\xi)}(y) 
     &=  
     H_{\varphi(\xi, s_0)}(y)  
     \geq  
     H
     = 
     \max_{x\in\partial \breve M} | H_{\breve{G}(\xi)}(x)|  \\
     &\geq 
     - H_{\breve G(\xi)}(y) 
     = 
     - H_{\Psi_{\breve M}(\xi)}(y).
\end{align*}
The first equality follows from  the choice of $\chi$, the first inequality follows from Proposition~\ref{goodstuff}~\ref{quattro}, and the last equality follows from Proposition~\ref{meancurv2}.
\end{enumerate}

For $\Gamma = S^1$, the situation on $\bar B_\rho(S) \cup U\subset M$ is depicted in Figure~\ref{fig.deformation2}.

\begin{figure}[ht]
\begin{overpic}[scale=0.18, trim=-20mm 0 0 0]{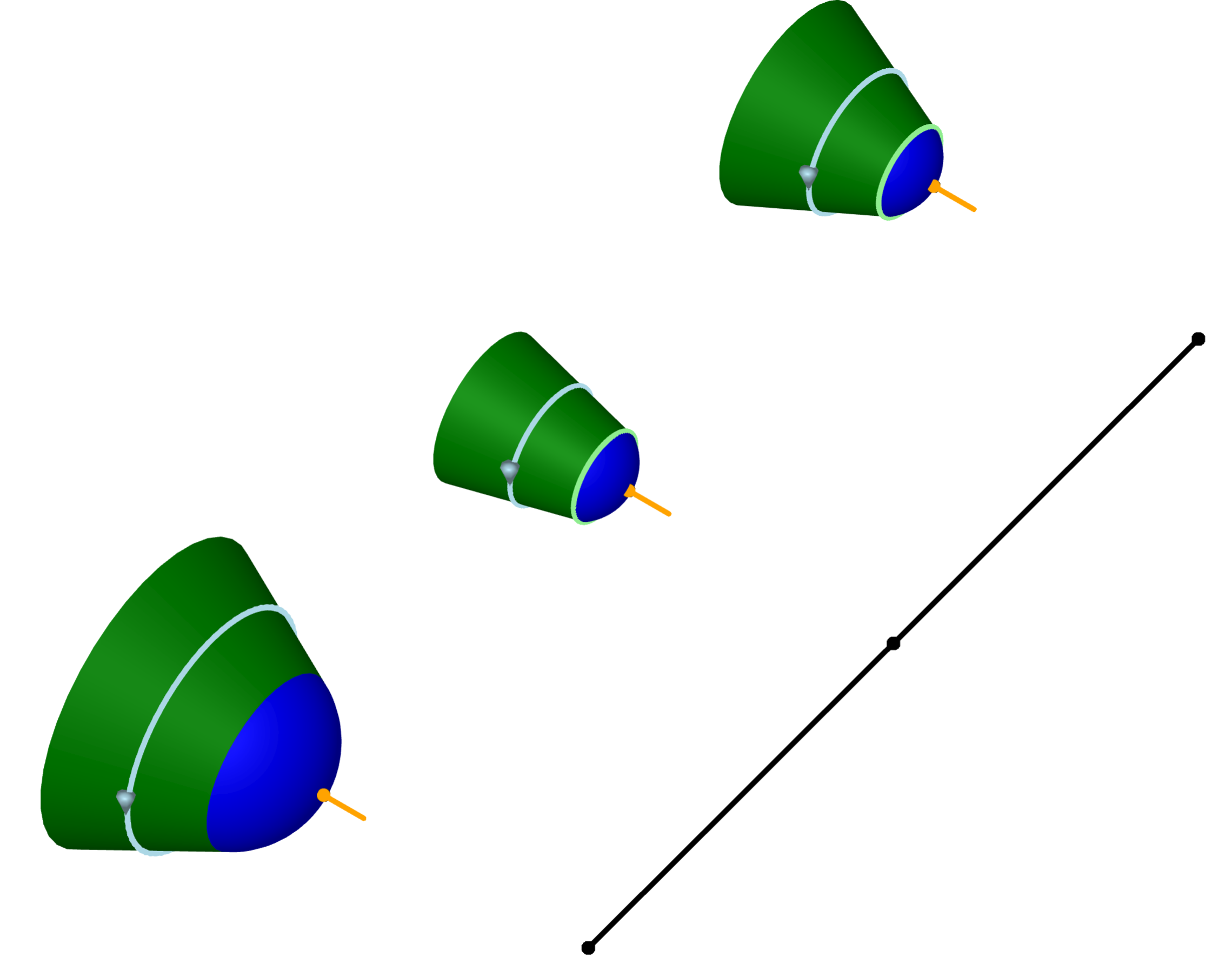}
\put(33,10){$\textcolor{gelb}{S}$}
\put(57,34){$\textcolor{gelb}{S}$}
\put(81,58){$\textcolor{gelb}{S}$}
\put(99,49){$|\xi|=0$}
\put(76,25){$|\xi|=\frac12$}
\put(53,1){$|\xi|=1$}
\end{overpic}
\caption{Illustration of the metrics $\Psi_{\bar B(S)}(\xi) \cup \Psi_{\breve M}(\xi)$ on $\bar B_\rho(S) \cup U$ for $|\xi| = 1$, $\tfrac{1}{2}$ and $0$.
Metrics in blue and green are of positive scalar curvature. 
At  $\partial \bar B_{\rho}(S) = \partial U$, the sum of the mean curvatures is non-negative for all $\xi \in D^m$ (indicated in light green).
One orbit of $\Gamma = S^1$ is shown in light blue.}
\label{fig.deformation2}
\end{figure}

We now apply the results of Section~\ref{sec.smoothing} to $\hat M := M$,  $\Sigma := \partial \bar B_{\rho}(S) \subset M$, $M_1 := \bar B_{\rho}(S) $ and  $M_2 := \breve M$.
By the previous argument, the following properties are satisfied
\begin{enumerate}[\myicon]
  \item For $\xi \in D^m$, we have $\Psi_{\bar B(S)}(\xi) \sqcup  \Psi_{\breve M}(\xi) \in \mathscr{R}^{\Gamma,\partial \bar B_{\rho}(S)}_{>0} (M)$.
  \item For $\xi \in \partial D^m$, we have $\Psi_{\bar B(S)}(\xi) \sqcup  \Psi_{\breve M}(\xi) \in \mathscr{R}^{\Gamma}_{>0} (M)$.
\end{enumerate}
Proposition~\ref{desingularise} yields a continuous map
\[ 
    \mathcal{P} \colon D^m \times [0,1] \to \mathscr{R}^{\Gamma,\partial \bar B_{\rho}(S)}_{>0}(M)
\]
such that
\begin{enumerate}[\myicon]
  \item For $\xi \in D^m$, we have $\mathcal{P}(\xi,0) = \Psi_{\bar B(S)}(\xi)  \sqcup  \Psi_{\breve M}(\xi)$,
  \item For $(\xi,s) \in  \big( \partial D^m \times [0,1]\big) \cup \big( D^m \times \{1\} \big)$, we have $\mathcal{P}(\xi,s) \in \mathscr{R}^{\Gamma}_{>0} (M)$.
\end{enumerate}
Now define the continuous map $G \colon D^m \to \mathscr{R}_{>0}^{\Gamma}(M)$ by
\[
G(\xi) = 
\begin{cases} 
\mathcal{P} ( 2 \xi, 1),  & \textrm{ for } 0 \leq |\xi| \leq \frac{1}{2} , \\
\mathcal{P}\left( \frac{\xi}{|\xi|} , 2 - 2 |\xi| \right), & \textrm{ for } \frac{1}{2} \leq |\xi| \leq 1 .
\end{cases}
\]
Using \eqref{boundaryval}, we have $G(\xi) = g(\xi)$ for $\xi \in \partial D^m$. 
This completes  the proof of Theorem~\ref{mainB}.
\end{proof}

\appendix
 
\section{Equivariant local flexibility for closed smooth submanifolds}
\label{sec:EqFlexLemma}

Theorem~1.2 in \cite{BaerHanke1} is a rather general extension result for locally defined homotopies of solutions to open partial differential relations.
In the literature, it is known as the ``flexibility lemma'' (see \cite{GromovPartial}*{p.~111}) and as the ``cut-off homotopy lemma'' (see \cite{Gromov2018}*{Section~11.1}).
In the present article, we need a $\Gamma$-invariant version which we provide in this appendix.
Function spaces of smooth maps will be equipped with  the weak $C^{\infty}$-topology, see Chapter~2 in \cite{Hirsch94}.

 \begin{setting}
 \label{set:setting}

 We denote by
 \begin{enumerate}[\myicon]
 \item 
 $\Gamma$ a compact Lie group and $M$ a smooth $\Gamma$-manifold;
 \item
 $S \subset M$ a compact smooth $\Gamma$-invariant submanifold without boundary;
 \item
 $U$ an open $\Gamma$-invariant neighborhood of $S$ in $M$;
 \item
 $E \to M$ a smooth $\Gamma$-vector bundle;
 \item
 $k\ge 0$ an integer;
  \item 
  $K$ a compact Hausdorff space; 
  \item 
 $\RR\subset  J^k E$ an open subset, where $J^k E \to E \to M$ is the $k$-th jet bundle; 
  \item
 $f_0 \colon  K \to C^{\infty}(M,E)$ a continuous family of smooth $\Gamma$-invariant sections on $M$ such that each $f_0(\xi)$ solves  $\RR$ over $M$;
 \item
 $F \colon   K \to C^{\infty}([0,1],  C^{\infty} (U,E))$ a  continuous family of smooth paths of $\Gamma$-equivariant sections on $U$ with the property that each $F(\xi)(t)$ solves $\RR$ over $U$ and  $F(\xi)(0) = f_0(\xi)|_U$. 
 \end{enumerate}
 \end{setting}
 
\begin{theorem} \label{thm:FlexLem1}
Assume that in Setting~\ref{set:setting}, we have 
\[
       j^{k-1}F(\xi)(t)|_{S}  = j^{k-1}f_0(\xi)|_{S} \, \textrm{  for all } \, \xi \in K \text{ and } t \in [0,1] . 
\]
(For $k=0$, this assumption is empty.)
Then there exists a $\Gamma$-invariant open neighborhood $U_0\subset U$ of $S$ and a continuous map 
 \[
     f\colon  K \to C^{\infty} ( [0,1], C^{\infty}(M,E))
 \]
 such that for all $\xi \in K$ and $t \in [0,1]$ we have:
 \begin{enumerate}[\myicon] 
 \item 
 $f(\xi)(t)$ is a $\Gamma$-equivariant section of $E$, solving $\RR$;
 \item
 $f(\xi)(0)=f_0(\xi)$;
 \item
 $f(\xi)(t)|_{U_0}=F(\xi)(t)|_{U_0}$;
 \item
 $f(\xi)(t)|_{M \setminus U}=f_0(\xi) |_{M \setminus U}$.
\end{enumerate} 
Moreover, if $F(\xi)$ is constant in $t$, then  $f(\xi)$ is constant in $t$. 
\end{theorem}

\begin{proof}
Theorem~1.2 in \cite{BaerHanke1} ensures the existence of $U_0$ and $f$, except that no statement about $\Gamma$-invariance is made.
Replacing $U_0$ by a smaller $\Gamma$-invariant neighborhood of $S$ in $U$ if necessary, we can assume that $U_0$ is $\Gamma$-invariant.

The homotopy $f$ is constructed in \cite{BaerHanke1} by the prescription
\begin{equation*} 
  f(\xi)(t)(p) 
  := 
\begin{cases}
F(\xi)(t\tau(p))(p), & \text{if }p \in U_0' ,\\
f_0(\xi)(p), & \text{if }p \in M\setminus U_0',
\end{cases}
\end{equation*}
where $U_0'$ is a suitable open subset with $U_0\subset U_0'\subset U$.
We can choose $U_0'$ to be $\Gamma$-invariant.
Moreover, since $S$ is a compact smooth submanifold,\footnote{In \cite{BaerHanke1}, $S$ is only assumed to be a closed subset of $M$. In that case, $\tau$ takes a more complicated form.} the function $\tau\colon M\to[0,1]$ can be chosen to be of the form $\tau(p)=\tilde{\tau}(\mathrm{dist}(p,S))$ for a suitable smooth function $\tilde{\tau}\colon\R\to[0,1]$ as in Lemma~2.8 of \cite{BaerHanke1}.
Note that $\tilde{\tau}$ is equal to $1$ near $0$ so that $\tau$ is smooth near $S$.
Here the distance is measured with respect to an auxiliary $\Gamma$-invariant Riemannian metric on $M$.
Then $\tau$ is $\Gamma$-invariant.

Therefore each $f(\xi)(t)$ is $\Gamma$-equivariant because
\begin{align*}
\gamma \big(f(\xi)(t)(\gamma^{-1}p)\big)
&=
\gamma (F(\xi)(t\tau(\gamma^{-1}p))(\gamma^{-1}p)) 
=
\gamma (F(\xi)(t\tau(p))(\gamma^{-1}p)) \\
&=
F(\xi)(t\tau(p))(p)
=
f(\xi)(t)(p) .
\qedhere
\end{align*}
\end{proof}

Simple examples show that the assumption $j^{k-1}F(\xi)(t)=j^{k-1}f_0(\xi)$ along $S$ cannot be dropped (see \cite{BaerHanke1}*{Remark~3.2}).

\section{\texorpdfstring{The $1$-jet of a metric along a submanifold}{The 1-jet of a metric along a submanifold}}
\label{sec.jet1}

Let $M$ be a smooth manifold and $S \subset M$ a smooth submanifold.
We provide a sufficient criterion which ensures that the $1$-jets of two metrics on $M$ agree along $S$.

\begin{lemma}\label{lem.jet1g}
Let $g$ and $\tilde{g}$ be two Riemannian metrics on $M$.
Assume that 
\begin{itemize}[\myicon] 
\item 
$g$ and $\tilde{g}$ induce the same metric on $S$.
\item 
The normal exponential maps $\exp^{\perp}_g$ and $\exp^{\perp}_{\tilde{g}}$ along $S$ coincide on a neighborhood of the zero-section in the normal bundle $(TM|_S)/TS\to S$.
\item 
The bundle metrics on $(TM|_S)/TS \to S$ induced by $g$ and $\tilde g$ coincide.
\item 
The second fundamental forms of $S$ in $M$ with respect to $g$ and $\tilde{g}$ coincide.
\item
The connections on $\nu$ induced by $g$ and $\tilde{g}$ coincide.
\end{itemize}
Then $j^1 g = j^1 \tilde{g}$ along $S$.
\end{lemma}

\begin{proof}
Any Riemannian metric on $M$ induces an identification of the abstract normal bundle $\nu=(TM|_S)/TS\to S$ with a subbundle of $TM|_S$, namely with the orthogonal complement of $TS$.
Since the normal exponential maps are assumed to coincide, $g$ and $\tilde{g}$ induce the same identification.
Thus the first three conditions imply that $g$ and $\tilde{g}$ coincide on $TM|_S$.
In other words, the $0$-jets of $g$ and $\tilde{g}$ agree along $S$.

To check the condition $j^1 g = j^1 \tilde{g}$ at a point $p\in S$, we choose Riemannian normal coordinates $(x^1, \ldots, x^m)$ of $S$ around $p$ where $m=\dim(S)$.
Let $e_{m+1},\dots,e_n$ be a local orthonormal frame of the normal bundle of $S$ in $M$ near $p$.
We extend the coordinates $(x^1, \ldots, x^m)$ to a local coordinate system $(x^1, \ldots, x^n)$ of $M$ near $p$ by $x^{k}\big(\exp_{(x^1,\ldots,x^m)}(\sum_{i=m+1}^n t^i e_i)\big) = t^k$ for $k=m+1,\dots,n$.
By the first and second assumption in the lemma, the coordinates are the same for both metrics $g$ and $\tilde{g}$.

In what follows, denote indices in $\{1,\dots,m\}$ by Roman letters and the ones in $\{m+1,\dots,n\}$ by Greek letters.
We investigate the Christoffel symbols of the Levi-Civita connection of $g$ in these coordinates.
The $\Gamma_{ij}^k$ are the Christoffel symbols of the metric induced on $S$ but they are also Christoffel symbols of $M$.
Since we started with normal coordinates on $S$, we have $\Gamma_{ij}^k(p)=0$ for $i,j,k\in\{1,\dots,m\}$.

For $(v_{m+1},\dots,v_n)\in\R^{n-m}$, the curve $c(t) = \exp_p(\sum_\alpha t v_\alpha e_\alpha)$ is a geodesic in $M$. 
Therefore, $0=\ddot{c}^\bullet + \sum_{\alpha\beta} \Gamma_{\alpha\beta}^\bullet \dot{c}^\alpha \dot{c}^\beta = 0 + \sum_{\alpha\beta} \Gamma_{\alpha\beta}^\bullet v_\alpha v_\beta$ where $\bullet\in\{1,\dots,n\}$.
Thus $\Gamma_{\alpha\beta}^k(p)=\Gamma_{\alpha\beta}^\gamma(p)=0$.

Moreover, $\Gamma_{ij}^\alpha(p) = g(\II_p(\partial_i,\partial_j),\partial_\alpha)$ where $\II$ denotes the second fundamental form.
Finally, $\Gamma_{i\alpha}^j(p) = - \Gamma_{ij}^\alpha(p) = - g(\II_p(\partial_i,\partial_j),\partial_\alpha)$.

Finally, the $\Gamma_{k\alpha}^\beta(p) = \Gamma_{\alpha k}^\beta(p)$ are the Christoffel symbols of the normal connection on~$\nu$.

The same reasoning applies to the Christoffel symbols for the metric $\tilde{g}$ with respect to the same coordinates $(x^1,\dots,x^n)$.
Thus, the Christoffel symbols of $g$ and $\tilde{g}$ coincide at $p$ and hence so do the first derivatives of the two metrics.
This proves $j^1g(p)=j^1\tilde{g}(p)$.
\end{proof}

\section{An auxiliary lemma}
\label{sec.aux}

\begin{lemma} \label{lem.tub.nbhd}
Let $\Gamma$ be a compact Lie group and $M$ a connected smooth $\Gamma$-manifold without boundary.
Let $K$ be a compact Hausdorff space and $\bar G \colon K \to \mathscr{R}^{\Gamma}(M)$ a family of $\Gamma$-invariant (not necessarily complete) Riemannian metrics on $M$, continuous with respect to the weak $C^\infty$-topology.
Let $\mathscr{W} \subset M$ be an embedded $\Gamma$-invariant submanifold of codimension $\ge1$.
Let $r_\xi \colon M \to [0,\infty)$ be the distance function to $\mathscr{W}$ with respect to the metric $\bar G(\xi)$.

Then there exists a $\Gamma$-invariant open neighborhood $V$ of $\mathscr{W}$ in $M$ such that $r_\xi^2$ is smooth on $V$ for every $\xi \in K$.
\end{lemma}

\begin{proof}
For $\xi \in K$, denote the geometric normal bundle of $\mathscr{W}$ in $M$ by $\pi\colon\nu_\xi \to \mathscr{W}$.
For $r>0$ let $B_{r}(\nu_\xi) := \{ v \in \nu_\xi \mid |v| < r\} \to \mathscr{W}$ be the open disk bundle of radius $r$.
Let $d_\xi$ be the distance function on $M$ with respect to $\bar G(\xi)$ and let $B_{\xi,r}(x) = \{y\in M \mid d_\xi(x,y)<r\}$ denote the corresponding open metric ball of radius $r$.

\medskip\noindent\textit{Step~1: local smoothness.}
For each $x\in\mathscr{W}$ and $\xi\in K$, by the tubular neighborhood theorem there exists $r_{x,\xi}>0$ such that the normal exponential map 
\[
\exp^{\perp}_{\bar G(\xi)} \colon B_{r_{x,\xi}}(\nu_\xi)\big|_{B_{\xi,r_{x,\xi}}(x)\cap\mathscr{W}} \to M
\] 
is a diffeomorphism onto its image and the normal geodesics are the unique (up to reparametrization) shortest geodesics from points in that image to $\mathscr{W}$.
Hence
\begin{equation}
r_\xi\!\left(\exp^{\perp}_{\bar G(\xi)}(v)\right) = |v| \quad \textrm{ for all } v \in B_{r_{x,\xi}}(\nu_\xi)\big|_{B_{\xi,r_{x,\xi}}(x)\cap\mathscr{W}} .
\label{eq.tub.nbhd.dist}
\end{equation}
Now let $p = \exp^{\perp}_{\bar G(\xi)}(v)$ with $v \in B_{\frac{r_{x,\xi}}{3}}(\nu_\xi)\big|_{B_{\xi,\frac{r_{x,\xi}}{3}}(x)\cap\mathscr{W}}$.
We show that there is no point $y\in\mathscr{W}$ closer to $p$ than $\pi(v)$.
Indeed, if $d_\xi(p,y) \le d_\xi(p,\pi(v)) < \frac{r_{x,\xi}}{3}$, then the triangle inequality gives
\[
d_\xi(y,x) \le d_\xi(y,p) + d_\xi(p,\pi(v)) + d_\xi(\pi(v),x) < r_{x,\xi},
\]
so $y \in B_{\xi,r_{x,\xi}}(x)\cap\mathscr{W}$.
Since $\exp^{\perp}_{\bar G(\xi)}$ maps $B_{r_{x,\xi}}(\nu_\xi)\big|_{B_{\xi,r_{x,\xi}}(x)\cap\mathscr{W}}$ diffeomorphically onto its image, we conclude $y=\pi(v)$.
By \eqref{eq.tub.nbhd.dist} it follows that $r_\xi(p) = d_\xi(p,\pi(v)) = |v|$, hence 
\[
r_\xi^2 = |\cdot|^2 \circ (\exp^{\perp}_{\bar G(\xi)})^{-1}
\] 
is smooth on the image $U_{x,\xi} := \exp^{\perp}_{\bar G(\xi)}\!\big(B_{\frac{r_{x,\xi}}{3}}(\nu_\xi)\big|_{B_{\xi,\frac{r_{x,\xi}}{3}}(x)\cap\mathscr{W}}\big)$.

\medskip\noindent\textit{Step~2: uniformity in $\xi$.}
By smooth dependence of the exponential map on the metric and compactness of $K$, the radius $r_{x,\xi}$ can be chosen to depend continuously on $\xi$.
Hence $\rho_x := \min_{\xi \in K} r_{x,\xi}$ is positive.
Fix an auxiliary Riemannian metric $g_0$ on $M$.
Since each $U_{x,\xi}$ is an open neighborhood of $x$ varying continuously with $\xi \in K$, compactness of $K$ yields $\delta_x > 0$ such that the $g_0$-ball $B^{g_0}_{\delta_x}(x)$ is contained in $U_{x,\xi}$ for every $\xi \in K$.
On this ball, $r_\xi^2$ is smooth for every $\xi \in K$.

\medskip\noindent\textit{Step~3: the global neighborhood.}
Set $V_0 := \bigcup_{x \in \mathscr{W}} B^{g_0}_{\delta_x}(x)$, an open neighborhood of $\mathscr{W}$ on which $r_\xi^2$ is smooth for every $\xi \in K$.
Since $\Gamma$ is compact and acts by isometries of each $\bar G(\xi)$ (hence preserves $r_\xi^2$), the $\Gamma$-saturation $V := \Gamma \cdot V_0$ is a $\Gamma$-invariant open neighborhood of $\mathscr{W}$ on which $r_\xi^2$ is smooth for every $\xi \in K$.
\end{proof}

\end{document}